\documentclass[12pt]{amsart}

\usepackage{amsmath,amsthm,amsfonts,amssymb,latexsym,verbatim,bbm}
\input xy
\xyoption{all}
\usepackage{hyperref,multirow}

\theoremstyle{plain}
\newtheorem{theorem}{Theorem}[section]
\newtheorem{thm}[theorem]{Theorem}
\newtheorem{lemma}[theorem]{Lemma}

\newtheorem{prop}[theorem]{Proposition}

\newtheorem{defn}[theorem]{Definition}

\newtheorem{cor}[theorem]{Corollary}

\newcommand{\iso}{\cong}
\newcommand{\arr}{\rightarrow}
\newcommand{\Id}{\mathbbm{1}}
\newcommand{\incl}{\hookrightarrow}
\DeclareMathOperator{\Hom}{Hom}
\DeclareMathOperator{\Spec}{Spec}
\DeclareMathOperator{\End}{End}

\DeclareMathOperator{\total}{total}

\newcommand{\R}{\mathbb{R}}
\newcommand{\C}{\mathbb{C}}
\newcommand{\Z}{\mathbb{Z}}

\newcommand{\mbA}{\mathbb{A}}
\newcommand{\eps}{\epsilon}

\newcommand{\mcB}{\mathcal{B}}

\newcommand{\mcM}{\mathcal{M}}

\newcommand{\mcP}{\mathcal{P}}
\newcommand{\mcN}{\mathcal{N}}

\newcommand{\mfg}{\mathfrak{g}}
\newcommand{\mfp}{\mathfrak{p}}
\newcommand{\mfh}{\mathfrak{h}}
\newcommand{\mfk}{\mathfrak{k}}
\newcommand{\mfm}{\mathfrak{m}}
\newcommand{\mfM}{\mathfrak{M}}
\newcommand{\mfn}{\mathfrak{n}}
\newcommand{\mfq}{\mathfrak{q}}
\newcommand{\mfr}{\mathfrak{r}}

\DeclareMathOperator{\tr}{tr}

\DeclareMathOperator{\re}{Re}

\DeclareMathOperator{\mfsl}{\mathfrak{sl}}
\DeclareMathOperator{\mfso}{\mathfrak{so}}

\DeclareMathOperator{\mfu}{\mathfrak{u}}
\DeclareMathOperator{\mfb}{\mathfrak{b}}

\DeclareMathOperator{\mfz}{\mathfrak{z}}

\DeclareMathOperator{\ad}{ad}

\DeclareMathOperator{\Switch}{Switch}

\newcommand{\deltabar}{\bar{\partial}}
\newcommand{\boxbar}{\overline{\square}}
\DeclareMathOperator{\Coinv}{\overline{R}}

\newcommand\TabTop{\rule{0pt}{2.6ex}}
\newcommand\TabBot{\rule[-1.2ex]{0pt}{0pt}}

\begin{document}

\title{Twisted strong Macdonald theorems and adjoint orbits}
\author{William Slofstra}
\email{slofstra@math.berkeley.edu}
\begin{abstract}
    The strong Macdonald theorems state that, for $L$ reductive and $s$ an odd
    variable, the cohomology algebras $H^*(L[z]/z^N)$ and $H^*(L[z,s])$ are
    freely generated, and describe the cohomological, $s$-, and $z$-degrees of
    the generators. The resulting identity for the $z$-weighted Euler
    characteristic is equivalent to Macdonald's constant term identity for a
    finite root system. We calculate $H^*(\mfp / z^N \mfp)$ and $H^*(\mfp[s])$
    for $\mfp$ a standard parahoric in a twisted loop algebra, giving strong
    Macdonald theorems that take into account both a parabolic component and a
    possible diagram automorphism twist. In particular we show that $H^*(\mfp /
    z^N \mfp)$ contains a parabolic subalgebra of the coinvariant algebra of
    the fixed-point subgroup of the Weyl group of $L$, and thus is no longer
    free. We also prove a strong Macdonald theorem for $H^*(\mfb; S^* \mfn^*)$
    and $H^*(\mfb / z^N \mfn)$ when $\mfb$ and $\mfn$ are Iwahori and nilpotent
    subalgebras respectively of a twisted loop algebra. For each strong
    Macdonald theorem proved, taking $z$-weighted Euler characteristics gives
    an identity equivalent to Macdonald's constant term identity for the
    corresponding affine root system. As part of the proof, we study the
    regular adjoint orbits for the adjoint action of the twisted arc group
    associated to $L$, proving an analogue of the Kostant slice theorem.
\end{abstract}
\maketitle

\section{Introduction}
Macdonald's constant term identity states that if $\Delta$ is a reduced root
system then
\begin{equation}\label{E:constantterm}
    [e^0] \prod_{\alpha \in \Delta^+ } \prod_{i=1}^N (1-q^{i-1} e^{-\alpha}) 
        (1-q^i e^{\alpha}) = \prod_{i=1}^l \binom{N(m_i+1)}{N}_q,
\end{equation}
where $m_1,\ldots,m_l$ is the list of exponents of $L$ and $\binom{a}{b}_q$ is
the $q$-binomial coefficient. Macdonald presented the identity as a conjecture
in \cite{Ma82}, and observed that it constitutes the untwisted case of a
constant term identity for affine root systems. Further extensions (including a
$(q,t)$-version) and proofs for individual affine root systems followed
(see for instance \cite{ZB85} \cite{Hab86} \cite{Ze87} \cite{St88} \cite{Ze88}
\cite{Ma88} \cite{Gu90} \cite{GG91} \cite{Kad94}) until Cherednik gave a uniform
proof of the most general version using double affine Hecke algebras
\cite{Ch95}.

Suppose $\Delta$ is the root system of a semisimple Lie algebra $L$ with
exponents $m_1,\ldots,m_l$. Prior to Cherednik's proof, Hanlon observed in
\cite{Ha86} that the constant term identity would follow from a stronger
conjecture:
\begin{equation}\label{E:strongmacdonald1}
    \parbox{5.7in}{
    The cohomology $H^*\left(L[z]/z^{N}\right)$ is a free super-commutative algebra with
    $N$ generators of cohomological degree $2m_i+1$ for each $i=1,\ldots,l$,
    of which, for fixed $i$, one has $z$-degree $0$ and the others have
    $z$-degree $Nm_i+j$ for $j=1,\ldots,N-1$.}
\end{equation}
Hanlon termed this the \emph{strong Macdonald conjecture}, and gave a proof for
$L = \mfsl_n$. Feigin observed in \cite{Fe91} that the identity of
(\ref{E:constantterm}) and the theorem of (\ref{E:strongmacdonald1}) follow
from:
\begin{equation}\label{E:strongmacdonald2}
    \parbox{5.7in}{
        The (restricted) cohomology $H^*(L[z,s])$ for $s$ an odd variable is a
        free super-commutative algebra with generators of tensor degree $2m_i+1$
        and $2m_i+2$, $z$-degree $n$, for $i=1,\ldots,l$ and $n \geq 0$, where
        tensor degree refers to combined cohomological and $s$-degree. 
    }
\end{equation}
This version of the strong Macdonald conjecture corresponds to the $(q,t)$
version of the Macdonald constant term conjecture.  However, an error was
discovered in Feigin's proof of (\ref{E:strongmacdonald2}).  A complete proof
of (\ref{E:strongmacdonald1}) and (\ref{E:strongmacdonald2}) was given by
Fishel, Grojnowski, and Teleman \cite{FGT08}, using an explicit description of
the relative cocycles combined with Feigin's idea (a spectral sequence
argument) to prove (\ref{E:strongmacdonald1}) from (\ref{E:strongmacdonald2}).
The free algebra $H^*(L[s])$ (which can easily be calculated from the
Hochschild-Serre spectral sequence) appears as a subalgebra of $H^*(L[z,s])$,
and Fishel, Grojnowski, and Teleman also prove that if $\mfb$ is the Iwahori
subalgebra $\{f \in L[z] : f(0) \in \mfb_0\}$ then $H^*(\mfb[s])$ is the free
algebra $H^*(\mfb_0[s]) \otimes_{H^*(L[s])} H^*(L[z,s])$. In this case
their proof does not yield explicit generating cocycles.

The purpose of this paper is to show that $H^*(\mfp[s])$ is a free
super-commutative algebra, and determine the degrees of the generators, when
$\mfp$ is a standard parahoric in the twisted loop algebra $L[z^{\pm
1}]^{\tilde{\sigma}}$, for $\sigma$ a (possibly trivial) diagram
automorphism of $L$. Our proof is along the same lines as \cite{FGT08}; in
particular, we are able to give an explicit description of cocycles for the
relative cohomology, and hence apply Feigin's spectral sequence to determine
the cohomology of the truncatations $\mfp / z^N \mfp$ when $N$ is a multiple of
the order of $\sigma$. Combined, our results for $L[z]^{\tilde{\sigma}}$ give
an extension of the strong Macdonald theorems to match the affine version of
Macdonald's constant term identity. For a general parahoric, our calculation
reveals that $H^*(\mfp[s])$ is isomorphic to $H(\mfp_0[s])
\otimes_{H^*(L_0[s])} H(L[z,s]^{\tilde{\sigma}})$, and hence can be viewed as
providing an interpolation between the two extremal results of Fishel,
Grojnowski, and Teleman.

The algebras $H^*(\mfp / z^N \mfp)$ also have an interesting description.  As
in (\ref{E:strongmacdonald1}), the algebras $H^*(L[z]^{\tilde{\sigma}} / z^N)$
are free, but this is no longer the case with a non-trivial parabolic
component.  The algebra $H^*(\mfp / z^N \mfp)$ is isomorphic to $H^*(\mfg_0)
\otimes \Coinv(L^{\sigma}, \mfg_0) \otimes_{H^*(L^{\sigma})}
H^*(L[z]^{\tilde{\sigma}}/z^N)$, where $\mfg_0 = \mfp_0 \cap \overline{\mfp_0}$
is the reductive component of the parabolic $\mfp_0$, and
$\Coinv(L^{\sigma},\mfg_0)$ is the parabolic subalgebra of the coinvariant
algebra of the Weyl group of $L^{\sigma}$.  A classic theorem of Borel states
that $\Coinv(L^{\sigma},\mfg_0)$ is isomorphic to the cohomology algebra of the
generalized flag variety $X$ corresponding to the Lie algebra pair
$(L^{\sigma}, \mfp_0)$ \cite{Bo53} \cite{BGG73}. The cohomology of $X$ is in
turn isomorphic to the Lie algebra cohomology algebra $H^*(L^{\sigma},
\mfg_0)$. If $\mfp$ is a parahoric in an untwisted loop algebra, then it is not
hard to show that $H^*(\mfp / z \mfp, \mfg_0)$ is isomorphic to
$H^*(L^{\sigma},\mfg_0)$, and hence in the simplest case our result gives a Lie
algebraic proof of Borel's theorem. This is not the first description of
$H^*(L^{\sigma},\mfg_0)$ using Lie algebraic methods: the cohomology of $X$ can
also be described using the Schubert cells, and a basis of $H^*(L^{\sigma},
\mfg_0)$ dual to the Schubert cells has been worked out in Lie algebraic terms
by Kostant \cite{Ko63a}. However, this description omits the ring structure.

One intriguing consequence of Hanlon's conjecture (\ref{E:strongmacdonald1}) is
that $H^*(L[z]/z^{N})$ is isomorphic as a vector space to $H^*(L)^{\otimes N}$.
Since $L[z]/(z^{N} - t) \iso L^{\oplus N}$ for $t \neq 0$, this means that
while the structure of $L[z]/(z^{N} - t)$ changes dramatically as $t$
degenerates to zero, the cohomology is unchanged.  Hanlon termed this
``property M'', and conjectured that it holds not only for semisimple Lie
algebras, but also for the nilpotent radical of a parabolic in a semisimple Lie
algebra and the Heisenberg Lie algebras \cite{Ha90}. Kumar gave counterexamples
to property M for the nilpotent radical of a parabolic \cite{Ku99}. The
conjecture for Heisenberg Lie algebras remains open, along with a number of
other questions \cite{Ha94} \cite{HW03}. In the case of a parahoric in a
twisted loop algebra $L[z^{\pm 1}]^{\tilde{\sigma}}$, if $t \neq 0$ then the
truncation $\mfp / (z^N - t) \mfp$ is isomorphic to $L^{\bigoplus N/k}$,
irregardless of the parahoric component. Our calculation shows that the
cohomology is unchanged for $L[z]^{\tilde{\sigma}} / (z^N - t)$ as $t$
degenerates to zero, but degenerates from $H^*(L)^{\otimes N/k}$ to
$H^*(\mfg_0) \otimes H^*(L^{\sigma}, \mfg_0) \otimes_{H^*(L^{\sigma})}
H^*(L)^{\otimes N/k}$ for a general parahoric truncation $\mfp / (z^N-t) \mfp$.

The proof of the strong Macdonald theorem in \cite{FGT08} is based on a
Laplacian calculation for $H^*(L[z,s])$ using the unique Kahler metric on the
loop Grassmannian. The Laplacian calculation shows that the ring of harmonic
forms is isomorphic to a ring of basic and invariant forms on the arc space
$L[[z]]$.  Kostant's theorems about adjoint orbits in a reductive Lie algebra
extend immediately from $L$ to $L[[z]]$, and can be used to determine the ring
of basic and invariant forms on $L[[z]]$. In the case of the parahoric, the
corresponding homogeneous space has many Kahler metrics. To follow the line of
the proof in \cite{FGT08}, we show that there is a particular choice of Kahler
metric that makes an analogous Laplacian calculation work. The ring of harmonic
forms is isomorphic to (a ring similar to) the ring of basic and invariant
forms on $\hat{\mfp}$. To calculate this ring, we study the adjoint orbits on
the twisted arc space $L[[z]]^{\tilde{\sigma}}$, proving a slice theorem for
twisted arcs in the regular semisimple locus, and an analogue of the Kostant
slice theorem (the significant facts about adjoint orbits extend immediately to
the arc space, but this is no longer the case when the diagram automorphism is
involved).  

Removing the super-notation, the cohomology ring of $\mfp[s]$ is isomorphic to
the cohomology ring of $\mfp$ with coefficients in the symmetric algebra $S^*
\mfp^*$ of the restricted dual of $\mfp$. Frenkel and Teleman have shown that
$H^*(\mfb; S^* \mfn^*)$ is a free algebra (and determined the degrees of the
generators) when $\mfb$ and $\mfn$ are Iwahori and nilpotent subalgebras
respectively of an untwisted loop algebra \cite{FT06}. We prove Frenkel and
Teleman's result in the twisted case and calculate the cohomology of the
corresponding truncation $\mfb / z^N \mfn$.  More generally, strong Macdonald
theorems for different choices of coefficients might allow us to determine the
cohomology of other truncations, such as $L[z]^{\tilde{\sigma}} / z^N$ when $N$
is not divisible by the order of $\sigma$. At the moment, this question appears
to be open. The question of finding a Lie algebra analogue of the constant
term identity for Koornwinder-Macdonald polynomials \cite{Di96} also seems to
be open.

\subsection{Organization}
Section \ref{S:overview} contains an overview of our cohomology results and the
connection with the constant term identity. Section \ref{S:laplacian} contains
the Laplacian calculation. Section \ref{S:jetschemes} defines twisted jet and
arc schemes.  Section \ref{S:orbit} contains theorems on adjoint orbits of
twisted jet and arc groups. Section \ref{S:mainproof} contains the cohomology
calculations of $H^*(\mfp[s])$ and $H^*(\mfb,S^* \mfn^*)$, while Section
\ref{S:truncatedcoh} contains the spectral sequence argument for calculating
the cohomology of the truncation $\mfp / z^N \mfp$.

\subsection{Acknowledgements}
This paper forms part of my Ph.D. thesis at U.C. Berkeley.  I thank my advisor,
Constantin Teleman, for suggesting the project and for many helpful
conversations. I also thank Dustin Cartwright, Anton Gerashenko, and Kevin Lin
for helpful conversations about algebraic geometry. This work was supported in
part by NSERC. Additional support was received from NSF grants DMS-1007255 and
DMS-0709448.

\tableofcontents

\section{Cohomology of standard parahorics}\label{S:overview}

\subsection{Notation and terminology}\label{SS:notation}

We fix the following terminology and notation throughout the paper, except
where explicitly stated. $L$ will be a reductive Lie algebra with diagram
automorphism $\sigma$ of finite order $k$. By definition $L$ has a triangular
decomposition $L = \overline{\mfu_0} \oplus \mfh \mfu_0$ where the Cartan
algebra $\mfh$ and nilpotent radicals $\overline{\mfu_0}$ and $\mfu_0$ are
$\sigma$-invariant, and such that $\sigma$ permutes the simple roots
corresponding to the Borel $\mfh \oplus \mfu_0$. We say that a Cartan,
Borel, or nilpotent radical is \emph{compatible with $\sigma$} if it
appears in such a decomposition.  The \emph{twisted loop algebra} is the Lie
algebra $\mfg = L[z^{\pm 1}]^{\tilde{\sigma}}$, where $\tilde{\sigma}$ is the
automorphism sending $f(z) \mapsto \sigma(f(q^{-1} z))$ for $q$ a fixed $k$th
root of unity.  $\mfg$ can be written as
\begin{equation*}
    \mfg = \bigoplus_{i=0}^{k-1} L_{a} \otimes z^a \C[z^{\pm k}],
\end{equation*}
where $L_a$ is the $q^a$th eigenspace of $\sigma$. If $L$ is simple then each
$L_a$ is an irreducible $L_0$-module. In particular if $L$ is simple then $L_0$
is also simple; in general $L_0$ will be reductive. A reductive Lie algebra $L$
has an anti-linear Cartan involution $\overline{\cdot}$ and a contragradient
positive-definite Hermitian form $\{,\}$. These two structures extend to the
twisted loop algebra $\mfg$ so that for any grading of type $d$,
$\overline{\mfg_n} = \mfg_{-n}$ and $\mfg_m \perp \mfg_n$ when $m \neq n$. 

The root system of $\mfg$ can be described as follows.  Let $\mfh$ be a Cartan
subalgebra of $L$ compatible with the diagram automorphism. Then $\mfh_0 :=
\mfh^{\sigma}$ is a Cartan in $L_0$, and $L_0$ has a set
$\alpha_1,\ldots,\alpha_l$ of simple roots which are projections of simple
roots of $L$. The roots of $\mfg$ can be described as $\alpha+ n \delta
\in \mfh_0^* \times \Z$ where either $\alpha$ is a weight of $L_a$ with
$n \equiv a \mod k$, or $\alpha=0$ and $n\neq 0$, and $\delta$ comes from the
rotation action of $\C^*$ on $\mfg$. Assume that $L$ is simple, and let
$\psi$ be either the highest weight of $L_1$ (an irreducible $L_0$-module) if
$k>1$, or the highest root of $L$, if $k=1$.  Then the set $\alpha_0 = \delta -
\psi, \alpha_1, \ldots, \alpha_l$ is a complete set of simple roots for $\mfg$.
If $L$ is reductive then we can choose a set of simple roots by decomposing
$L$ as a direct sum of $\sigma$-invariant simple subalgebras plus centre, and
taking the simple root sets from each corresponding factor of $\mfg$.

The twisted loop algebra $\mfg$ can be given a $\Z$-grading by assigning degree
$d_i \geq 0$ to the positive root vector associated to $\alpha_i$. In Kac's
terminology this is called a grading of type $d$ \cite{Ka83}. A \emph{parahoric
subalgebra} of $\mfg$ is a subalgebra of the form $\mfp = \bigoplus_{n \geq 0}
\mfg_n$, for some $\Z$-grading of $\mfg$ of type $d$.  A parahoric subalgebra
contains a \emph{nilpotent subalgebra} defined by $\mfu = \bigoplus_{n > 0}
\mfg_n$. We will say that a parahoric is \emph{standard} with respect to 
the choice of simple roots if it comes from a grading of type $d$ such that
$d_i>0$ whenever $\alpha_i$ is of the form $\delta - \psi$ for $\psi \in
\mfh_0^*$. Suppose $\mfp$ is a standard parahoric. Let $S = \{ \alpha_i : d_i =
0 \}$, and $\mfp_0$ be the standard parabolic subalgebra of $L_0$ defined by
\begin{equation*}
    \mfp_0 = \mfh_0 \oplus \bigoplus_{\alpha \in \Delta^+} (L_0)_{\alpha}
        \oplus \bigoplus_{\alpha \in \Delta^- \cap \Z[S]} (L_0)_{\alpha},
\end{equation*}
where $\Delta^{\pm}$ are the positive and negative roots of $L_0$ with respect
to the chosen simple roots. Then $\mfp = \{ f \in \mfg : f(0) \in \mfp_0\}$,
while $\mfu = \{ f \in \mfg : f(0) \in \mfu_0 \}$, where $\mfu_0$ is the
nilpotent radical of $\mfp_0$. Note that in this context the nilpotent radical
of an algebra $\mfk$ is defined to be the largest nilpotent ideal in
$[\mfk,\mfk]$ (or equivalently the intersection of the kernels of all
irreducible representations), so that $\mfu_0$ does not intersect the centre of
$L$. If $\mfp_0$ is a Borel, then $\mfp$ is called a standard Iwahori
subalgebra.

The completion of a subalgebra $K \subset \mfg$ with respect to a $\Z$-grading
is the algebra $\hat{K} = \lim_{\leftarrow} K / K^{(k)}$, where $K^{(k)} =
\bigoplus_{n > k} K_n$. If a parahoric subalgebra $\mfp$ is completed with
respect to a grading of $\mfg$ of type $d$, the result is a pro-Lie algebra
$\hat{\mfp}$. The pro-algebra structure on $\hat{\mfp}$ is independent of the
choice of grading. The dual of a pro-algebra $\hat{\mfp}$ will always refer to
the continuous dual $\hat{\mfp}^* \iso \bigoplus \mfp_n^*$ with respect to
the inverse limit topology. The continuous cohomology $H^*_{cts}(\hat{\mfp};V)$
is defined similarly to the ordinary cohomology using a version of the Koszul
complex with continuous cochains.  We refer to \cite{Fu86} for more details on
continuous cohomology. In this case, continuous cohomology is the same as the
restricted cohomology of \cite{FGT08}.

\subsection{Exponents and diagram automorphisms}
The exponents of $L$ are integers $m_1,\ldots,m_l$ such that $H^*(L)$ is the
free super-commutative algebra generated in degrees $2m_1+1,\ldots,2m_l+1$,
where $l$ is the rank of $L$. Equivalently, we can define the exponents by
saying that $(S^* L^*)^L$ is the free commutative algebra generated in degrees
$m_1+1,\ldots,m_l+1$. Extend the action of $\sigma$ to $S^* L^*$ by
$\sigma(f)(z) = f(\sigma^{-1} z)$. This convention is chosen so that
$\sigma(\ad^t(x) f) = \ad^t(\sigma(x)) \sigma(f)$ for all $f \in S^* L^*$ and
$x \in L$. Let $\mfM$ be the ideal in $(S^* L^*)^L$ generated by elements of
degree greater than zero. The diagram automorphism $\sigma$ acts diagonalizably
on the space $\mfM / \mfM^2$ of generators for $(S^* L^*)^L$, and consequently
it is possible to find homogeneous generators of $\C[Q]$ which are eigenvectors
of $\sigma$. 
\begin{defn}\label{D:twistedexp}
    Choose a set of homogeneous generators for $(S^* L^*)^L$ which are
    eigenvectors of $\sigma$. The exponents of $L$ can be sorted into different
    sets $m_1^{(a)},\ldots,m_{l_a}^{(a)}$, $a \in \Z_k$, by letting
    $m_1^{(a)}+1, \ldots,m_{l_a}^{(a)}+1$ be the list of degrees of homogeneous
    generators of $(S^* L^*)^L$ with eigenvalue $q^{-a}$ (note the negative
    exponent). We call the elements of these sets the exponents of $L_a$. 
\end{defn}

Recall that if $V$ is an $L_0$-module and $\{h,e,f\}$ is a principal
$\mfsl_2$-triple in $L_0$, then the generalized exponents of $V$ are the
eigenvalues of $h/2$ on the subspace $V^{L_0^e}$ fixed by the abelian
subalgebra $L_0^e$. The generalized exponents are always non-negative integers,
and the dimension of $V^{L_0^e}$ is equal to the dimension of the zero weight
space of $V$.
\begin{prop}\label{P:twistedexp}
    The exponents of $L_a$ are the generalized exponents of $L_a$ as an $L_0$-module.
\end{prop} 
The proof of Proposition \ref{P:twistedexp} will be given in Subsection
\ref{SS:twistedexp}. The generalized exponents of $L_0$ are the same as the
ordinary exponents, and $l_0$ is the rank of $L_0$, so there is no conflict in
our terminology. In general $l_a$ is the dimension of $\mfh \cap L_a$, where $\mfh$
is a Cartan compatible with $\sigma$.  If $L$ is simple, then $k$ is either $1$
or $2$, except when $L=\mfso(8)$ in which case $k$ can be $3$ and $L_1$ is
isomorphic to $L_2$. As a result, the exponents of $L_a$ are the same as the
exponents of $L_{-a}$. A principal $\mfsl_2$-triple in $L_0$ is also principal
in $L$ (see Lemma \ref{L:diagramregular}), so $L^e$ is abelian and hence
$L_a^{L_0^e} = L_a^e$, simplifying the definition of generalized exponents in
this case. The eigenvalues of $h/2$ give a principal grading $L_a = \bigoplus
L_a^{(i)}$ of each $L_a$ such that $L = \bigoplus_i \bigoplus_a L_a^{(i)}$ is a
principal grading for $L$. The representation theory of $\mfsl_2$ then implies:
\begin{cor}\label{C:twistedexp}
    The multiplicity of $m$ in the list of exponents of $L_a$
    is $\dim L_a^{(m)} - \dim L_a^{(m+1)}$, where $L_a = \bigoplus
    L_a^{(i)}$ is a principal grading.
\end{cor}
The exponents of $L_a$ can be easily determined when $L$ is simple, and are
given in the following table:

\begin{tabular}{c|c|c|c|c}
    \hline
    Type of $L$ \TabTop \TabBot & $k$ & Type of $L_0$ & Exponents of $L_0$ & Exponents of $L_{\pm 1}$ \\
    \hline
    $A_{2n}$ \TabTop \TabBot & 2 & $B_n$ & $1,3,\ldots,2n-1$ & $2,4,\ldots,2n$ \\
    \hline
    $A_{2n-1}$ \TabTop \TabBot & 2 & $C_n$ & $1,3,\ldots,2n-1$ & $2,4,\ldots,2n-2$ \\
    \hline
    $D_{n} $ \TabTop \TabBot & 2 & $ B_{n-1} $ & $1,3,\ldots,2n-3$ & $n-1$ \\
    \hline
    $E_6$ \TabTop \TabBot & 2 & $ F_4 $ & $1,5,7,11$ & $4,8$ \\
    \hline
    $D_4$ \TabTop \TabBot & 3 & $ G_2 $ & $1,5$ & $3$ \\
\end{tabular}

\subsection{Cohomology of superpolynomials in a standard parahoric} 
Let $\mfp = \{f \in \mfg : f(0) \in \mfp_0\}$ be a standard parahoric in a
twisted loop algebra $\mfg$, and let $\hat{\mfp}[s]$ denote the superpolynomial
algebra in one odd variable with values in $\hat{\mfp}$. The cohomology of the
super Lie algebra $\hat{\mfp}[s]$ can be calculated as in the ordinary case
using the Koszul complex, so any grading on $\hat{\mfp}[s]$ induces a grading
on $H^*_{cts} (\hat{\mfp}[s])$. In particular $H^*_{cts}(\hat{\mfp}[s])$ is
graded by $z$-degree and by $s$-degree. 
\begin{thm}\label{T:parahoriccoh}
    Let $m^{(a)}_1,\ldots,m^{(a)}_{l_a}$ denote the exponents of $L_a$, and let
    $r_1,\ldots, r_{l_0}$ denote the exponents of the reductive algebra $\mfp_0
    \cap \overline{\mfp_0}$, where $\mfp_0$ is a parabolic in $L_0$. If $\mfp$
    is the standard parahoric $\{f \in L[[z]]^{\tilde{\sigma}}
    : f(0) \in \mfp_0\}$ then the cohomology ring $H^*_{cts}(\hat{\mfp}[s])$ is
    a free super-commutative algebra generated in degrees given in the
    following table:

    \noindent
    \begin{tabular}{c|c|c|c}
        \hline
        Cohomological degree & $s$-degree & $z$-degree & Index set \TabTop \TabBot \\
        \hline
        $2 r_i+1$ & $0$ & $0$ & $i=1,\ldots,l_0$ \TabTop \TabBot\\
        \hline
        $r_i+1$ & $r_i+1$ & $0$ & $i=1,\ldots,l_0$ \TabTop \TabBot \\
        \hline
        $m_i^{(-a)}+1$ & $m_i^{(-a)}+1$ & $kn -a$ & 
        $n \geq 1$, $a=0,\ldots,k-1$, $i=1,\ldots,l_{-a}$ \TabTop \TabBot \\
        \hline
        $m_i^{(-a)}+1$ & $m_i^{(-a)}$ & $kn -a$ & 
        $n \geq 1$, $a=0,\ldots,k-1$, $i=1,\ldots,l_{-a}$ \TabTop \TabBot \\ 
    \end{tabular}
\end{thm}
To prove Theorem \ref{T:parahoriccoh}, we give an explicit description of a
generating set of cocycles for the relative cohomology ring
$H_{cts}^*(\hat{\mfp},\mfg_0; S^* \hat{\mfp}^*)$, where $\mfg_0 = \mfp_0 \cap
\overline{\mfp_0}$.  Choose a set of generators $I_i^{a}$, $a \in \Z_k$,
$i=1,\ldots, l_a$ for $(S L^*)^*$ such that $I_i^{a}$ is an eigenvector of
$\sigma$ with eigenvalue $q^{-a}$. Also choose a set of homogeneous generators
$R_1,\ldots,R_{l_0}$ for $(S^* \mfg_0^*)^{\mfg_0}$. The polynomial functions
$I_k^a$ on $L$ induce functions $\tilde{I}_k^a : L[[z]] \arr \C[[z]]$, and the
coefficients $[z^n] \tilde{I}^a_k$ of $z^n$ in $\tilde{I}_k^a$ restrict to
$\hat{\mfp}$-invariant polynomial functions on $\hat{\mfp}$. Similarly, the
polynomials $R_i$ on $\mfp_0$ can be pulled back via the quotient map $\mfp
\arr \mfg_0$ to $\hat{\mfp}$-invariant polynomials on $\hat{\mfp}$.  Finally,
1-cocycles can be constructed as follows. If $J$ is a derivation of
$\hat{\mfp}$ that kills $\mfg_0$ and $\phi \in S^k \hat{\mfp}^*$ is
$\hat{\mfp}$-invariant then the tensor
\begin{equation}\label{E:1cocycle}
    \hat{\mfu} \otimes S^{k-1} \hat{\mfp} \arr \C : x \otimes s_1 \circ \ldots
        \circ s_{k-1} \mapsto \phi(Jx \circ s_1\circ \cdots \circ s_{k-1}).
\end{equation}
is a cocycle (see Lemma \ref{L:cocycle}). 

\begin{thm}\label{T:relativecoh} Let $\mfp$ be a standard parahoric in $\mfg$,
    and let $J$ be the derivation from Theorem \ref{T:basicforms}. Then there
    is a metric on the Koszul complex such that the harmonic cocycles for
    $H^*_{cts}(\hat{\mfp}, \mfg_0; S^* \hat{\mfp}^*)$ form a free
    supercommutative ring generated by the cocycles in the following table:
    \vspace{0.1in}

    \noindent
    \begin{tabular}{c|c|c|c|c}
        \hline
         Cocycle description \TabTop \TabBot & Coh. deg. & Sym. deg & $z$-deg. & Index set \\
        \hline
        $R_i$\TabTop \TabBot & $0$ & $\deg R_i$ & $0$ & $i=1,\ldots, l_0$ \\
        \hline
        \multirow{2}{*}{$[z^{kn-a}] \tilde{I}_i^{-a}$} \TabTop & \multirow{2}{*}{$0$} & 
            \multirow{2}{*}{$\deg I_i^{-a}$} & \multirow{2}{*}{$kn-a$} & $n \geq 1,\ i=1,\ldots,l_{-a}$, \\
        \TabBot & & & & $a=0,\ldots,k-1$ \\
        \hline
        \multirow{2}{*}{$x \otimes s \mapsto [z^{kn-a}] \tilde{I}_j^{-a} (Jx\circ s)$} \TabTop & 
            \multirow{2}{*}{$1$} & \multirow{2}{*}{$\deg I_j^{-a} -1$} & \multirow{2}{*}{$kn-a$} & $n \geq 1,\ i=1,\ldots,l_{-a}$, \\
        \TabBot & & & & $a=0,\ldots,k-1$ \\
    \end{tabular}
\end{thm}
Proving Theorem \ref{T:relativecoh} is the main concern of the paper; the proof is
finished in Subsection \ref{SS:standardbasic}.
\begin{proof}[Proof of Theorem \ref{T:parahoriccoh} from Theorem \ref{T:relativecoh}]
    Since the bracket of $\hat{\mfp}[s]$ is zero on the odd component, the
    Koszul complex for $\hat{\mfp}[s]$ reduces to the Koszul complex for
    $\hat{\mfp}$ with coefficients in $S^* \hat{\mfp}^*$. Thus there is a ring
    isomorphism $H_{cts}^*(\hat{\mfp}[s]) \iso H_{cts}^*(\hat{\mfp};S^*
    \hat{\mfp}^*)$ in which $H_{cts}^{n-q}(\hat{\mfp};S^q \hat{\mfp}^*)$
    corresponds to the cohomology classes in $H_{cts}^n(\hat{\mfp}[s])$ of
    $s$-degree $q$.  This isomorphism preserves $z$-degree. 
    The degree zero component of $\mfp$ is $\mfg_0 = \mfp_0 \cap
    \overline{\mfp_0}$, a reductive algebra which is the quotient of $\mfp$ by
    the standard nilpotent subalgebra $\mfu$. It follows from the
    Hochschild-Serre spectral sequence (in particular Theorem 12 of
    \cite{HS53}) that there is a ring isomorphism
    \begin{equation*}
        H^*_{cts}(\hat{\mfp}; S^* \hat{\mfp}^*) \iso H^*(\mfg_0) \otimes
            H^*_{cts}(\hat{\mfp}, \mfg_0; S^* \hat{\mfp}^*).
    \end{equation*}
    Then Theorem \ref{T:parahoriccoh} follows from the description of
    relative cohomology.
\end{proof}
When $\mfp_0 = L_0$, Theorem \ref{T:parahoriccoh} states that the algebra
$H^*_{cts}(L[z,s])^{\tilde{\sigma}}$ is the free super-commutative algebra with
generators in tensor degree $2m_i^{(a)}+1$ and $2m_i^{(a)}+2$, and $z$-degree
$nk+a$, for $a=0, \ldots,k-1$, $i=1,\ldots,l_a$, and $n \geq 0$. In addition
the Hochschild-Serre spectral sequence implies that $H^*(\mfp_0[s]) \iso
H^*(\mfg_0) \otimes (S^* \mfg_0^*)^{\mfg_0}$, where $\mfg_0 = \mfp_0 \cap
\overline{\mfp_0}$, so $H^*(\mfp_0[s])$ is isomorphic to the subalgebra of
$H^*_{cts}(\hat{\mfp}[s])$ of $z$-degree zero. In fact, the inclusion is the
pullback map given by evaluation at zero, as can be seen from the explicit
description of harmonic cocycles, so $H^*_{cts}(\hat{\mfp}[s])$ is naturally
isomorphic to $H^*(\mfp_0[s]) \otimes_{H^*(L_0[s])}
H^*_{cts}(L[z,s]^{\tilde{\sigma}})$. 

We can also ask for an explicit description of the relative cohomology groups
$H^*_{cts}(\hat{\mfp},\mfg_0; S^* \hat{\mfu}^*)$. In this case, we can only
provide an answer when $\mfp$ is an Iwahori subalgebra---that is, a standard
parahoric $\{ f \in L[[z]]^{\tilde{\sigma}} : f(0) \in \mfp_0\}$ where $\mfp_0$
is a Borel subalgebra. 
\begin{thm}\label{T:nilpotentcoh}
    Let $\mfb$ be an Iwahori subalgebra of $\mfg$, and let $\mfn$ be the nilpotent
    subalgebra. Let $J$ be the derivation from Theorem \ref{T:basicforms}. Then
    there is a metric on the Koszul complex such that the harmonic cocycles for
    $H^*_{cts}\left(\hat{\mfb}, \mfh_0; S^* \hat{\mfn}^*\right)$ form a free
    supercommutative ring generated by the cocycles in the following table:
    \vspace{0.1in}

    \noindent
    \begin{tabular}{c|c|c|c|c}
        \hline
         Cocycle description \TabTop \TabBot & Coh. deg. & Sym. deg & $z$-deg. & Index set \\
        \hline
        \multirow{2}{*}{$[z^{kn-a}] \tilde{I}_i^{-a}$} \TabTop &  
            \multirow{2}{*}{$0$} & \multirow{2}{*}{$\deg I_i^{-a}$} & \multirow{2}{*}{$kn-a$} 
            & $n \geq 1,\ i=1,\ldots,l_{-a}$, \\
        \TabBot & & & & $a=0,\ldots,k-1$ \\
        \hline
        \multirow{2}{*}{$x \otimes s \mapsto [z^{kn-a}] \tilde{I}_j^{-a} (Jx\circ s)$} \TabTop & 
         \multirow{2}{*}{$1$} & \multirow{2}{*}{$\deg I_j^{-a} -1$} & \multirow{2}{*}{$kn-a$} 
          & $n \geq 1,\ i=1,\ldots,l_{-a}$, \\
         \TabBot & & & & $a=0,\ldots,k-1$ \\
    \end{tabular}
\end{thm}
Theorem \ref{T:nilpotentcoh} can be used to calculate
$H^*_{cts}\left(\hat{\mfb}; S^* \hat{\mfn}\right)$ as in the proof of Theorem
\ref{T:parahoriccoh}. With an appropriate degree shift, the cohomology ring
$H^*_{cts}\left(\hat{\mfb},\mfh; S^* \hat{\mfn}\right)$ can also be regarded as
the $\mfh$-invariant part of $H^*_{cts}(\hat{\mfn}[s])$. The proof of Theorem
\ref{T:nilpotentcoh} will be completed in Subsection \ref{SS:nilpotentbasic}.

\subsection{Cohomology of the truncated algebra}
If $N$ is a multiple of $k$ then $z^N L[z]^{\tilde{\sigma}}$ is a subset of
$L[z]^{\tilde{\sigma}}$, and hence $z^N \mfp$ is an ideal of $\mfp$.  Theorem
\ref{T:relativecoh} can be used to determine the cohomology of the
finite-dimensional Lie algebra $\mfp / z^N \mfp$.

Recall that the coinvariant algebra of the Weyl group $W(L_0)$ is the quotient
of $S^* \mfh_0^*$ by the ideal generated by $(S^{>0} \mfh_0^*)^{W(L_0)}$.  We
define $\Coinv(L_0,\mfg_0)$ to be the graded algebra which is the quotient of
$(S^* \mfg_0^*)^*{\mfg_0}$ by the ideal generated by $(S^{>0} L_0^*)^{L_0}$,
where $S^* L_0^*$ acts on $S^* \mfg_0^*$ by restriction. By the Chevalley
restriction theorem, $\Coinv(L_0,\mfg_0)$ is isomorphic to the subalgebra of
$W(\mfg_0)$-invariants in the coinvariant algebra of $W(L_0)$. It is well-known
that the Poincare series for $\Coinv(L_0,\mfg_0)$ with the symmetric grading is
$\prod_{i=1}^{l_0} (1-q^{r_i+1})^{-1} \prod_{i=1}^{l_0}
\left(1-q^{m_i^{(0)}+1}\right)$, where $m_i^{(0)}$ refers to the exponents of
$L_0$ and $r_i$ refers to the exponents of $\mfg_0$. The dimension of
$\Coinv(L_0,\mfg_0)$ is $|W(L_0)| / |W(\mfg_0)|$. 

\begin{thm}\label{T:truncatedcoh}
    Let $m^{(a)}_1,\ldots,m^{(a)}_{l_a}$ denote the exponents of $L_a$, and let
    $r_1,\ldots,r_{l_0}$ be the exponents of the reductive Lie algebra $\mfg_0
    = \mfp_0 \cap \overline{\mfp_0}$. Let $\Coinv(L_0, \mfg_0)$ denote the
    coinvariant algebra, with a cohomological grading (resp. $z$-grading)
    defined by setting the cohomological degree (resp. $z$-degree) to twice
    (resp.  $N$ times) the symmetric degree. 

    If $\mfp$ is the standard parahoric $\{f \in L[[z]]^{\tilde{\sigma}} : f(0)
    \in \mfp_0\}$ and $N$ is a multiple of $k$ then the cohomology algebra
    $H^*(\mfp / z^N \mfp)$ is isomorphic to $\Coinv(L_0,\mfg_0) \otimes
    \Lambda$, where $\Lambda$ is the free super-commutative algebra
    generated in degrees given by the following table:

    \noindent
    \begin{tabular}{c|c|c}
        \hline
        Cohomological degree \TabTop \TabBot & $z$-degree & Index set\\
        \hline
        $2r_i + 1$ \TabTop \TabBot & $0$ & $i=1,\ldots,l_0$ \\
        \hline
        $2m_i^{(a)}+1$ \TabTop \TabBot & $N m_i^{(a)} + nk + a$ 
            &  $a=0,\ldots,k-1$, $i=1,\ldots,l_{a}$, $0 < nk+a < N$ \\
    \end{tabular}
\end{thm}
As in the proof of Theorem \ref{T:parahoriccoh}, we have
\begin{equation*}
    H^*(\mfp / z^N \mfp) \iso H^*(\mfg_0) \otimes
        H^*(\mfp / z^N \mfp, \mfg_0),
\end{equation*}
so we only need to compute the relative cohomology. This will be done with a
spectral sequence argument in Section \ref{S:truncatedcoh} (see Proposition
\ref{P:truncatedrel}).

When the parabolic component is trivial, $H^*\left(L[z]^{\tilde{\sigma}}/
z^N\right)$ is simply the free super-commutative algebra with one set of
generators in cohomological degree $2m_i^{(0)}+1$ and $z$-degree $0$ for
$i=1,\ldots,l_0$, and another set of generators in cohomological degree
$2m_i^{(a)}+1$ and $z$-degree $N m_i^{(a)} + nk + a$, where $a=0,\ldots,k-1$,
$i=1,\ldots,l_a$, and $n$ such that $0 < nk+a < N$. Theorem
\ref{T:truncatedcoh} can be restated as saying that $H^*(\mfp / z^N \mfp)$ is
the algebra $H^*(\mfg_0) \otimes \Coinv(L_0,\mfg_0) \otimes_{H^*(L_0)}
H^*(L[z]^{\tilde{\sigma}} / z^N)$. 

In Lemma \ref{L:borelpic}, we prove that if $\mfg$ is untwisted and $N=1$ then
$H^*(\mfp / z \mfp,\mfg_0)$ is isomorphic to $H^*(L_0,\mfg_0)$. This algebra is
in turn isomorphic to the cohomology ring of the generalized flag variety
corresponding to the pair $(L_0,\mfp_0)$. The $z$-grading on $H^*(\mfp / z
\mfp, \mfg_0)$ corresponds to the holomorphic grading appearing in the Hodge
decomposition. The fact that $\Coinv(L_0,\mfg_0)$ is isomorphic to
$H^*(L_0,\mfg_0)$ is a classic theorem of Borel (\cite{Bo53}, see Theorem 5.5
of \cite{BGG73} for the parabolic case). Thus Theorem \ref{T:truncatedcoh} can
be seen as a generalization of Borel's theorem.

We can compare the cohomology of $\mfp / z^N \mfp$ with the cohomology of
more general truncations. If $P(z)$ is a polynomial in $z$, then $P(z^k)
L[z]^{\tilde{\sigma}}$ is a subset of $L[z]^{\tilde{\sigma}}$, and hence
$P(z^k) \mfp$ is an ideal of $\mfp$. We can assume that $P$ is monic, and write
$P = z^d + P_0$, where $d$ is the degree of $P$ and $P_0$ contains lower degree
terms. Suppose $x \in L_{i}$ for $i \geq 0$. Then $(z^{dk} + P_0(z^k)) x z^i$
is in $P(z^k) \mfp$ if and only if either $i>0$ or $x \in
\mfp_0$, so the dimension of $\mfp / P(z^k)\mfp$ is $d \cdot \dim L$. 
\begin{lemma}\label{L:coprime}
    If $P$ and $Q$ are coprime then $\mfp / P(z^k)Q(z^k) \mfp \iso \mfp /
    P(z^k) \mfp \oplus \mfp / Q(z^k) \mfp$.
\end{lemma}
By Lemma \ref{L:coprime} the study of $\mfp / P(z^k) \mfp$ reduces to the case
where $P$ is the power of a linear factor. In the untwisted case, $L[z] \iso
L[z-\alpha]$, so $L[z]/(z-\alpha)^N \iso L[z] / z^N$. However, in the twisted
case this argument does not apply, since the automorphism $z \mapsto q^{-1} z$
is different from $z-\alpha \mapsto q^{-1} (z-\alpha)$. In particular:
\begin{lemma}\label{L:deform}
    If $\alpha \neq 0$ then $\mfp / (z^k - \alpha) \mfp$ is isomorphic to $L$. 
\end{lemma}
\begin{proof}
    Let $\beta$ be a $k$th root of $\alpha$. Then evaluation at $\beta$ defines
    a morphism $\mfp / (z^k - \alpha) \mfp \arr L$. Both $L$ and $\mfp / (z^k -
    \alpha) \mfp$ have the same dimension, so we just need to show that this
    map is onto. Given $x \in L$, write $x = \sum_{i=0}^{k-1} x_i$ where $x_i
    \in L_i$.  Let $f = \sum_{i=1}^{k-1} x_i \beta^{-i} z^i + x_0 \alpha^{-1}
    z^k$. Then $f(\beta)=x$.
\end{proof}
The author does not know if an analogue of Lemma \ref{L:deform} holds for higher
powers of $(z^k - \alpha)$.  The main case of interest is $\mfp / (z^N - t)
\mfp$, which can be regarded as a deformation of $\mfp / z^N \mfp$. Since
$z^{N/k} - t$ splits into $N/k$ coprime linear factors, the algebra $\mfp /
(z^N - t) \mfp$ is isomorphic to $L^{\oplus N/k}$ for $t \neq 0$. At $t = 0$,
the algebra $\mfp / z^N \mfp$ has a large nilpotent ideal. Ignoring
$z$-degrees, Theorem \ref{T:truncatedcoh} tells us that
$H^*\left(L[z]^{\tilde{\sigma}}/z^N\right) \iso H^*(L)^{\otimes N/k}$, so the
cohomology of $L[z]^{\tilde{\sigma}}/(z^N - t)$ is independent of the value of
$t$. On the other hand, Theorem \ref{T:truncatedcoh} tell us that
$H^*\left(\mfp / (z^N-t) \mfp \right)$ changes from $H^*(L)^{\otimes N/k}$ to
$H^*(\mfg_0) \otimes H^*(L_0,\mfg_0) \otimes_{H^*(L_0)} H^*(L)^{\otimes N/k}$
as $t$ degenerates to zero, where $H^*(L_0)$ acts on $H^*(\mfg_0)$ via
pullback. Interestingly, the cohomology of $\mfp / z^k (z^N - t) \mfp$ is
unchanged as $t$ degenerates to zero.

If $\mfp = \mfb$ is an Iwahori and $\mfn$ is the nilpotent subalgebra, then a
similar analysis can be performed for $\mfb / z^N \mfn$. 
\begin{thm}\label{T:truncatednil}
    Let $m_1^{(a)},\ldots,m_{l_a}^{(a)}$ denote the exponents of $L_a$,
    let $\mfb$ be an Iwahori subalgebra of the twisted loop algebra $\mfg$, and
    let $\mfn$ be the nilpotent subalgebra.  Then $H^*(\mfb/z^N \mfn)$ is the
    free super-commutative algebra with a generator in cohomological degree
    $2m_i^{(a)}+1$ and $z$-degree $Nm_i^{(a)}+nk+a$ for every $a=0,\ldots,k-1$,
    $i=1,\ldots,l_a$, and $n$ such that $0 < nk +a\leq N$, as well as $l_0$
    generators of cohomological degree $1$ and $z$-degree $0$.
\end{thm}
As with Theorem \ref{T:truncatedcoh}, the proof of Theorem \ref{T:truncatednil}
reduces via the Hochschild-Serre spectral sequence to the computation of the
relative cohomology, which is also completed in Section \ref{S:truncatedcoh}
(see Proposition \ref{P:nilpotentrel}).

If $P(z)$ is a polynomial of degree $d$, then $\mfb / P(z^k) \mfn$ has
dimension $d \cdot \dim L + l_0$. Furthermore, $[\mfb, P(z^k) \mfh_0]$ is
contained in $P(z^k) \mfn$, and there is a morphism  $\mfb / P(z^k) \mfn \arr
\mfb / P(z^k) \mfb$ with kernel $P(z^k) \mfh_0$, so $\mfb / P(z^k) \mfn$ is a
central extension of $\mfb / P(z^k) \mfb$ of rank $l_0$. 
\begin{lemma}
    If $t \neq 0$ then $\mfb / (z^N - t) \mfn$ is isomorphic to $L^{\oplus N/k}
    \oplus \C^{l_0}$, where the second summand is abelian.
\end{lemma}
\begin{proof}
    $\mfb / (z^N - t) \mfb$ is isomorphic to the direct sum of $N/k$ copies of
    $L$. If $L$ is semisimple, then so is $\mfb / (z^N - t) \mfb$, so
    all central extensions are trivial. The reductive case reduces to the
    semisimple case by splitting off the centre. 
\end{proof}
Thus $H^*\left(\mfb / (z^N-t) \mfn\right)$ is also independent of $t$ when
$z$-degrees are disregarded.

\subsection{The Macdonald constant term identity}
Theorem \ref{T:truncatedcoh} can be used to prove the affine version of
Macdonald's constant term conjecture. If $\hat{\alpha} = \alpha + n \delta$ is
an affine root, $\alpha$ a weight of $L_0$, set $e^{\hat{\alpha}} = q^{-n}
e^{\alpha}$.  In a slight abuse of notation, the operator $[e^0]$ will denote
the sum of the $e^{n\delta}$ terms, ie. it is $\C(q)$-linear. Let $\delta^*$
denote the dual element to $\delta$.  The following theorem is Conjecture 3.3
of \cite{Ma82}, and was proven for all root systems by Cherednik \cite{Ch95}.
\begin{thm}[Cherednik]\label{T:affineconstant}
    Let $N$ be a multiple of $k$, and let $S_N$ be the set of real
    roots\footnote{ie. $\alpha \neq 0$} $\alpha + n \delta$ of the twisted loop
    algebra $\mfg$ with $0 \leq n \leq N$, such that $\alpha$ is a positive
    (resp. negative) root of $L_0$ if $n=0$ (resp. $n=N$). Let $\rho$ be the
    element of $\mfh_0$ such that $\alpha_i(\rho) = 1$ for all simple roots
    $\alpha_1,\ldots,\alpha_{l_0}$ of $L_0$, and let $\rho_N =
    -N\rho+\delta^*$.  Then
    \begin{equation*}
        [e^0] \prod_{\alpha \in S_N} (1-e^{-\alpha}) =
            \prod_{\alpha \in S_N} \left(1-q^{|\alpha(\rho_N)|}\right)^{\eps(\alpha)},
    \end{equation*}
    where $\eps(\alpha)$ is the sign of $\alpha(\rho_N)$.
\end{thm}
Define a twisted $q$-binomial coefficent for $a \in \Z_k$ and multiples $N,M$
of $k$ by

\begin{equation*}
    \binom{N}{M}_{k,a} = 
        \prod_{\substack{ N-M < i \leq N \\ i \equiv a \mod k }} (1-q^i)
             \prod_{\substack{ 0 < i \leq M \\ i \equiv a \mod k}} (1-q^i)^{-1}.
\end{equation*}
The right-hand side of Theorem \ref{T:affineconstant} can be simplified by
extending an idea of \cite{Ma82} from the untwisted case.
\begin{lemma}\label{L:affineconstant}
    The identity of Theorem \ref{T:affineconstant} is equivalent to
    \begin{equation}\label{E:affineconstant}
        [e^0] \prod_{\alpha \in S_N} (1-e^{-\alpha}) = 
            \prod_{a\in\Z_k} \prod_{i=1}^{l_a} \binom{N(m_i^{(a)}+1)}{N}_{k,a}.
    \end{equation}
\end{lemma}
\begin{proof}
    Let $\Delta_a$ be the set of weights of the $L_0$-module $L_a$, and
    let $\Delta_a^+$ denote the subset of $\alpha \in \Delta_a$ such that
    $\alpha(\rho) > 0$. If $\theta$ is an arbitrary function from positive
    integers to a multiplicative group, then
    \begin{equation*}
        \prod_{\alpha \in \Delta_a^+}
            \frac{\theta(\alpha(\rho)+1)}{\theta(\alpha(\rho))} = 
                \prod_{i=1}^{l_a} \frac{\theta\left(m_i^{(a)}+1\right)}{\theta(1)}.
    \end{equation*}
    To prove this, note that the eigenvalues of $\rho$ on $L_a$ are integers
    giving the principal grading of $L_a$, so the identity follows immediately
    from Corollary \ref{C:twistedexp} by comparing the number of times
    $\theta(m)$ occurs on the top versus the bottom.

    Define
    \begin{equation*}
        A_a = \prod_{\substack{\alpha + n \delta \in S_N \\ \alpha \in \Delta_a}}
            \left(1-q^{|\alpha(\rho_N)|}\right)^{\eps(\alpha)},
    \end{equation*}
    and set $\theta_{-a}(m) = (1-q^{Nm-a}) (1-q^{Nm-k-a}) \cdots (1-q^{Nm
    -N+k-a})$, for $a \in \Z_k$ represented by one of $0,\ldots,k-1$.  Then
    \begin{equation*}
        A_0 = \prod_{\alpha \in \Delta_0^+} \prod_{n=0}^{N/k-1} 
                \left(1-q^{N \alpha(\rho) - nk}\right)^{-1}
            \prod_{n=1}^{N/k} \left(1-q^{N\alpha(\rho) + nk}\right)
    \end{equation*}
    while if $a \neq 0$ we have
    \begin{equation*}
        A_a = \prod_{\alpha \in \Delta_a^+} \prod_{n=0}^{N/k-1}
            \left(1-q^{N\alpha(\rho)-nk-a}\right)^{-1} \left(1-q^{N\alpha(\rho) + nk+a}\right).
    \end{equation*}
    In both cases,
    \begin{equation*}
        A_a = \prod_{\alpha \in \Delta_a^+} \theta_{-a}(\alpha(\rho))^{-1} 
            \theta_{a}(\alpha(\rho)+1).
    \end{equation*}
    Even if $-a$ and $a$ are not congruent, $L_a$ and $L_{-a}$ are still isomorphic, so
    \begin{equation*}
        A_a A_{-a} = \prod_{\alpha \in \Delta_a^+} \theta_{-a}(\alpha(\rho))^{-1}
            \theta_{a}(\alpha(\rho)+1) \theta_{a}(\alpha(\rho))^{-1}
            \theta_{-a}(\alpha(\rho)+1).
    \end{equation*}
    Hence the right hand side of Theorem \ref{T:affineconstant} is equal to 
    \begin{equation*}
        \prod_{a =0}^{k-1} A_a = \prod_{a \in \Z_k} \prod_{i=1}^{l_a}
            \frac{\theta_{a}\left(m_i^{(a)}+1\right)}{\theta_{a}(1)},
    \end{equation*}
    as required.
\end{proof}

Let $C^*$ be a chain complex with an additional grading $C^* = \bigoplus
C^*_n$. The weighted Euler characteristic of $C^*$ is
\begin{equation*}
    \chi(C^*; q) = \sum_{n,i} (-1)^* \dim C^i_n q^n.
\end{equation*}
As in the unweighted case, the weighted Euler characteristic is invariant under
taking homology. Let $\mfp = \{f \in L[[z]]^{\tilde{\sigma}} : f(0) \in
\mfp_0\}$ be a standard parahoric and $\mfg_0 = \mfp_0 \cap \overline{\mfp_0}$.
Theorem \ref{T:affineconstant} can be proved by comparing the $z$-weighted
Euler characteristic for the Koszul complex of the pair $(\mfp/z^N\mfp,
\mfg_0)$ with the weighted Euler characteristic of the cohomology ring:
\begin{proof}
    Write $\mfp_0 = \mfg_0 \oplus \mfu_0$, for $\mfu_0$ the nilpotent radical.  Let
    $K$ be a compact subgroup acting on $L$ with complexified Lie algebra
    $\mfg_0$, and let $T$ be a maximal torus in $K$ with complexified Lie
    algebra $\mfh_0$. Let $\pi_a$ denote the respresentation of $K$ on $L_a$,
    and let $\phi$ and $\overline{\phi}$ denote the representation of $K$ on
    $\mfu_0$ and $\overline{\mfu_0}$ respectively. The weighted Euler
    characteristic of the Koszul complex is
    \begin{equation*}
        \chi(q) = \sum (-1)^i q^i \dim \left(\bigwedge^i (\mfp / z^N \mfp)^* \right)^{K}.
    \end{equation*}
    By orthogonality of traces of representations with respect to Haar measure, 
    \begin{equation*}
        \chi(q) = \int_K \det(\Id - \phi(k)) \det(\Id - q^N \overline{\phi}(k))
            \prod_{0 < n < N} \det(\Id - q^n \pi_n(k)) dk.
    \end{equation*}
    The integrand is conjugation invariant, so by the Weyl integral formula,
    \begin{align*}
        \chi(q) & = \frac{1}{|W(\mfg_0)|} 
            \int_T \det(\Id - \phi(t)) \det(\Id - q^N \overline{\phi}(t)) 
                \prod_{0 < n < N} \det(\Id - q^n \pi_n(t)) 
                \prod_{\alpha \in \Delta(\mfg_0)} (1 - e^{\alpha(t)})dt \\
                & = \frac{1}{|W(\mfg_0)|} [e^0] \prod_{\alpha \in \Delta(\mfg_0)} 
                        (1 - e^{\alpha}) \cdot \Phi,
    \end{align*}
    where $\Delta(\mfg_0)$ is the root system of $\mfg_0$ and
    \begin{equation*}
        \Phi = \prod_{\alpha \in \Delta^+(\mfg_0)} 
                (1-e^{-\alpha})^{-1} (1-q^N e^{\alpha})^{-1}
            \prod_{\alpha \in S_N} (1 - e^{-\alpha}) \prod_{0 < n < N} (1-q^{n})^{l_n}
    \end{equation*}
    (note that the inverses divide into the other multiplicands).  The
    coefficient of $q^j$ in $\Phi$ is (up to sign) the character of a
    $\mfg_0$-module, so $\Phi$ is $W(\mfg_0)$ invariant. Now we use the
    identity 
    \begin{equation}\label{E:macdonaldid}
        \sum_{w \in W(\mfg_0)} \prod_{\alpha \in \Delta^+(\mfg_0)}
            \frac{1 - q^N e^{w \alpha}}{1-e^{w \alpha}} = \prod_{i = 1}^{l_0}
                \frac{1-q^{N(r_i+1)}}{1-q^N}
    \end{equation}
    found in \cite{Ma82}\cite{Ma72} to get
    \begin{align*}
        \chi(q) & = \frac{1}{|W(\mfg_0)|} \prod_{i = 1}^{l_0} \frac{1-q^N}{1-q^{N(r_i+1)}} 
            \cdot \left[e^0\right] \sum_{w \in W(\mfg_0)} \prod_{\alpha \in \Delta^+(\mfg_0)}
                 \frac{1 - q^N e^{w \alpha}}{1-e^{w \alpha}} \prod_{\alpha \in \Delta(\mfg_0)}
                        (1 - e^{\alpha}) \cdot \Phi \\
            & = \frac{1}{|W(\mfg_0)|} \prod_{i = 1}^{l_0} \frac{1-q^N}{1-q^{N(r_i+1)}}
                \cdot \left[e^0\right] \sum_{w \in W(\mfg_0)} w \cdot 
                    \prod_{\alpha \in \Delta^+(\mfg_0)} \frac{1-q^N e^{\alpha}}{1-e^{\alpha}} 
                        \prod_{\alpha \in \Delta(\mfg_0)} (1 - e^{\alpha}) \cdot \Phi. 
    \end{align*}
    Since the action of $W(\mfg_0)$ does not change the constant term, this last sum gives
    \begin{equation*}
        \chi(q) = \prod_{i = 1}^{l_0} (1-q^{N(r_i+1)})^{-1} \prod_{0 < n \leq N} (1-q^{n})^{l_n}
                \cdot \left[e^0\right]  \prod_{\alpha \in S_N} (1 - e^{-\alpha}).
    \end{equation*}

    On the other hand, Theorem \ref{T:truncatedcoh} implies
    \begin{equation*}
        \chi(q) = \prod_{i=1}^{l_0} \left(1-q^{N(r_i+1)}\right)^{-1}
            \prod_{0 < n \leq N}\prod_{i=1}^{l_n} \left(1 - q^{Nm_i^{(n)} + n}\right)
    \end{equation*}
    Identifying these two equations gives the identity of Lemma
    \ref{L:affineconstant}. 
\end{proof}
Note that when $\mfp = \mfb$ is an Iwahori the equivalence follows without
using the Weyl integration argument or identity (\ref{E:macdonaldid}). The
$z$-weighted Euler characteristic identity for $H^*(\mfb / z^N \mfn, \mfh_0)$,
is similarly equivalent to the identity of Lemma \ref{L:affineconstant}.

\section{The Lapalacian calculation and the set of harmonic forms}\label{S:laplacian}

As in the overview, let $\mfp$ be a parahoric (in this case not necessarily
standard) in the twisted loop algebra $\mfg$, and $\mfu_0$ the corresponding
nilpotent.  Choose a homogeneous basis $\{z_k\}$ for $\mfu$, and let $\{z^k\}$
be the dual basis of $\hat{\mfu}^*$.  Let $(V,\pi)$ be a $\hat{\mfp}$-module.
The Koszul complex for $H_{cts}^*(\hat{\mfp},\mfg_0; V)$ is the chain complex
$(C^q,\deltabar)$ defined by
\begin{equation*}
    C^q(\hat{\mfp},\mfg_0; V) = \left(\bigwedge^q \hat{\mfu}^* \otimes V\right)^{\mfg_0}
\end{equation*}
and
\begin{equation*}
    \deltabar = \sum_{k \geq 1} \eps(z^k) \left( \frac{1}{2} \ad^t_{\mfu}(z_k)
        + \pi(z_k) \right).
\end{equation*}
If $C^q$ is given a positive-definite Hermitian form then the cohomology $H^*$
can be identified with the set $\ker \boxbar$ of harmonic forms, where $\boxbar
= \deltabar \deltabar^* + \deltabar^* \deltabar$. The goal of this section is
to calculate $\ker \boxbar$ for $V = S^* \hat{\mfp}^*$ and $V = S^*
\hat{\mfu}^*$, in a metric that we will introduce.
\begin{lemma}\label{L:cocycle}
    Let $V$ be an $\hat{\mfp}$-module, and $J$ a derivation of $\mfp$ which
    annihilates $\mfg_0$. If $\phi \in \bigwedge^k \hat{\mfu}^* \otimes V$
    is $\hat{\mfp}$-invariant then 
    \begin{equation}\label{E:cocycle}
        x_1 \wedge \cdots \wedge x_k \mapsto \phi(Jx_1,\ldots,Jx_k)
    \end{equation}
    is a cocycle in $C^k(\hat{\mfp},\mfg_0; V)$.
\end{lemma}
\begin{proof}
    Let $f$ be the cochain constructed as in equation (\ref{E:cocycle}). Then
    \begin{align*}
        (\deltabar f)(x_0,\ldots,x_k) 
            & = \sum_{i < j} (-1)^{i+j} f([x_i,x_j],\ldots,\check{x}_j,\ldots)  
                + \sum_i (-1)^i x_i f(\ldots, \check{x}_i, \ldots) \\
            & = \sum_{i < j} (-1)^{i+j} \phi(J[x_i,x_j],Jx_0, \ldots) 
                + \sum_i (-1)^i x_i \phi( Jx_0,\ldots) \\
            & = \sum_i (-1)^i \left(\ad^t(x_i) \phi\right) (J x_0,\ldots,\check{x}_i,\ldots) 
                + \sum_i (-1)^i x_i \phi( Jx_0,\ldots,\check{x}_i,\ldots),
    \end{align*}
    where the last equality follows from the fact that $J$ is a derivation.  If
    $\phi$ is $\hat{\mfp}$-invariant then the last line is zero, so $f$ is a
    cocycle. That $f$ is $\mfg_0$-invariant is clear from the
    $\hat{\mfp}$-invariance and the fact that $J$ annihilates $\mfg_0$. 
\end{proof}

\begin{defn}
    A linear function $f : \hat{\mfp} \arr \hat{\mfu}$ defines a contraction
    operator 
    \begin{equation*}
        \iota(f) : \bigwedge^p \hat{\mfu}^* \otimes S^q \hat{\mfp}^* \arr
            \bigwedge^{p-1} \hat{\mfu}^* \otimes S^{q+1} \hat{\mfp}^*.
    \end{equation*}
    A cochain $\omega \in \bigwedge^* \hat{\mfu}^* \otimes S^* \hat{\mfp}^*$ is
    \emph{$\hat{\mfu}$-basic} if $\iota(f) \omega = 0$ for all $f$ of the form $y
    \mapsto [x,y]$, $x \in \hat{\mfu}$.
\end{defn}

The main theorem of this section allows us to identify
$H_{cts}^*(\hat{\mfp},\mfg_0; S^* \hat{\mfp}^*)$ with the ring of
$\hat{\mfu}$-basic $\hat{\mfp}$-invariant cochains.
\begin{thm}\label{T:basicforms}
    Let $\mfp$ be a parahoric in a twisted loop algebra $\mfg$ and let
    $\mfu$ be the nilpotent subalgebra.  Then there is a positive-definite
    Hermitian form on $C^* = C^*(\hat{\mfp},\mfg_0; S^*\hat{\mfp}^*)$ and a
    derivation $J$ of $\mfp$ such that the harmonic forms in $C^*$ are closed
    under multiplication, and furthermore the map in Lemma \ref{L:cocycle}
    gives an isomorphism between the ring of $\hat{\mfu}$-basic
    $\hat{\mfp}$-invariant elements of $\bigwedge^* \hat{\mfu}^* \otimes S^*
    \hat{\mfp}^*$ and the ring of harmonic forms. 
\end{thm}
Before proceeding to the proof, we note that Theorem \ref{T:basicforms} can be
rephrased in a geometric manner.  Let $\mcP$ and $\mcN$ be pro-Lie groups with
Lie algebras $\hat{\mfp}$ and $\hat{\mfu}$ respectively. The space $\mfp /
\bigoplus_{n > k} \mfg_n$ has the structure of an affine variety, so the
pro-algebra $\hat{\mfp}$ can be regarded as a scheme with coordinate ring
$S^* \hat{\mfp}^*$. 
\begin{defn} 
    The tangent space $T \hat{\mfp}$ is isomorphic to $\hat{\mfp} \times
    \hat{\mfp}$. Let $T_{>0} \hat{\mfp}$ denote the subbundle of $T \hat{\mfp}$
    isomorphic to $\hat{\mfp} \times \hat{\mfu}$, and $T_{>0}^* \hat{\mfp}$
    the continuous dual bundle of $T_{>0}$. Let $\Omega_{>0}^* \hat{\mfp}$
    denote the ring of global sections of $T^*_{>0} \hat{\mfp}$.

    The bundle $T_{>0} \hat{\mfp}$ contains all tangents to $\mcN$-orbits. We
    will say that an element of $\Omega_{>0}^* \hat{\mfp}$ is
    \emph{$\mcN$-basic} if it vanishes on all tangents to $\mcN$-orbits.
\end{defn}
With this terminology, we can identify the ring of $\hat{\mfu}$-basic
$\hat{\mfp}$-invariant cochains with the ring of $\mcP$-invariant $\mcN$-basic
elements of $\Omega_{>0}^* \hat{\mfp}$.

Although Theorem \ref{T:basicforms} covers the main case of interest, a more
natural result occurs if $S^* \hat{\mfp}^*$ is replaced with $S^*
\hat{\mfu}^*$. An element $\omega$ of $\bigwedge^* \hat{\mfu}^* \otimes S^*
\hat{\mfu}^*$ is $\hat{\mfp}$-basic if $\iota(f) \omega = 0$ for all linear
endomorphisms $f$ of $\hat{\mfu}$ of the form $y \mapsto [x,y]$, $x \in
\hat{\mfp}$.
\begin{theorem}\label{T:basicforms2}
    Let $\mfp$ be a parahoric in a twisted loop algebra $\mfg$, and let $\mfu$
    be the nilpotent subalgebra. Then there is a positive-definite Hermitian
    form on $C^*(\hat{\mfp},\mfg_0; S^* \hat{\mfu}^*)$ and a derivation $J$
    (the same as in Theorem \ref{T:basicforms}) such that the harmonic
    forms are closed under multiplication, and furthermore the map of 
    Lemma \ref{L:cocycle} gives an isomorphism between the ring of
    $\hat{\mfp}$-basic and invariant elements of $\bigwedge^* \hat{\mfu}^*
    \otimes S^* \hat{\mfu}^*$ and the ring of harmonic forms.
\end{theorem}
In geometric language, the ring of $\hat{\mfp}$-basic and invariant cochains is
the same as the ring of $\mcP$-basic and invariant algebraic forms on
$\hat{\mfu}$. 

For the proofs of Theorems \ref{T:basicforms} and \ref{T:basicforms2}, we
assume that the underlying Lie algebra $L$ is semisimple, as this simplifies
our Kahler metric construction.  If $L$ is reductive then $L = [L,L] \oplus
\mfz$, where $\mfz$ is the centre, and consequently $\mfg = [\mfg,\mfg] \oplus
\mfz[z^{\pm 1}]$. It is easy to deduce the reductive case of Theorems \ref{T:basicforms} 
and \ref{T:basicforms2} from the semisimple case by splitting off the centre (for
instance we can extend $J$ to be the identity on $\mfz[z]$). The proof we give
actually holds in more generality.  Let $\mfg$ be a $\Z$-graded Lie algebra
(such that $\dim \mfg_n < +\infty$ for all $n$) with a conjugation (an
anti-linear automorphism sending $\mfg_n \mapsto \mfg_{-n}$) and a
contragradient positive-definite Hermitian form (satisfying $\mfg_n \perp
\mfg_m$ for $m \neq n$). Let $\mfb = \bigoplus_{n \geq 0} \mfg_n$ and $\mfu =
\bigoplus_{n > 0} \mfg_n$. Notice that $\overline{\mfu} = \mfg / \mfb$ is an
$\mfb$-module. If $\overline{\mfu}^{\mfb} = 0$\footnote{This is the condition
that does not hold if $L$ has centre.} then Theorems \ref{T:basicforms} and
\ref{T:basicforms2} hold for $C^*(\hat{\mfb},\mfg_0; S^* \hat{\mfb}^*)$
and $C^*(\hat{\mfb},\mfg_0; S^* \hat{\mfu}^*)$.

\subsection{Nakano's identity and the semi-infinite chain complex}

For the purposes of this section, let $\mfg$ be a $\Z$-graded Lie algebra with
a conjugation, as in the last paragraph of the previous section (this time a
contragradient metric is not required). Let $\mfb$ and $\mfu$ be the
subalgebras $\bigoplus_{n \geq 0} \mfg_n$ and $\bigoplus_{n > 0} \mfg_n$
respectively. These working assumptions are based on the standard conventions
in semi-infinite cohomology, see e.g. \cite{FGZ86}.  In this section we state a
version of Nakano's identity for the relative cohomology of $(\mfb,\mfg_0)$.
This version of Nakano's identity is a straight-forward generalization of a
version for loop groups due to Teleman \cite{Te95}.

Let $V$ be a locally finite $\mfb$-module, such that the action of $\mfg_0$
extends to an action of the complex conjugate $\bar{\mfb}$. Both the $\mfb$ and
the $\bar{\mfb}$ action will be denoted by $\pi$. The relative semi-infinite
chain complex with coefficients in $V$ is a bicomplex $C^{*,*}(V)$ defined by
\begin{equation*}
    C^{-p,q}(V) = \left(\bigwedge^{q} \hat{\mfu}^* \otimes \bigwedge^p \overline{\mfu}
            \otimes V\right)^{\mfg_0}.
\end{equation*}
There are truncated actions of $\mfu$ on $\overline{\mfu} = \mfg /
\overline{\mfb}$ and of $\overline{\mfu}$ on $\mfu = \mfg / \mfb$. Both will be
denoted by $\widetilde{\ad}$. The bicomplex $C^{*,*}$ has two differentials
$\deltabar$ and $D$, of degrees $(0,1)$ and $(1,0)$ respectively. $\deltabar$
is the differential for the Lie algebra cohomology of $(\mfb,\mfg_0)$ with
coefficients in $\bigwedge^* \overline{\mfu} \otimes V$, and can be explicitly
defined as
\begin{equation*}
    \sum_{k >0} \eps(z^k) \left(\frac{1}{2} \ad^t(z_k) + \widetilde{\ad}^t(z_k)
            + \pi(z_k) \right),
\end{equation*}
where $\{z_k\}_{k \geq 1}$ is a homogeneous basis of $\mfu$ as before and
$\eps$ is exterior multiplication. $D$ is the differential for the Lie algebra
homology of $\overline{\mfu}$ with coefficients in $\bigwedge^q \hat{\mfu}^* \otimes
V$, restricted to the $\mfg_0 = \overline{\mfb} / \overline{\mfu}$ invariants.
$D$ can be explicitly defined as
\begin{equation*}
    \sum_{k <0} \left(\frac{1}{2} \ad(z_k) + \widetilde{\ad}(z_k) + \pi(z_k) \right)
        \iota(z^k),
\end{equation*}
where $z_{-k} = \overline{z_k}$ and $\iota$ is the contraction operator on
$\bigwedge^* \overline{\mfu}$. Note that $\iota$ is extended to $C^{*,*}(V)$ so
as to respect super-commutativity, so $\iota(f) \alpha \otimes \beta \otimes v
= (-1)^{q} \alpha\otimes \iota(f) \beta \otimes v$ for $\alpha \in \bigwedge^q
\hat{\mfu}^*$, $\beta \in \bigwedge^* \overline{\mfu}$, and $v \in V$. 

Let $d$ be the total differential $d = \deltabar + D$. A standard fact from
semi-infinite cohomology is that $d^2 = \eps(\gamma)$, where $\gamma$ is the
the semi-infinite cocycle defined by $\gamma|_{\mfg_m \times \mfg_n} = 0$ if
$m+n \neq 0$ or $m=n=0$, and otherwise by
\begin{equation*}
    \gamma(x,y) = \sum_{0 \leq n < k} \tr_{\mfg_n} (\ad(x) \ad(y))
\end{equation*}
for $x \in \mfg_k$, $y \in \mfg_{-k}$, and $k>0$.  Since $\gamma$ has type
$(1,1)$, the operator $\eps(\gamma)$ on $C^{*,*}$ should be interpreted in the
semi-infinite sense, wherein $\eps(z^k \wedge z^{-l}) = \eps(z^k)
\iota(z^{-l})$, for $k,l > 0$.

\begin{defn}
    A \emph{Kahler metric} for the pair $(\mfb,\mfg_0)$ is a Hermitian form
    $(,)$ on $\mfg$ such that
    \begin{itemize}
        \item $(,)$ is positive-definite on $\mfu$ and zero on $\mfg_0$,
        \item $([x,y],z) = -(y,[\overline{x},z])$ for all $x \in \mfg_0$, and
        \item the fundamental form $\omega \in \mfu^* \otimes \overline{\mfu}^*
            \subset \bigwedge^2 \mfg^*$ defined by $\omega(a,b) =
            -i(a,\overline{b})$ for $a \in \mfu$, $b \in \overline{\mfu}$
            is a cocycle.
    \end{itemize}
\end{defn}
Note that we can define a Kahler metric by giving the the restricted Hermitian
form on $\mfu$ and then extending to $\mfg$ by zero on $\mfg_0$ and by
$(\overline{a},\overline{b}) = \overline{(a,b)}$ for $a,b \in \mfu$.

Suppose that there is a Kahler metric for the $(\mfb,\mfg_0)$. Let $L$ denote
multiplication by the fundamental form $\omega$, defined explicitly as
\begin{equation*}
    L = - i \sum_{k \geq 1} \eps(z^k) \iota(z^{-k}),
\end{equation*}
where the basis $\{z_k\}$ is now required to be orthonormal in the Kahler
metric. Let $\Lambda = L^*$ be the adjoint of $L$ on the complex $C^{*,*}(\C)$
with trivial coefficients, extended by $\otimes \Id$ on $C^{*,*}(V)$. Letting
$H = [\Lambda, L]$, it is not hard to check that $\{H,\Lambda,L\}$ is an
$\mfsl_2$-triple, and that $H$ acts on $C^{-p,q}$ by $p-q$ (in other words, if
the degree of $C^{-p,q}$ is defined to be $q-p$ then $H$ acts by $-\deg$). 
This $\mfsl_2$-action is used in Hodge theory to prove Nakano's identity.
Teleman adapted this proof to give an algebraic version of Nakano's identity
for the loop algebra \cite{Te95}. More generally, the same proof gives:
\begin{prop}[Nakano's identity]\label{P:nakano}
    Suppose there is a Kahler metric for $(\mfb,\mfg_0)$, and $V$ has a
    contragradient positive-definite Hermitian form.  Then in the induced
    metric on $C^{*,*}(V)$ we have
    \begin{equation*}
        \boxbar = \square + i [\eps(\gamma + \Theta), \Lambda],
    \end{equation*}
    where $\boxbar$ is the $\deltabar$-Laplacian, $\square$ is the
    $D$-Laplacian, $\gamma$ is the semi-infinite cocycle, and $\Theta$
    is the curvature form
    \begin{equation*}
        \Theta = \sum_{i,j\geq 1} z^{-i} \wedge z^{j} \left( [\pi(z_{-i}),\pi(z_j)] 
                        - \pi([z_{-i},z_j]) \right).
    \end{equation*}
\end{prop}
On restricting to $p=0$, the complex $(C^{0,q}(V),\deltabar)$ becomes the
Koszul complex for the Lie algebra cohomology of the pair $(\mfb,\mfg_0)$ with
coefficients in $V$. The curvature term $i[\eps(\Theta),\Lambda]$ is
straight-forwardly shown to be 
\begin{equation}\label{E:curvature}
    -\sum_{i,j \geq 1} \eps(z^i) \iota(z_j) \left( [\pi(z_i),\pi(z_{-j})]
        - \pi([z_i,z_{-j}]) \right)
\end{equation}
on $C^{0,q}(V)$, where $\{z_i\}$ is a homogeneous basis orthonormal in the
Kahler metric.

\subsection{Kahler metrics for parahorics and the derivation $J$} 
Now we return to the case of a parahoric $\mfp$ in a twisted loop algebra
$\mfg$.  Recall that $\mfp = \bigoplus_{n \geq 0} \mfg_n$ for some grading of
type $d$, and that for the purposes of the proof we are assuming that $L$ is
semisimple. Consequently there is a Kac-Moody algebra $\tilde{\mfg}$ associated
to $\mfg$, and this Kac-Moody algebra has a standard non-degenerate invariant
symmetric bilinear form $\langle,\rangle$. The contragradient Hermitian form
$\{,\}$ on $\mfg$ defines a symmetric invariant bilinear form
$\{\cdot,\overline{\cdot}\}$, and this symmetric form extends to a scalar
multiple of the standard invariant form on $\tilde{\mfg}$. The twisted loop
algebra $\mfg$ is also graded by the root lattice of the Kac-Moody algebra
associated to $\mfg$. Let $\rho$ be a weight of the Kac-Moody defined on simple
coroots by $\rho(\alpha_i^{\vee}) = 0$ if $d_i=0$ and $\rho(\alpha_i^{\vee}) =
1$ if $d_i > 0$ (note that the $\alpha_i^{\vee}$'s are coroots of the
associated Kac-Moody, not of the twisted loop algebra $\mfg$).  Let $J$ be the
derivation of $\mfp$ acting on root spaces $\mfg_\alpha$ as multiplication by
$2 \langle \rho, \alpha \rangle$. 

\begin{prop}\label{P:kahler}
    Let $\{,\}$ be the contragradient positive definite Hermitian form on
    $\mfg$, normalized to match the standard invariant form on the associated
    Kac-Moody.  Then $J$ is positive-definite and $(\cdot,\cdot) =
    \{J\cdot,\cdot\} = \{\cdot,J \cdot\}$ is a Kahler metric for
    $(\mfp,\mfg_0)$ with fundamental form $i \gamma$, where $\gamma$ is the
    semi-infinite cocycle.
\end{prop}
\begin{proof}
    The only thing to prove is that $(,)$ has fundamental form $i \gamma$.
    We go about the proof somewhat backwards: let $\widetilde{\ad}_{\mfp}$
    denote the truncated action of $\mfp \oplus \overline{\mfp}$ on $\mfp =
    \mfg  / \overline{\mfu}$, and define
    \begin{equation*}
        J' = \sum_{k \geq 1} \widetilde{\ad}_{\mfp}(x_k) \widetilde{\ad}_{\mfp} (x_k)^*,
    \end{equation*}
    where $\{x_k\}_{k \geq 1}$ is a homogeneous basis for $\mfu$, orthonormal
    in the contragradient metric. Define $(\cdot,\cdot)' = \{J'\cdot,\cdot\}$.
    Then $J$ is positive-semidefinite by definition, so $(,)'$ is a
    positive-semidefinite Hermitian form. Suppose $a \in \mfg_n$, $b \in
    \mfg_{-n'}$, $n,n' \geq 0$, and assume without loss of generality that
    $x_1,\ldots,x_m$ is a basis of $\mfg_1 \oplus \ldots \oplus \mfg_n$. Since
    $\widetilde{\ad}_{\mfp}(x_k)^* = -\widetilde{\ad}_{\mfp}(\overline{x_k})$
    we have
    \begin{align*}
        (a,\overline{b})' & = \{ J' a, \overline{b} \}
                          = \sum_{k=1}^m \{ [\overline{x_k},a],
                            [\overline{x_k},\overline{b}]\} \\
                         & =  -\sum_{k=1}^m \{ [b,[a,\overline{x_k}]], \overline{x_k}\}
                          = -\sum_{l=1}^n \tr_{\mfg_{-l}} (\ad(b) \ad(a)).
    \end{align*}
    Now $\tr_{\mfg_{-l}}(\ad(b)\ad(a)) = \tr_{\mfg_{n-l}}(\ad(a)\ad(b))$, so
    $-i(a,\overline{b})' = i \gamma(a,b)$. Since $\gamma$ is a cocycle and $\{,\}$
    is contragradient, it follows that $J'$ is a derivation. (It is possible to 
    show that $J'$ is a derivation directly but it takes a little work). 

    Now $\mfp$ is generated by $\mfg_0$ and the root vectors $e_i$ with $d_i >
    0$. $J'$ annihilates $\mfg_0$, and if $d_i > 0$ then 
    \begin{equation*}
        J' e_i = \frac{[e_i,[f_i,e_i]]}{\{e_i,e_i\}} = 2 \langle \rho,\alpha_i \rangle e_i,
    \end{equation*}
    where $f_i = - \overline{e_i}$.  It follows that $J' = J$, finishing the
    proof.
\end{proof}

More generally, if $\mfg$ is a $\Z$-graded Lie algebra with conjugation and a
contragradient positive-definite Hermitian form, then we can define a Kahler
metric simply by using the operator $J'$ form the proof of Proposition \ref{P:kahler}.
The hypothesis $\overline{\mfu}^{\mfb} = 0$ is needed to ensure that the metric
is positive-definite.

\subsection{Calculation of the curvature term}

If $S$ is a linear operator $\hat{\mfu}^* \arr \hat{\mfp}^*$, define an
operator $d_R(S)$ on $\bigwedge^* \hat{\mfu}^* \otimes S^* \hat{\mfp}^*$ by
\begin{equation*}
    d_R(S) (\alpha_1 \wedge \ldots \wedge \alpha_k \otimes b) = 
        \sum_i (-1)^{i-1} \alpha_1 \wedge \ldots \hat{\alpha}_i \ldots
            \wedge \alpha_k \otimes S(\alpha_i) \circ b.
\end{equation*}
If $T$ is an operator $\hat{\mfp}^* \arr \hat{\mfu}^*$, define a similar
operator $d_L(T)$ by
\begin{equation*}
    d_L(T) (\alpha \otimes b_1 \circ \cdots \circ \cdots b_l) = \sum_i T(b_i)
        \wedge \alpha \otimes b_1 \circ \cdots \hat{b}_i \cdots \circ b_l
\end{equation*}
Recall that truncated actions are denoted by $\widetilde{\ad}$, with subscripts
denoting the appropriate truncated space.  By abuse of notation, let $J^{-1}$
denote the inverse of the restriction of the derivation of Proposition
\ref{P:kahler} to $\mfu$. We will also use $J^{-1}$ to denote the dual operator
on $\hat{\mfu}^*$.
\begin{prop}\label{P:curvature}
    Let $\mfp$ be a parahoric subalgebra of a twisted loop algebra $\mfg$.  Let
    $V = S^* \hat{\mfp}^*$ with the contragradient metric.  The Laplacian on
    $C^*(V)$ with respect to the dual Kahler metric from Proposition
    \ref{P:kahler} has curvature term
    \begin{equation*}
        i[\eps(\gamma + \Theta),\Lambda] = \sum_{i > 0} 
            d_R\left(\widetilde{\ad}^t_{\mfp}(x_i) J^{-1}\right)^*
            d_R\left(\widetilde{\ad}^t_{\mfp}(x_i) J^{-1}\right),
    \end{equation*}
    where $\{x_i\}$ is a basis for $\mfu$ orthonormal in the contragradient
    metric, and $\hat{\mfu}^*$ is considered as the subset of $\hat{\mfp}^*$
    that is zero on $\mfg_0$ (so that $\widetilde{\ad}^t_{\mfp}$ sends
    $\hat{\mfu}^*$ to $\hat{\mfp}^*$). 
\end{prop}
\begin{proof}
    Let $R = i[\eps(\Theta),\Lambda]$. The action $\widetilde{\ad}^t$ acts as a
    derivation on the symmetric algebra $S^* \hat{\mfp}^*$, so by Equation
    (\ref{E:curvature}), $R$ is a second-order operator.  This means that if
    $\alpha_0,\ldots,\alpha_k \in \hat{\mfu}^*$, $b_0,\ldots,b_l \in
    \hat{\mfp}^*$ then
    \begin{equation*}
        R(\alpha_0 \wedge \cdots \wedge \alpha_k \otimes b_0 \circ \cdots \circ b_l)
            = \sum_{i,j} (-1)^i R(\alpha_i \otimes b_j) \alpha_0 \cdots
            \hat{\alpha_ i} \cdots \hat{b_j} \cdots b_l.
    \end{equation*}
    In particular, $R$ is determined by its action on $\hat{\mfu}^* \otimes
    \hat{\mfp}^*$. 

    The truncated action on $V$ is isomorphic to the truncated action on $V' =
    S^* \overline{\mfp}$ via the contragradient metric. Let $R' =
    i[\eps(\Theta_{V'}),\Lambda]$.  If $f \in \hat{\mfu}^*$ and $w \in
    \overline{\mfp}$ then we claim that
    \begin{equation*}
        R'(f \otimes w) = \sum_{i > 0} \widetilde{\ad}^t_{\mfu}(w) z^i \otimes 
            \widetilde{\ad}_{\overline{\mfp}}(z_i) \phi^{-1}(f),
    \end{equation*}
    where $\{z_i\}$ is any homogeneous basis of $\mfu$, $\phi$ is the
    isomorphism $\overline{\mfu} \arr \hat{\mfu}^*$ induced by the Kahler
    metric, and $\overline{\mfu}$ is considered as a subset of
    $\overline{\mfp}$, so that $\widetilde{\ad}_{\overline{\mfp}}$ maps from
    $\overline{\mfu}$ to $\overline{\mfp}$. To prove this, let $\{z_i\}$ be
    orthonormal with respect to the Kahler metric, and think about $f = z^k$,
    $w$ arbitrary. Observe that
    \begin{equation*}
        \widetilde{\ad}_{\overline{\mfp}} (z) w = \sum_{s \geq 1} y^{-s}([z,w]) y_{-s},
    \end{equation*}
    where $\{y_s\}_{s\geq 1}$ is a homogeneous basis of $\mfp$ and $y_{-s} =
    \overline{y_s}$.
    So if $z_{-j} \in \mfg_{-m}$, then 
    \begin{equation*}
        \widetilde{\ad}_{\overline{\mfp}}(z_i) \widetilde{\ad}_{\overline{\mfp}}(z_{-j}) w 
            = \sum_{s \geq 1} y^{-s}([z_i,[z_{-j}, w]]) y_{-s},
    \end{equation*}
    \begin{align*}
        \widetilde{\ad}_{\overline{\mfp}} (z_{-j}) \widetilde{\ad}_{\overline{\mfp}} (z_i) w & 
            = \sum_{s \geq 1} y^{-s}([z_i,w]) [z_{-j},y_{-s}] \\
            & = \sum_{n \leq 0} \sum_{y_{-s} \in \mfg_{n-m}} 
                y^{-s}([z_{-j},[z_i,w]]) y_{-s}, \text{ and}
    \end{align*}
    \begin{equation*}
        \widetilde{\ad}_{\overline{\mfp}}([z_i,z_{-j}]) w 
            = \sum_{s \geq 1} y^{-s}([[z_i,z_{-j}],w]) y_{-s}.
    \end{equation*}
    Consequently
    \begin{equation*}
        \left(\left[\widetilde{\ad}_{\overline{\mfp}}(z_i),\widetilde{\ad}_{\overline{\mfp}}(z_{-j})\right]
             - \widetilde{\ad}_{\overline{\mfp}}([z_i,z_{-j}])\right)w 
            = \sum_{-m < n \leq 0} \sum_{y_{-s} \in \mfg_n} y^{-s} ([z_{-j},[z_i,w]]) y_{-s}.
    \end{equation*}
    After removing the reference to $m$ here, we get
    \begin{equation*}
        \left(\left[\widetilde{\ad}_{\overline{\mfp}}(z_i),\widetilde{\ad}_{\overline{\mfp}}(z_{-j})\right] 
            - \widetilde{\ad}_{\overline{\mfp}}([z_i,z_{-j}])\right)w = 
            \sum_{s \geq 1} \sum_{l \geq 1} y^{-s}([z_{-j},z_l]) z^l([z_i,w]) y_{-s}.
    \end{equation*}
    Now from Equation (\ref{E:curvature}),
    \begin{align*}
        R' (z^k \otimes w) & 
            = -\sum_{i>0} z^i \otimes\left(\left[\widetilde{\ad}_{\overline{\mfp}}(z_i),
                            \widetilde{\ad}_{\overline{\mfp}}(z_{-k})\right] 
                            - \widetilde{\ad}_{\overline{\mfp}}([z_i,z_{-k}])\right) w \\
        & = -\sum_{i,l,s > 0} z^i \otimes y^{-s}([z_{-k},z_l]) z^l([z_i,w]) y_{-s}.
    \end{align*}            
    Moving the $w$ action from $z_i$ to $z^i$, the last expression becomes
    \begin{equation*}
        - \sum_{i,s>0} \widetilde{\ad}^t_{\mfu}(w) z^i \otimes y^{-s}([z_{-k},z_i]) y_{-s}
            = \sum_{i>0} \widetilde{\ad}^t_{\mfu}(w) z^i \otimes
                    \widetilde{\ad}_{\overline{\mfp}}(z_i) z_{-k}.
    \end{equation*}
    The proof of the claim is finished by noting that this last expression is
    independent of the choice of basis $\{z_i\}$ for $\mfu$ and that that
    $z_{-k} = \phi^{-1}(z^k)$.

    Now we can translate from $V'$ to $V$ using the isomorphism $\psi :
    \overline{\mfp} \arr \hat{\mfp}^*$ induced by the contragradient form. 
    The operator $J$ on $\mfu$ has a basis $\{x_i\}$ of eigenvectors
    orthonormal in the contragradient metric.  If $J x_i = \lambda_i x_i$ then
    $\phi(\overline{x_i}) = \lambda_i x^i$, and thus $\psi \circ \phi^{-1}
    (x^i) = \lambda_i^{-1} x^i$. It follows that $\psi \circ \phi^{-1} =
    J^{-1}$ on $\hat{\mfu}^*$. Next, $\widetilde{\ad}^t_{\mfu}(w) \psi(x) = -
    \widetilde{\ad}^t_{\mfp}(x) \psi(w)$. Since $\psi(\overline{x_i}) = x^i$ we
    can conclude that
    \begin{equation*}
        R(f \otimes g) = -\sum_{i > 0} \widetilde{\ad}^t_{\mfp} (\overline{x_i}) g
            \otimes \widetilde{\ad}^t_{\mfp}(x_i) J^{-1} f, 
    \end{equation*}
    where  $\widetilde{\ad}^t_{\mfp}(\overline{x_i})$ is regarded as a map from
    $\hat{\mfp}^*$ to $\hat{\mfu}^*$. 

    If $S,T \in \End(\hat{\mfp}^*)$, let $\Switch(S,T)$ be the second order
    operator on $\bigwedge^* \hat{\mfp}^* \otimes S^* \hat{\mfp}^*$ sending
    $\alpha \otimes \beta \mapsto T \beta \otimes S \alpha$. Note that
    $\widetilde{\ad}^t_{\mfp}(x_i)^* = - \widetilde{\ad}^t_{\mfp}
    (\overline{x_i})$. We have shown that $R$ is the restriction of the 
    operator
    \begin{equation*}
        \sum_i \Switch\left(\widetilde{\ad}^t_{\mfp}(x_i) J^{-1}, \widetilde{\ad}^t_{\mfp}(x_i)^*\right)
    \end{equation*}
    to $\bigwedge^* \hat{\mfu}^* \otimes S^* \hat{\mfp}^*$, where $J^{-1}$ is
    zero on $\mfg_0$. It is easy to see that $\Switch(S,T) = d_L(T) d_R(S)
    - (T S)^{\wedge}$, where $(TS)^{\wedge}$ is the operator $TS$ extended to
    $\bigwedge^* \hat{\mfp}^*$ as a derivation. Also, $d_L(T)^* = d_R(T^*
    J^{-1})$, where $T^*$ is the adjoint of $T$ in the contragradient metric.
    Note that the $J^{-1}$ term comes from the difference between the
    contragradient metric on the symmetric factor and the Kahler metric on
    the exterior factor. Finally we have
    \begin{equation*}
        R = \sum_i d_R\left(\widetilde{\ad}^t_{\mfp}(x_i) J^{-1}\right)^* 
                        d_R\left(\widetilde{\ad}^t_{\mfp}(x_i) J^{-1}\right)
            + \sum_i \left(\widetilde{\ad}^t_{\mfp}(\overline{x_i}) \widetilde{\ad}^t_{\mfp}(x_i)
                J^{-1} \right)^{\wedge}.
    \end{equation*}
    Now $\sum_i \widetilde{\ad}^t_{\mfp}(\overline{x_i}) \widetilde{\ad}^t_{\mfp}(x_i)$
    is the negative of the dual of the derivation $J$ on $\mfu$, while $J^{-1}$ is
    the dual of the inverse of $J$. Thus on $\bigwedge^* \hat{\mfu}^*$, this
    second summand is simply $-\deg$. But since we have chosen a Kahler metric
    with fundamental form $i \gamma$, we have $i[\eps(\gamma),\Lambda] =
    [L,\Lambda] = -H = \deg$, finishing the proof of the Proposition.
\end{proof}

Similarly, given endomorphisms $S$ and $T$ of $\hat{\mfu}^*$ we can define
operators $d_R(S)$ and $d_L(T)$ on $\bigwedge^* \hat{\mfu}^* \otimes S^*
\hat{\mfu}^*$.
\begin{prop}\label{P:curvature2}
    Let $\mfp$ be a parahoric subalgebra of a twisted loop algebra $\mfg$, and
    let $\mfu$ be the nilpotent subalgebra.  Let $V = S^* \hat{\mfu}^*$ with
    the contragradient metric.  The Laplacian on $C^*(V)$ with respect to the
    dual Kahler metric from Proposition \ref{P:kahler} has curvature term
    \begin{equation*}
        i[\eps(\gamma + \Theta),\Lambda] = \sum_{i \geq 0} 
            d_R\left(\widetilde{\ad}^t_{\mfu}(y_i) J^{-1}\right)^*
            d_R\left(\widetilde{\ad}^t_{\mfu}(y_i) J^{-1}\right),
    \end{equation*}
    where $\{y_i\}_{i \geq 0}$ is a basis for $\mfp$ orthonormal in the contragradient
    metric. 
\end{prop}
The proof of Proposition \ref{P:curvature2} is similar to the proof of Proposition
\ref{P:curvature}. A proof of the analogous result for symmetrizable Kac-Moody
algebras with the principal grading can be found in \cite{Sl10}.

\subsection{Proof of Theorems \ref{T:basicforms} and \ref{T:basicforms2}}
Once again let $J$ denote the operator on $\hat{\mfu}^*$ which is the dual of
the derivation $J$ on $\mfu$. We give a proof of Theorem \ref{T:basicforms};
the proof of Theorem \ref{T:basicforms2} is identical. Let $J_{\Delta}$ denote
the diagonal extension of $J$ to the exterior factor of $\bigwedge^*
\hat{\mfu}^* \otimes S^* \hat{\mfp}^*$. The adjoint of
$\widetilde{\ad}_{\mfu}^t(x)$ in the Kahler metric is $- J
\widetilde{\ad}_{\mfu}(\overline{x}) J^{-1}$. Thus we can directly calculate
that
\begin{equation*}
    D^* = - \sum_{i < 0} \eps(x_i) \left( (J \widetilde{\ad}_{\mfu}^t(x_{-i}) 
        J^{-1})^{\wedge} + \widetilde{\ad}_{\mfp}^t(x_{-i})^{Sym} \right),
\end{equation*}
where $\{x_i\}$ is a basis of $\mfu$ orthonormal in the contragradient metric
and $x_{-i} = \overline{x_i}$. On $C^{0,q}(V)$ the $D$-Laplacian is $\square = D
D^*$, so the set of harmonic cocycles is the joint kernel of the operators
$D^*$ above and $d_R\left(\widetilde{\ad}_{\mfp}^t(x_i) J^{-1}\right)$, $i\geq 1$. The
kernel of $D^*$ on $C^{0,q}(V)$ is the joint kernel of the operators $\left(J
\widetilde{\ad}_{\mfu}^t(x_{i}) J^{-1}\right)^{\wedge} +
\widetilde{\ad}_{\mfp}^t(x_i)^{Sym}$, $i \geq 1$. Now we have
\begin{align*}
    d_R\left(\widetilde{\ad}_{\mfp}^t(x_i) J^{-1}\right) J_{\Delta} & = 
            J_{\Delta} d_R\left(\widetilde{\ad}_{\mfp}^t(x_i)\right) \text{ and} \\
    \left(\left(J \widetilde{\ad}_{\mfu}^t(x_{i}) J^{-1}\right)^{\wedge} +
        \widetilde{\ad}_{\mfp}^t(x_i)^{Sym} \right) J_{\Delta}
        & = J_{\Delta} \left( \widetilde{\ad}_{\mfu}^t(x_i)^{\wedge} +
            \widetilde{\ad}^t_{\mfp}(x_i)^{Sym} \right).
\end{align*}
Thus we see that $J_{\Delta}^{-1}$ identifies the set of harmonic cocycles
with the joint kernels of the operators $d_R\left(\widetilde{\ad}_{\mfp}^t(x_i)\right)$,
$i \geq 1$, and $\left( \widetilde{\ad}_{\mfu}^t(x_i)^{\wedge} +
\widetilde{\ad}^t_{\mfp}(x_i)^{Sym} \right)$, $i \geq 1$. Since the elements
of $C^{0,q}(V)$ are $\mfg_0$-invariant by definition, the kernel of the 
latter family of operators is the set of $\mfp$-invariant cochains. The kernel
of the former family of operators is the set of $\mfu$-basic cochains,
finishing the proof.

\section{Twisted arc and jet schemes and the twisted arc group}\label{S:jetschemes}

This section covers background material on twisted arc and jet schemes, and
proves basic facts about the twisted arc group arising from a diagram
automorphism. The material in this section will be used to study adjoint orbits
in Section \ref{S:orbit}.

\subsection{Twisted arc and jet schemes}\label{SS:arcspace}
By a variety, we mean a separated, reduced, but not necessarily irreducible,
scheme of finite type over $\C$. The arc scheme $J_\infty X$ of a variety $X$ over
$\C$ is a separated scheme of infinite type representing the functor $Y \mapsto
\Hom(Y \times \Spec \C[[z]], X)$. Intuitively the arc scheme is the space of
maps from the formal arc $\Spec \C[[z]]$ into $X$. The $m$th jet scheme $J_m X$
($0 \leq m < +\infty$) is a separated scheme of finite type over $\C$
representing the functor $Y \mapsto \Hom(Y \times \Spec \C[z]/z^m,X)$.  If $m
\leq n$ then there is a morphism $J_n X \arr J_m X$, and $J_{\infty} X$ is the
inverse limit of the jet schemes of $X$.  The $\C$-points of $J_m X$ are
$m$-jets, ie. morphisms $\Spec \C[z] / z^m \arr X$.  For example, $J_0 X = X$
and $J_1 X$ is the tangent scheme of $X$. If $X$ is the affine subset of $\C^n$
cut out by the equations $f_1 = \ldots = f_k = 0$ then $J_m X$ is the subscheme
of $(\C[z]/z^m)^{n}$ cut out by the equations $f_i(x_1,\ldots,x_n) = 0$,
$i=1,\ldots,f_k$, where $x_i \in \C[z]/z^m$ and $(\C[z]/z^m)^{n}$ is regarded
as the affine space of dimension $mn$. The association $V \mapsto J_m V$ is
functorial, so if $G$ is an algebraic group then $J_m G$ is an algebraic group
when $m < +\infty$, and a pro-group when $m = +\infty$. The arc scheme of $X$
is sometimes denoted by $X[[z]]$, but we use the notation $J_{\infty} X$ so
that propositions can be stated uniformly for both arc and jet schemes.

The following well-known lemma is useful for working with jet schemes:
\begin{lemma}[\cite{Mu01}]\label{L:etalepullback} 
    If $X \arr Y$ is etale then $J_m X = X \times_Y J_m Y$ for all $0 \leq m
    \leq +\infty$.
\end{lemma}
Open immersions are etale, so if $U \subset X$ is an open subset then the
pullback of $U$ via $J_m X \arr X$ is equal to $J_m U$. In particular, $J_m X$
is covered by open subsets $J_m U$, where $U \subset X$ is an affine open
(if $m = +\infty$ then we use the fact that the inverse limit of affine
schemes is affine).  If there is an etale map $X \arr \mbA^n$ then $J_m X = X
\times \mbA^{mn}$ for all $m < +\infty$, while $J_{\infty} X = X \times
\mbA^{\infty}$. Consequently if $X$ is smooth of dimension $n$ then $J_m X \arr
J_k X$ is a (Zariski) locally-trivial $\mbA^{n(m-k)}$-bundle for all $k \leq m
< +\infty$.  In particular, $J_m X$ is smooth and the truncation morphisms $J_m
X \arr J_k X$ are surjective. Similarly $J_{\infty} X \arr J_k X$ is a
locally-trivial $\mbA^{\infty}$-bundle for all $0 \leq k <
+\infty$.\footnote{The infinite-type schemes we work with are nice enough that
they could be called ``smooth'' in their own right. However, we avoid this
complication and only use smoothness for schemes of finite type. See for
instance the wording of Lemma \ref{L:surjective}.}

The following lemma is likely well-known (and follows easily from formal
smoothness):
\begin{lemma}\label{L:surjective}
    If $X \arr Y$ is smooth and surjective then the maps $J_m X \arr J_m Y$ are
    smooth for all $0 \leq m < +\infty$, and surjective for all $0 \leq m \leq
    +\infty$
\end{lemma}

From Lemmas \ref{L:etalepullback} and \ref{L:surjective} we get the
following proposition:
\begin{prop}\label{P:principalbundle}
    Let $0 \leq m < +\infty$. If $E \arr M$ is an etale-locally trivial
    principal $G$-bundle then $J_m E \arr J_m M$ is an etale-locally trivial
    principal $J_m G$-bundle. 
\end{prop}
\begin{proof}
    For $E \arr M$ to be etale-locally trivial means that there is a surjective
    etale morphism $U \arr M$ such that the pullback of $E$ over $U$ is
    isomorphic to the trivial $G$-bundle $U \times G$. Now $J_m$ preserves
    etale maps (by Lemma \ref{L:etalepullback} and the fact that etale maps are
    preserved under base change) and thus $J_m U \arr J_m M$ is etale and
    surjective. The proof is finished by observing that $J_m$ preserves
    pullbacks (which follows from the definition of the pullback via the
    functor of points).
\end{proof}

Proposition \ref{P:principalbundle} on jet schemes has the following corollary:
\begin{cor}\label{C:proquotient}
    Suppose $X$ has a free $G$-action such that an etale-locally trivial
    quotient $X \arr X / G$ exists. Then $J_m (X/G)$ is isomorphic to $J_m X /
    J_m G$, $0 \leq m < +\infty$, and $J_{\infty} (X/G)$ is isomorphic to
    $J_{\infty} X / J_{\infty} G$, where this last quotient is the pro-group
    quotient, ie. the inverse limit of the quotients $J_m X / J_m G$, $0 \leq m
    < +\infty$.
\end{cor}
If $X \arr X / G$ is etale-locally trivial then it is also surjective, so by
Corollary \ref{C:proquotient} and Lemma \ref{L:surjective} the map $J_\infty X
\arr J_\infty X / J_\infty G$ is surjective. If $X$ is affine with a free
$G$-action and $G$ is reductive then $X / G = X// G$, the GIT quotient, and $X
\arr X / G$ is etale-locally trivial by Luna's slice theorem \cite{Lu73} (the
theorem applies because all orbits under a free action are closed, see the
discussion on page 53 of \cite{Bo91}). All the quotients we study will be of
this type. 

Now suppose that $X$ has an automorphism $\sigma$ of finite order $k$. This
automorphism lifts to an automorphism $\sigma$ of the jet and arc schemes $J_m
X$.  Choose a fixed $k$th root of unity $q$, and let $m(q)$ denote the
automorphisms of $\C[z]/z^n$ and $\C[[z]]$ induced by sending $z \mapsto qz$.
\begin{defn}
    Let $\tilde{\sigma}$ denote the automorphism $\sigma \circ m(q)^{-1}$. The
    \emph{twisted jet (resp. arc) scheme} $J_m^{\tilde{\sigma}} X$ is the
    equalizer of the morphisms $\Id_{J_m X}$ and $\tilde{\sigma}$ in the
    category of schemes. 
    
    In other words, if $m < +\infty$ then $J_m^{\tilde{\sigma}} X$ represents
    the functor $Y \mapsto \{f \in \Hom(Y \times \Spec \C[z] / z^m, X) : f
    \circ m(q) = \sigma \circ f\}$, while $J_\infty^{\tilde{\sigma}} X$
    represents the functor $Y \mapsto \{f \in \Hom(Y \times \Spec \C[[z]], X) :
    f \circ m(q) = \sigma \circ f\}$.
\end{defn}

$J_m^{\tilde{\sigma}} X$ is a closed subscheme of $J_m X$, and $(J_m
X)^{\tilde{\sigma}}$ is separated for all $m$. Since $\Spec \C[[z]]$ is the
direct limit of schemes $\Spec \C[z]/z^n$, it follows from the functor of
points characterisation that $J_{\infty}^{\tilde{\sigma}} X$ is the inverse
limit of schemes $J_m^{\tilde{\sigma}} X$, $0 \leq m < +\infty$.  Since
$J_m^{\tilde{\sigma}} X$ is a closed subscheme of $J_m X$, it is covered by the
inverse images of the open subschemes $J_m U \subset J_m X$, for $U \subset X$
open affine. The inverse image of $J_m U$ in $J_m^{\tilde{\sigma}} X$ is the
same as the inverse image of $\tilde{\sigma}(J_m U) = J_m \sigma(U)$. Thus the
inverse image of $J_m U$ in $J_m^{\tilde{\sigma}} X$ is the same as in the
inverse image of $J_m V$, where $V = U \cap \sigma(U) \cap \ldots \cap
\sigma^{k-1}(U)$. By definition $\sigma(V) = V$, and $V$ is affine because $X$
is separated.  Finally, the pullback of $J_m V$ to $J_m^{\tilde{\sigma}} X$ is
$J_m^{\tilde{\sigma}} V$, and $J_m^{\tilde{\sigma}} V$ is affine. We conclude
that $J_m^{\tilde{\sigma}} X$ is covered by open affines $J_m^{\tilde{\sigma}}
U$ where $U\subset X$ runs through open affines such that $\sigma(U) = U$.

The following lemma is an immediate consequence of the definition of tangent
and jet (resp. arc) schemes via functor of points. 
\begin{lemma}\label{L:tangent}
    Let $\sigma_*$ be the automorphism induced by $\sigma$ on $TX$. Then the
    tangent scheme to $J_m^{\tilde{\sigma}} X$ is naturally isomorphic to
    the twisted jet (resp. arc) scheme $J_m^{\widetilde{\sigma_*}}(TX)$ of
    the tangent scheme to $X$.
\end{lemma}

Using known results for finite-dimensional varieties, we can show that the
twisted jet scheme of a smooth variety is also smooth. 
\begin{lemma}\label{L:twistedsmooth}
    Let $0 \leq m < +\infty$.  If $X$ is a smooth variety with a finite-order
    automorphism $\sigma$ then $J_m^{\tilde{\sigma}} X$ is a smooth variety.
    In addition, if $X$ and $Y$ are both smooth varieties with finite-order
    automorphisms $\sigma_X$ and $\sigma_Y$ and $X \arr Y$ is a
    $\sigma$-equivariant smooth map then $J_m^{\tilde{\sigma}} X \arr 
    J_m^{\tilde{\sigma}} Y$ is smooth.
\end{lemma}
\begin{proof}
    We can assume that $X$ is affine. Since $X$ is smooth, $J_m X$ is also a
    smooth variety. The twisted jet scheme $J_m^{\tilde{\sigma}} X$ is the
    fixed-point scheme of the finite group $\langle \tilde{\sigma} \rangle$. It
    is a well-known consequence of Luna's slice theorem that the fixed-point
    variety of a reductive algebraic group acting on a smooth variety is also
    smooth. This also holds for the fixed-point scheme by Proposition 7.4 of
    \cite{Fo73}, so $J_m^{\tilde{\sigma}} X$ is smooth.\footnote{If $X$ is
    not smooth then the fixed-point scheme of a reductive group action can be
    non-reduced.} 

    Since $J_m^{\tilde{\sigma}} X$ is a smooth variety the tangent scheme
    is a vector bundle. By Lemma \ref{L:tangent}, $T_x J_m^{\tilde{\sigma}} X =
    (T_x J_m X)^{\tilde{\sigma}_*}$ and similarly $T_y J_m^{\tilde{\sigma}} Y =
    (T_y J_m Y)^{\tilde{\sigma}_*}$.  If $X \arr Y$ is smooth then $J_m X \arr
    J_m Y$ is smooth by Lemma \ref{L:surjective}, hence $T J_m X \arr T J_m Y$
    is surjective on fibres, and it follows that $(T_x J_m
    X)^{\tilde{\sigma}_*} \arr (T_y J_m Y)^{\tilde{\sigma}_*}$ is surjective.
    Since both $ J_m^{\tilde{\sigma}} X$ and $J_m^{\tilde{\sigma}} Y$ are
    smooth, $J_m^{\tilde{\sigma}} X \arr J_m^{\tilde{\sigma}} Y$ is a smooth
    map.
\end{proof}
Note that $J_m^{\tilde{\sigma}} X$ is not necessarily irreducible, as
$X^{\sigma}$ can be disconnected. 

We also have the following analogue of Lemma \ref{L:etalepullback}.
\begin{lemma}\label{L:twistedetale}
    Let $0 \leq m \leq +\infty$. Suppose that $X$ and $Y$ have finite-order
    automorphisms $\sigma_X$ and $\sigma_Y$. If $X \arr Y$ is an etale
    $\sigma$-equivariant map then $J_m^{\tilde{\sigma}} X = X^\sigma
    \times_{Y^{\sigma}} J_m^{\tilde{\sigma}} Y$.
\end{lemma}
\begin{proof}
    By Lemma \ref{L:etalepullback}, $J_m X \iso X \times_Y J_m Y$.  The
    automorphism $\tilde{\sigma}_X$ on $J_m X$ translates to the unique
    automorphism on the latter space which lies above $\sigma_X$ on $X$,
    $\sigma_Y$ on $Y$, and $\tilde{\sigma}_Y$ on $J_m Y$. The result follows
    from the functor of points characterisations of the twisted jet and arc
    schemes and the fibre product.
\end{proof}

Finally, the jet structure distinguishes a subbundle of the tangent bundle
of a jet or arc space.
\begin{defn}
    If $X$ is a variety with finite-order automorphism $\sigma$, we let
    $T_{const} J_m^{\tilde{\sigma}} X$ denote the pullback $X^{\sigma}
    \times_{T X^{\sigma}} T J_m^{\tilde{\sigma}}$, where $T
    J_m^{\tilde{\sigma}} X \arr T X^{\sigma}$ is the differential of the
    projection $J_m^{\tilde{\sigma}} X \arr X^{\sigma}$ and $X^{\sigma} \arr T
    X^{\sigma}$ is the zero section. Intuitively $T_{const}
    J_m^{\tilde{\sigma}}$ is the space of infinitesimal families of jets (resp.
    arcs) which are constant at $z=0$.
\end{defn}

\subsection{Connectedness of the twisted arc group}

In this section $G$ will be a connected algebraic group with Lie algebra $L$,
such that the diagram automorphism $\sigma$ lifts to $G$ (for example, this
occurs if $G$ is simply-connected). $H$ will be the torus corresponding to the
chosen Cartan $\mfh$. 

We recall some basic facts about diagram automorphisms and the structure of
$L$, using terminology and basic results from Chapter 9, Section 5 of
\cite{Ca05}.  Let $\mfh_i$ denote the $q^i$th eigenspace of $\sigma$ acting on
$\mfh$. By definition, there is a choice of simple roots
$\alpha_1,\ldots,\alpha_l$ such that $\sigma$ permutes the corresponding
coroots $h_{\alpha_i}$ and Chevalley generators $e_{\alpha_i}$. If $J$ is an
orbit the $\sigma$-action on simple roots, let $\alpha_J = \frac{1}{|J|}
\sum_{\alpha \in J} \alpha$. Then the set $\{ \alpha_J |_{\mfh_0} : J \text{ is
an orbit of } \sigma \}$ is a set of simple roots for $L_0$. Restriction to
$\mfh_0$ gives an isomorphism between the subgroup $W^{\sigma}$ (where $W$ is
the Weyl group of $L$) and the Weyl group $W(L_0)$ of $L_0$. The simple
generator $s_J$ of $W(L_0)$ given by reflection through $\alpha_J$ on $\mfh_0$
corresponds to the element of $W^{\sigma} \subset W(L)$ which is the maximal
element in the subgroup of $W(L)$ generated by reflection through the simple
roots in $J$. In addition, we will need:
\begin{lemma}\label{L:diagramweyl}
    If $N(H)$ is the normalizer of $H$ in $G$, then $N(H)^{\sigma} =
    N_{G^{\sigma}}(H^{\sigma})$, the normalizer of $H^{\sigma}$ in
    $G^{\sigma}$. Consequently $W^{\sigma} = N(H)^{\sigma} / H^{\sigma} \subset
    W(L)$. Furthermore, the inclusion $W(L_0) \iso W^{\sigma} \incl W(L)$ is
    length-preserving, in the sense that if $w \in W^{\sigma}$, then it is
    possible to get a reduced expression for $w$ by first taking a reduced
    expression $w = s_{J_1} \cdots s_{J_r}$ for $w$ in $W(L_0)$, and then
    replacing each $s_{J_i}$ with a reduced expression in $W(L)$.
\end{lemma}
\begin{proof}
    For the first part, let $\rho \in \mfh$ be the element such that
    $\alpha(\rho)=1$ for all simple roots $\alpha$ of $L$. Then $\rho$ is
    regular in $\mfh$ and belongs to $\mfh_0$. Any element of
    $N_{G^{\sigma}}(H^{\sigma})$ sends $\rho$ to another regular element of
    $\mfh$, and hence belongs to $N(H)$.

    For the second part, we we refer to the proof of Proposition 9.17 of
    \cite{Ca05}.
\end{proof}

We can use Lemma \ref{L:diagramweyl} to prove:
\begin{lemma}\label{L:bigcell}
    Choose a Borel subgroup $\mcB$ of $G$ containing $H$ and compatible with
    $\sigma$ and let $X = \overline{B} B$ be the big cell of the corresponding
    Bruhat decomposition. If $x \in G$ belongs to a Bruhat cell $B w B$ with $w
    \in W^{\sigma}$ then there is $g \in N(H)^{\sigma}$ such that $g x \in X$.
\end{lemma}
\begin{proof}
    If we take for $g$ a representative of $w^{-1}$ in $N(H)^{\sigma}$, then 
    $g B w B \subset \overline{B} B$.
\end{proof}

\begin{prop}
    $G^{\sigma}$ is connected.
\end{prop}
\begin{proof}
    The connected component $(G^{\sigma})^{\circ}$ of $G^{\sigma}$ is a
    connected reductive group with Lie algebra $L_0$. Since $\sigma$ permutes
    coroots, it is easy to see that $H^{\sigma}$ is a connected torus, and in
    fact is a Cartan in $(G^{\sigma})^{\circ}$. As in Lemma \ref{L:bigcell},
    let $\mcB$ be a Borel subgroup of $G$ containing $H$ and compatible with
    $\sigma$, and let $X$ be the corresponding big cell. If $g \in G^{\sigma}$
    belongs to a Bruhat cell $B w B$ then $g \in B w B \cap \sigma(B w B)$, so
    $w \in W^{\sigma}$. By Lemma \ref{L:diagramweyl}, every element of
    $N(H)^{\sigma}$ can be implemented by an element of $(G^{\sigma})^{\circ}$.
    So by Lemma \ref{L:bigcell}, we just need to prove that $G^{\sigma} \cap X$
    is contained in $(G^{\sigma})^{\circ}$.

    Now as an algebraic variety, $X \iso \overline{U} \times H \times U$, where
    $U$ is the unipotent radical of $\mcB$. The action of $\sigma$ on $X$
    translates to the action of $\sigma$ on each factor. Let $\mfu$ be the Lie
    algebra of $U$. The exponential map for nilpotent Lie algebras is bijective,
    so $U^{\sigma}$ is the unipotent subgroup corresponding to the nilpotent
    Lie algebra $\mfu^{\sigma}$. In particular $U^{\sigma}$ is connected, and
    similarly with $\overline{U}^{\sigma}$. We conclude that $X^{\sigma} =
    G^{\sigma} \cap X$ is connected. 
\end{proof}

Using the fact that the exponential map for nilpotent (resp. pro-nilpotent) Lie
algebras is bijective, we immediately get the following corollary.
\begin{cor}
    If $0 \leq m < +\infty$ then $J_m^{\tilde{\sigma}} G$ is a connected
    algebraic group with Lie algebra $L[z] / z^m$. Similarly
    $J_\infty^{\tilde{\sigma}} G$ is a connected pro-algebraic group with Lie
    algebra $L[[z]]^{\tilde{\sigma}}$.
\end{cor}
As a scheme the Lie algebra of $J_{m}^{\tilde{\sigma}} G$ can be
identified with $J_m^{\tilde{\sigma}} L$.

The following proposition will be crucial in the next section, since it proves
that $J_m^{\tilde{\sigma}} (G/H)$ is a $J_m^{\tilde{\sigma}} G$-homogeneous
space. 
\begin{prop}\label{P:homog}
    $J_m^{\tilde{\sigma}} (G/H) \iso J_m^{\tilde{\sigma}} G /
    J_m^{\tilde{\sigma}} H$, where the latter space is either the group
    quotient if $0 \leq m < +\infty$, or the pro-group quotient if $m =
    +\infty$.
\end{prop}
\begin{proof}
    $G \arr G/ H$ is an etale-locally trivial principal bundle, so $J_m (G/H)
    \iso J_m G / J_m H$. There is an inclusion $J_m^{\tilde{\sigma}} G /
    J_m^{\tilde{\sigma}} H \incl (J_m G / J_m H)^{\tilde{\sigma}}$. To prove
    the proposition, we will show that this inclusion is surjective for all $m
    < +\infty$. If $m < +\infty$ then biregularity follows from bijectivity
    because $(J_m G / J_m H)^{\tilde{\sigma}}$ will be a homogeneous space.
    Biregularity for $m = +\infty$ follows from the universal property of
    inverse limits.

    Define $\alpha : J_m G \arr J_m G$ by $g \mapsto g^{-1} \tilde{\sigma}(g)$. 
    To show that the inclusion is surjective we need to show that every element
    of $(J_m G/ J_m H)^{\tilde{\sigma}}$ has a representative $x \in J_m G$
    such that $\alpha(x) =e$. The map $\alpha$ has a number of nice properties.
    First, the fibres of $\alpha$ are left $J_m^{\tilde{\sigma}} G$-cosets.
    Second, $g \in J_m G$ represents an element of $(J_m G / J_m
    H)^{\tilde{\sigma}}$ if and only if $\alpha(g) \in J_m H$.  Third, if
    $\alpha(g) \in J_m H$ and $h \in J_m H$ then $\alpha(gh) = \alpha(g)
    \alpha(h)$. By these last two properties, we will have $(J_m G /
    J_m H)^{\tilde{\sigma}} = J_m^{\tilde{\sigma}} G /
    J_m^{\tilde{\sigma}} H$ if and only if $\alpha(J_m G) \cap J_m H =
    \alpha(J_m H)$. 

    Our proof depends on the Bruhat geometry of $G$, so pick a Borel subgroup
    $B \subset G$ compatible with $\sigma$. Let $X = \overline{B} B$ be the big
    cell. Suppose $x \in J_m G$ and $\alpha(x) \in J_m H$. Writing $x(0) =
    b_0 w b_1$, we get $\alpha(x(0)) = b_1^{-1} w^{-1} \alpha(b_0) \sigma(w)
    \sigma(b_1) \in H$. But $\alpha(b_0) \in B$, so $w B \cap B \sigma(w) B
    \neq \emptyset$, and thus $w$ belongs to $W^{\sigma}$. Consequently there
    is $g_0 \in G^{\sigma}$ such that $g_0 x(0) \in X$, implying that $g_0 x
    \in J_m X$. Since $\alpha(x) = \alpha(g_0 x)$ for $g_0 \in G^{\sigma}$, we
    just need to show that $\alpha(J_m X) \cap J_m H$ is contained in
    $\alpha(J_m H)$.

    The space $X$ is isomorphic to $\overline{U} \times B$ via the
    multiplication map, where $\overline{U}$ is the unipotent subgroup of
    $\overline{B}$. Thus we can write any element of $J_m X$ uniquely as $a(z)
    b(z)$, where $a(z) \in J_m \overline{U}$ and $b(z) \in J_m B$. Suppose
    $\alpha(a(z)b(z)) = h(z) \in J_m H$. Since $\alpha(a(z)b(z)) = 
    b(z)^{-1} \alpha(a(z)) \tilde{\sigma}(b(z))$, we see that $\alpha(a(z)) =
    b(z) h(z) \tilde{\sigma}(b(z))^{-1} \in J_m B$.  Since $\alpha(a(z)) \in
    J_m \overline{U}$, this implies that $\alpha(a(z)) = e$ and consequently
    $\alpha(b(z)) = h(z)$. To finish the proof, observe that $B \iso U \times
    H$ via the multiplication map, where $U$ is the unipotent subgroup of $B$.
    Writing $b(z) = b'(z)h'(z) $ for $b'(z) \in J_m U$ and $h'(z) \in J_m H$,
    we get $\alpha(b(z)) = h'(z)^{-1} \alpha(b'(z)) \tilde{\sigma}(h'(z))$, and
    hence $\alpha(b'(z))$ can be written as an element of $J_m H$. This
    implies that $\alpha(b'(z)) = e$, finishing the proof, since $\alpha(h'(z))
    = h(z)$.
\end{proof}

\section{Slice theorems for the adjoint action}\label{S:orbit}

We continue to use the notation from Section \ref{S:jetschemes}. In particular,
$G$ is a connected algebraic group with Lie algebra $L$ such that $\sigma$
extends to $G$, and the Lie algebra of $J_m^{\tilde{\sigma}} G$ is identified
with $J_m^{\tilde{\sigma}} L$. In addition, we fix a standard parabolic
subalgebra $\mfp_0 \subset L_0$, and let $\mfp_m = \{f \in J_m^{\tilde{\sigma}}
: f(0) \in \mfp_0\}$. Note that $\mfp_\infty$ is the completion of a standard
parahoric in $L[z^{\pm 1}]^{\tilde{\sigma}}$, which we also denote by
$\hat{\mfp}$.  We let $\mcP_{m}$ be the connected algebraic (resp.
pro-algebraic) subgroup of $J_m^{\tilde{\sigma}} G$ corresponding to $\mfp_m$,
and $\mcN_m$ be the nilpotent (resp.  pro-nilpotent) radical of $\mcP_m$.  The
reductive factor $\mfp_0 \cap \overline{\mfp_0}$ of $\mfp_0$ is denoted by
$\mfg_0$.

In this section we prove two slice theorems for the adjoint action of $\mcP_m$
on $\mfp_m$. The first is an analogue of the well-known slice theorem for
regular semisimple elements in $L$, and is given in Subsection
\ref{SS:regularss}.  The second is an analogue of the Kostant slice theorem,
and is given in Subsection \ref{SS:regular}.  These theorems will be used in
the next section to determine the $\mcP_\infty$-invariant $\mcN_\infty$-basic
elements of $\Omega_{>0}^* \mfp_{\infty}$. 

The slice theorems are stated in terms of the GIT quotients $Q := L // G$ (ie.
$Q$ is the affine variety with coordinate ring $\C[Q] = (S^* L^*)^G$) and $R :=
\mfp_0 // \mcP_0$. Recall that $\C[Q]$ is a free algebra generated by
homogeneous elements in degrees $m_1+1,\ldots,m_l+1$, where $l$ is the rank of
$L$ and $m_1,\ldots, m_l$ are the exponents. A similar result holds for
$\C[R]$:
\begin{lemma}\label{L:parabolicinv} 
    Let $\mfu_0$ the nilpotent radical of $\mfp_0$, so that $\mfp_0 = \mfg_0
    \oplus \mfu_0$. If $f \in \C[R]$ then $f(x,y) = f(x,0)$ for all $x \in
    \mfg_0$, $y \in \mfu_0$. Consequently, if $\mcM$ is the Levi subgroup of
    $\mcP_0$ then $R\iso \mfg_0 // \mcM \iso \mfh_0 // W(\mfg_0)$, where
    $W(\mfg_0)$ is the Weyl group of $\mfg_0$, and $\C[R]$ is a free algebra
    generated by homogeneous elements in degrees given by the exponents of
    $\mfg_0$.
\end{lemma}
\begin{proof}
    The set of regular elements $\mfh_0^r$ is dense in $\mfh_0$. Since
    $[\mfg_0,x] + \mfh_0 = \mfg_0$ for any element
    $x \in \mfh_0^r$, the set $\mcM \mfh_0^r$ is dense in $\mfg_0$.
    Let $\mcN_0$ be the unipotent subgroup corresponding to $\mfu_0$. If $x$
    belongs to $\mfh_0^r$ then $\mcN_0 x = x + \mfu_0$. Since $\mcN_0$ is
    normal in $\mcP_0$, this property extends to any $x \in \mcM \mfh_0^r$. So
    if $f$ is invariant then $f(x,y) = f(n(x,0)) = f(x,0)$ for $x$ in an open
    dense subset of $\mfp_0 \cap \overline{\mfp_0}$.
\end{proof}

\subsection{The regular semisimple slice}\label{SS:regularss}

Let $L^{rs} \subset L$ be the subset of regular semisimple elements. $L^{rs}$
is an affine open subset of $L$ (its complement is the vanishing set of a
single $G$-invariant function) and consequently the image $Q^r$ of $L^{rs}$ in
$Q = L // G$ is open. The well-known regular semisimple slice theorem states
that there is a commutative square
\begin{equation}\label{E:finiteslice}
    \xymatrix{ G/H \times_W \mfh^r \ar[r] \ar[d] & L^{rs} \ar[d] \\
                \mfh^r / W \ar[r] & Q^r },
\end{equation}
where $W$ is the Weyl group of $L$ and $\mfh^r$ is the set of regular elements
in $\mfh$. The notation $G / H \times_W \mfh^r$ denotes the quotient of $G / H
\times \mfh^r$ under the free action of $W = N(H) / H$ acting by right
multiplication on $G / H$ and by the adjoint action on $\mfh^r$. Both
horizontal maps are isomorphisms. The top horizontal map is given by
multiplication, while the bottom horizontal map is projection to $Q$.

Since $\sigma$ is an automorphism, the sets $L^{rs}$ and $\mfh^r$ are closed
under $\sigma$ and we can apply $J_m^{\tilde{\sigma}}$ to both spaces. The image
$R^{r}$ of $\mfp_0 \cap L_0^{rs}$ in $R$ is open, since it's complement is the
zero set of a single $\mcP_0$-invariant function. As usual, let $\mcP_\infty /
J_\infty^{\tilde{\sigma}} H$ denote the pro-group quotient.  Similarly
$\mcP_{\infty} / J_\infty^{\tilde{\sigma}} H \times_{W(\mfg_0)}
J_\infty^{\tilde{\sigma}} \mfh^r$ will denote the pro-group quotient of
$\mcP_{\infty} / J_\infty^{\tilde{\sigma}} H \times J_\infty^{\tilde{\sigma}}
\mfh^r$ by $W(\mfg_0)$, and $J_\infty^{\tilde{\sigma}} \mfh^r / W(\mfg_0)$
denotes the pro-group quotient of $J_m^{\tilde{\sigma}} \mfh^r$ by $W(\mfg_0)$.
We have the following analogue of Equation (\ref{E:finiteslice}) for twisted
jet and arc schemes. 
\begin{thm}\label{T:parahoricslice}
    Let $0 \leq m \leq +\infty$. Then there is a commutative diagram
    \begin{equation}\label{E:parahoricslice}
        \xymatrix{ \mcP_m / J_m^{\tilde{\sigma}} H \times_{W(\mfg_0)}
            J_m^{\tilde{\sigma}} \mfh^r \ar[r] \ar[d] & 
                \mfp_m \cap J_m^{\tilde{\sigma}} L^{rs} \ar[d] \\
            \left(J_m^{\tilde{\sigma}} \mfh^r\right) / W(\mfg_0) \ar[r] & 
                R^r \times_{Q^{\sigma}} J_m^{\tilde{\sigma}} Q^r}
    \end{equation}
    in which the horizontal maps are isomorphisms, with the top map induced
    by multiplication and the bottom map induced from the two projections
    $J_m^{\tilde{\sigma}} \mfh^r / W(\mfg_0) \arr J_m^{\tilde{\sigma}} Q^r$ and
    $\mfh_0^r / W(\mfg_0) \iso R^r$. 
\end{thm} 

To prove Theorem \ref{T:parahoricslice}, we start with the case $m=0$ (likely
well-known, but we give the proof for completeness).
\begin{lemma}\label{L:parabolicslice}
    There is a commutative diagram
    \begin{equation*}
        \xymatrix{ \mcP_0 / H^{\sigma} \times_{W(\mfg_0)} \mfh_0^r \ar[r] \ar[d]
                        & \mfp_0 \cap L_0^{rs} \ar[d] \\
                    \mfh^r_0 / W(\mfg_0) \ar[r] & R^{r} } 
    \end{equation*}
    in which both horizontal maps are isomorphisms. The top horizontal map is
    induced by multiplication, while the bottom horizontal map is induced by
    the projection $\mfh_0 \arr R$.
\end{lemma}
\begin{proof}
    That the bottom map is an isomorphism comes from Lemma \ref{L:parabolicinv}.
    
    The Weyl groups of $\mfg_0$ and $L_0$ can be expressed in terms of $\mcM$ and
    $G^{\sigma}$ as $W(\mfg_0) = N_{\mcM}(H^{\sigma} \cap \mcM) / (H^{\sigma}
    \cap \mcM)$ and $W(L_0) = N_{G^{\sigma}}(H^{\sigma}) / H^{\sigma}$. Using
    the Bruhat decomposition for $G^{\sigma}$ and $\mcM$ simultaneously, as
    well as the Levi decomposition for $\mcP_0$, it is possible to show that
    $N_{G^{\sigma}}(H^{\sigma}) \cap \mcP_0 \subset N_{\mcM}(H^{\sigma} \cap
    \mcM)$.  The resulting inclusion $N_{G^{\sigma}}(H^{\sigma}) \cap \mcP_0 /
    H^{\sigma} \subset N_{\mcM}(H^{\sigma} \cap \mcM) / H^{\sigma} \cap \mcM$
    is an isomorphism.
         
    Now the commutative diagram in Equation (\ref{E:finiteslice}) can be
    extended by adding the commutative square
    \begin{equation}\label{E:parabolicslice}
        \xymatrix{ \mcP_0 / H^{\sigma} \times_{W(\mfg_0)} \mfh_0^r \ar[r] \ar[d]
            & \mfp_0 \cap L_0^{rs} \ar[d] \\
                G^{\sigma}/H^{\sigma} \times_{W(L_0)} \mfh_0^r \ar[r] & L_0^{rs} },
    \end{equation}
    in which the vertical maps are the natural inclusions. To show that the
    left vertical map is injective take two elements $([p],x)$ and $([p'],x')$
    which are equal in the codomain. This means that there is $w \in
    N_{G^{\sigma}}(H^{\sigma})$ with $[p w^{-1}] = [p']$ and $w x_0 = x_0'$.
    The former condition implies that $w \in \mcP_0 \cap N_G(H)$, so $[w] \in
    W(L_0)$ represents an element of $W(\mfg_0)$, and $([p],x) = ([p'],x')$ in
    $\mcP_0 / H_0 \times_{W(\mfg_0)} \mfh_0^r$. 

    Since the bottom map of Equation (\ref{E:parabolicslice}) is an
    isomorphism, we just need to show that $\mcP_0 / H^{\sigma}
    \times_{W(\mfg_0)} \mfh_0^r$ maps onto $\mfp_0 \cap L_0^{rs}$. Suppose $x \in
    \mfp_0$ is semisimple in $L_0$. Since diagonalizability is preserved by
    restriction to an invariant subspace and by descent to a quotient by an
    invariant subspace, we can write $x = x_0 + x_1$, where $x_0$ is a
    semisimple element of $\mfp_0 \cap \overline{\mfp_0}$ and $x_1 \in \mfu_0$.
    Conjugating $x_0$ by an element of the Levi factor $\mcM$ to be in
    $\mfh_0$, we can assume that $x \in \mfb_0$, a Borel subalgebra of $L_0$
    contained in $\mfp_0$. Thus the problem is reduced to showing that $\mfb_0
    \cap L_0^{rs} \subset \mcB_0 \mfh_0^r$. Given $x$ in the former set, take
    $g \in G^{\sigma}$ such that $g x = y \in \mfh^r_0$. Then $\mfb_0$ and $g
    \mfb_0$ both contain $\mfh_0$, so there is $w \in
    N_{G^{\sigma}}(H^{\sigma})$ such that $w \mfb_0 = g \mfb_0$. Since Borel's
    are self-normalizing, $g^{-1} w \in B_0$ and $x = (g^{-1} w) (w^{-1} y) \in
    B_0 \mfh_0^r$.
\end{proof}

We need two facts about diagram automorphisms and the structure of $L$. We use
the convention from Section \ref{S:jetschemes} to express the simple roots
$\{\alpha_J\}$ of $L_0$ in terms of simple roots $\{\alpha\}$ of $L$. 
\begin{lemma}\label{L:diagramrs}
     $\mfh_0 \cap \mfh^r = \mfh_0^r$, the set of elements in $\mfh_0$ which are
    regular in $L_0$. Similarly, $L_0 \cap L^{rs} = L_0^{rs}$.
\end{lemma}
\begin{proof}
    The restriction map $\mfh^* \arr \mfh_0^*$ sends roots of $L$ to positive
    multiples of roots of $L_0$ by Proposition 9.18 of \cite{Ca05}.  All the
    roots of $L_0$ are covered by this map, so $\mfh_0 \cap \mfh^r = \mfh^r_0$.
    An element $x \in L_0$ is semisimple in $L_0$ if and only if it is
    semisimple in $L$. If it is semisimple in $L_0$ then it can be conjugated
    to an element of $\mfh_0$, so the statement for $L_0$ follows from the
    statement for $\mfh_0$.
\end{proof}

\begin{lemma}\label{L:diagrampara}
    There is a parabolic $\mfp'$ of $L$ preserved by $\sigma$ such that $\mfp'
    \cap L_0 = \mfp_0$. If $\mfm$ is the standard reductive factor of $\mfp'$
    then $\mfm \cap L_0 = \mfg_0$, the reductive factor of $\mfp_0$, and
    $W(\mfm)^{\sigma} = W(\mfg_0)$, where both are regarded as subgroups of
    $W(L)$. 
\end{lemma}
\begin{proof}
    Let $S$ be the subset of simple roots $\{\alpha_J\}$ determining $\mfp_0$
    and let $S'$ be the subset of simple roots of $L$ which appear in some
    $\sigma$-orbit $J$ for $\alpha_J \in S$. Let $\mfp'$ be the parabolic
    subalgebra determined by $S'$. Clearly $\mfp'$ is $\sigma$-invariant. By
    Lemma \ref{L:diagramweyl}, an element $w \in W^{\sigma}$ belongs to
    $W(\mfg_0)$ if and only if it has a reduced expression consisting of
    reflections through simple roots in $S'$, which is exactly the condition
    that $w$ belongs to $W(\mfm)$.  If $\alpha_J$ is a simple root of $L_0$,
    then the corresponding positive Chevalley generator $e_J$ is a linear
    combination of the positive Chevalley generators corresponding to the
    simple roots of $L$ in $J$, and similarly for the negative Chevalley
    generator $f_J$. Since $\mfp_0$ is generated as a Lie algebra by $\mfh_0$,
    all the $e_J$'s, and the $f_J$'s such that $\alpha_J \in S$, it follows
    that $\mfp_0 \subset \mfp' \cap L_0$.  Since the $f_J$'s with $\alpha_J \in
    S$ are the only negative generators in $\mfp' \cap L_0$, and $\mfp' \cap
    L_0$ is a parabolic subalgebra of $L_0$, it follows that $\mfp' \cap L_0 =
    \mfp_0$. Similarly $\mfm \cap L_0 = \mfg_0$.
\end{proof}

The real form $\mfh_{\R}$ of $\mfh$ is the real subspace where all roots take
real values, or equivalently the real span of the coroots. If $x \in \mfh$ let
$\re x$ be the projection of $x$ to $\mfh_{\R}$ under the (real-linear)
splitting $\mfh = \mfh_{\R} \oplus i \mfh_{\R}$. Note that $\re \sigma x =
\sigma \re x$ and $\re w x = w \re x$ for all $w \in W$. 
\begin{proof}[Proof of Theorem \ref{T:parahoricslice}]
    First we show that the bottom map of Equation (\ref{E:parahoricslice}) is
    an isomorphism.  Let $\mfp'$ be the parabolic of $L$ over $\mfp_0$, as in
    Lemma \ref{L:diagrampara}. We start by proving that $J_m^{\tilde{\sigma}}
    \mfh^r / W(\mfg_0) \iso J_m^{\tilde{\sigma}} (\mfh^r / W(\mfm))$, where
    $W(\mfm)$ is the Weyl group of the reductive factor of $\mfp'$. Since
    $J_m^{\tilde{\sigma}} (\mfh^r / W(\mfm))$ is smooth when $m < +\infty$ by
    Lemma \ref{L:twistedsmooth}, it is sufficient to prove that the map is
    bijective. By Corollary \ref{C:proquotient}, $J_m \mfh^r/W(\mfm) \iso J_m
    (\mfh^r/W(\mfm))$, so every element of $J_m^{\tilde{\sigma}}
    (\mfh^r/W(\mfm))$ is represented by an element of $f \in J_m \mfh^r$ such
    that $w \tilde{\sigma}(f) = f$ for some $w \in W(\mfm)$. Let $S'$ be the
    set of simple roots determining $\mfp'$, let $\Delta'$ be the set of all
    roots of $\mfm$, and let $D = \{x \in \mfh_{\R} : \alpha(x) \neq 0, \alpha
    \in \Delta'\}$. The connected components of $D$ are of the form $C \times
    \R^r$, where $C$ is an open Weyl chamber of $\mfm$ and $r = \dim \mfh -
    |S'|$.  Consequently $W(\mfm)$ acts transitively and freely on the
    connected components of $D$, so we can assume that $\re f(0) \in D_0 = \{ x
    \in \mfh_{\R} : \alpha(x) > 0, \alpha \in S'\}$. But $S'$ is
    $\sigma$-invariant, so $D_0$ is also $\sigma$-invariant, and thus $\re
    \sigma f(0) = \sigma \re f(0) \in D_0$.  Since $\re f(0) = w \sigma \re
    f(0)$, this implies that $w = e$ and consequently $f \in
    J_m^{\tilde{\sigma}} \mfh^r$. Thus the map $J_m^{\tilde{\sigma}} \mfh^r
    \arr J_m^{\tilde{\sigma}} (\mfh^r/W(\mfm))$ is surjective. Suppose $f,g \in
    J_m^{\tilde{\sigma}}\mfh^r$ are equal in
    $J_m^{\tilde{\sigma}}(\mfh^r/W(\mfm))$. Then there is $w \in W(\mfm)$ such
    that $w f = g$. Since $f(0),g(0) \in \mfh^r_0$, we have $\sigma(w) f(0) =
    g(0) = w f(0)$, and consequently $\sigma(w) = w$. Thus $f$ and $g$ are
    related by an element of $W(\mfm) \cap W^{\sigma} = W(\mfg_0)$. 
 
    As a special case of the above argument, we have $(\mfh^r /
    W(\mfm))^{\sigma} \iso \mfh_0^r / W(\mfg_0) = R^r$. Consequently
    $J_m^{\tilde{\sigma}}(\mfh^r/W(\mfm))$ maps to $R^r$ via evaluation at
    zero, and we conclude that the map $J_m^{\tilde{\sigma}}\mfh^r / W(\mfg_0)
    \arr R^r \times_{Q^{\sigma}} J_m^{\tilde{\sigma}} Q^r$ factors through the
    isomorphism to $J_m^{\tilde{\sigma}}(\mfh^r / W(\mfm))$.  Since $\mfh^r /
    W(\mfm) \arr \mfh^r / W(L)$ is etale, the space
    $J_m^{\tilde{\sigma}}(\mfh^r / W(\mfm))$ is isomorphic to $R^r
    \times_{Q^{\sigma}} J_m^{\tilde{\sigma}}Q^r$ by Lemma \ref{L:twistedetale}.

    We have shown that the bottom map of Equation (\ref{E:parahoricslice}) is
    an isomorphism, so we just need to do the same for the top map.
    Consider the case when $\mfp_0 = L_0$, so that $\mcP_m =
    J_m^{\tilde{\sigma}} G$ and $W(\mfg_0) = W(L_0)$.  Combining Corollary
    \ref{C:proquotient} (note that $H$ is reductive so that $G/H$ is affine)
    and the isomorphism $(G/H) \times_W \mfh^r \arr L^{rs}$, we get an
    isomorphism $J_m (G/H) \times_W J_m \mfh^r \arr J_m L^{rs}$, where the
    former space is the quotient (resp.  pro-quotient). The automorphism
    $\tilde{\sigma}$ on $J_m L^{rs}$ translates to the diagonal action on $J_m
    (G/H) \times_{W(L)} J_m \mfh^r$, and we can show that this isomorphism
    identifies $J_m^{\tilde{\sigma}} L^{rs}$ with $J_m^{\tilde{\sigma}} (G/H)
    \times_{W(L_0)} J_m^{\tilde{\sigma}}\mfh^r$ by a similar argument to the
    proof of Lemma \ref{L:parabolicslice}. Namely, $([f],g) \in J_m (G/H) \times
    J_m \mfh^r$ represents an element of $J_m^{\tilde{\sigma}}L^{rs}$ if and
    only if there is $w \in W(L)$ such that $[\tilde{\sigma}(f)]w^{-1} = [f]$
    and $w \tilde{\sigma}(g) = g$. Assuming that $\re g(0)$ is in the open Weyl
    chamber we get that $w = e$ and thus $[f] \in J_m^{\tilde{\sigma}}(G/H)$,
    $g \in J_m^{\tilde{\sigma}}\mfh^r$.  Similarly, any two elements of
    $J_m^{\tilde{\sigma}}(G/H) \times J_m^{\tilde{\sigma}}\mfh^r$ with the same
    image in $J_m L^{rs}$ are $W(L_0)$-translates. Finally we can apply
    Proposition \ref{P:homog} to replace $J_m^{\tilde{\sigma}} (G/H)$ with
    $J_m^{\tilde{\sigma}}G / J_m^{\tilde{\sigma}}H$.

    Now for the general case look at the square
    \begin{equation*}
        \xymatrix{ \mcP_m / J_m^{\tilde{\sigma}}H\times_{W(\mfg_0)} J_m^{\tilde{\sigma}}\mfh^r \ar[r] \ar[d] 
                            & \mfp_m \cap J_m^{\tilde{\sigma}}L^{rs} \ar[d] \\
                    J_m^{\tilde{\sigma}}G / J_m^{\tilde{\sigma}}H \times_{W(L_0)} J_m^{\tilde{\sigma}}\mfh^r \ar[r] 
                            & J_m^{\tilde{\sigma}}L^{rs} }.
    \end{equation*}
    The group quotient (resp. pro-group quotient) $\mcP_m /
    J_m^{\tilde{\sigma}}H$ is a closed subscheme of
    $J_m^{\tilde{\sigma}}G/J_m^{\tilde{\sigma}}H$.  As in
    Lemma \ref{L:parabolicslice}, both vertical maps are inclusions and
    consequently the top horizontal map is injective. Every $x \in
    J_m^{\tilde{\sigma}}L^{rs}$ can be written as $g y$ for $g \in
    J_m^{\tilde{\sigma}}G$ and $y \in J_m^{\tilde{\sigma}}\mfh^r$. If $x \in
    \mfp_m$ then $x(0) \in \mfp_0$, after which Lemma \ref{L:parabolicslice}
    implies that there is $w \in W(L_0)$ such that $g(0) w^{-1} \in \mcP_0$.
    Consequently $g w^{-1} \in \mcP_m$ and $(gw^{-1},w y)$ maps to $x$, so the
    top map is surjective as required. 
\end{proof}

\subsection{Arcs in the regular locus}\label{SS:regular}

Let $L^{reg}$ denote the open subset of regular elements in $L$, ie. the set of
elements $x$ such that the stabilizer $L^x$ has dimension equal to the rank $l$
of $L$. Note that $L^{reg}$ is $\sigma$-invariant. Kostant famously proved
that the map $L^{reg} \arr Q$ is surjective and smooth, and furthermore is a
$G$-orbit map, in the sense that every fibre is a single $G$-orbit
\cite{Ko63b}. The proof uses the Kostant slice, an affine subspace $\nu
\subset L^{reg}$ of the form $e + L^{f}$, where $\{h,e,f\}$ is a principal
$\mfsl_2$-triple. Kostant showed that  $\nu$ intersects each regular $G$-orbit
in a unique point, and that $\nu \incl L^{reg} \arr Q$ is an isomorphism. 
The following theorem extends this idea to jet and arc groups. 
\begin{thm}\label{T:kostantslice}
    There is a Kostant slice $\nu$ of $L$ which is $\sigma$-invariant and such
    that $\nu^{\sigma}$ is a Kostant slice for $L_0$. If $\nu$ is such a slice
    then $J_m^{\tilde{\sigma}} \nu \arr J_m^{\tilde{\sigma}}
    Q^{\tilde{\sigma}}$ is an isomorphism for all $0 \leq m \leq +\infty$, and
    every $J_m^{\tilde{\sigma}} G$-orbit in $J_m^{\tilde{\sigma}} L^{reg}$
    intersects $J_m^{\tilde{\sigma}} \nu$ in a unique point.
\end{thm}
At $m=0$, Theorem \ref{T:kostantslice} implies that $Q^{\sigma} = L_0 //
G^{\sigma}$. 
    
For Kostant's smoothness result it is possible to incorporate a parabolic
component.
\begin{thm}\label{T:kostantsmooth}
    The map $\mfp_m \cap J_m^{\tilde{\sigma}} L^{reg} \arr R
    \times_{Q^{\sigma}} J_m^{\tilde{\sigma}} Q$ is a surjective $\mcP_m$-orbit
    map for all $0 \leq m \leq +\infty$, and is smooth for $0 \leq m <
    +\infty$. 
\end{thm}

Finally, we have a technical corollary which we will need in the next section.
Recall the definition of $T_{const}$ from the previous section, and define
$T_{>0} \mfp_m$ to be the subbundle of $T \mfp_m$ of the form $\mfp_m \times
\mfu_m$ where $\mfu_m$ is the nilpotent subalgebra of $\mfp_m$, ie. the subset
of elements $f \in \mfp_m$ with $f(0) \in \mfu_0$, the nilpotent radical of
$\mfp_0$.
\begin{cor}\label{C:bundlemap}
    Let $0 \leq m \leq +\infty$.  The differential of the map $\mfp_m \arr R
    \times_{Q^{\sigma}} J_m^{\tilde{\sigma}} Q$ induces a bundle map $T_{>0}
    \mfp_m \arr R \times_{Q^{\sigma}} T_{const} J_m^{\tilde{\sigma}}$. Over
    $\mfp_m \cap J_m^{\tilde{\sigma}} L^{reg}$ the bundle map is surjective on
    fibres. 
\end{cor}

To prove Theorem \ref{T:kostantslice}, we start by proving some simple facts
about regular elements in $L_0$, using Kostant's characterisation of regular
elements (Proposition 0.4 of \cite{Ko63b}) in $L$: if $x = y+z$ is the Jordan
decomposition of $x$, so that $y$ is semisimple, $z$ is nilpotent, and
$[y,z]=0$, then $x$ is regular if and only if $z$ is a principal nilpotent in
the reductive subalgebra $L^y$. Note that, by definition, a nilpotent element
of a reductive algebra $L$ is required to be in $[L,L]$, and if $z$ is a
nilpotent in $L$ commuting with a semisimple element $y$, then $z$ is also a
nilpotent in $L^y$.
\begin{lemma}\label{L:diagramregular} $L^{reg} \cap L_0 = L_0^{reg}$, the set
    of regular elements in $L_0$.
\end{lemma}
\begin{proof}
    Suppose $x$ in $L_0$ has Jordan decomposition $x = y+z$ in $L$. Then $x =
    y+z$ is also the Jordan decomposition in $L_0$, and in particular $y$ and
    $z$ are in $L_0$. Now by conjugating by an element of $G^{\sigma}$ we can
    assume that $y \in \mfh_0$, and in fact that $y$ is in the closed Weyl
    chamber corresponding to the Borel $L_0 \cap \mfb$, where $\mfb$ is the
    Borel in $L$ compatible with $\sigma$. Since the simple roots of $L$
    project to positive multiples of the simple roots of $L_0$, $y$ is also in
    the closed Weyl chamber of $L$ corresponding to $\mfb$.  Let $S$ be the set
    of simple roots $\alpha_J$ for $L_0$ that are zero on $y$, and similarly
    let $S'$ be the set of simple roots for $L$ that are zero on $y$. Since $y$
    is in the closed Weyl chamber, the stabilizer $L_0^y$ (respectively $L^y$)
    is the reductive Lie algebra $\mfh_0 \oplus \bigoplus_{\alpha \in \Z[S]}
    (L_0)_{\alpha}$ (respectively $\mfh \oplus \bigoplus_{\alpha \in \Z[S']}
    L_{\alpha}$). Now $x$ is regular in $L_0$ (respectively $L$) if and only if
    $z$ is a principal nilpotent in $L_0^y$ (respectively $L^y$). Every
    nilpotent element of $L_0^y$ is contained in a Borel, and all Borels are
    conjugate, so we can conjugate $z$ by an element of $(G^{\sigma})^y$ to get
    $z$ contained in the Borel $L_0^y \cap \mfb$ (since it does not have a 
    component in the centre, $z$ will in fact be in the nilpotent radical of
    $L_0^y \cap \mfb$). By Theorem 5.3 of \cite{Ko59}, $z$ is a principal
    nilpotent in $L_0^y$ if and only if the component of $z$ in
    $(L_0)_{\alpha}$ is non-zero for all $\alpha \in S$.  But by the
    construction of the simple Chevalley generators of $L_0$, this is
    equivalent to the component of $z$ in $L_\alpha$ being non-zero for all
    $\alpha \in S'$. So $z$ is a principal nilpotent in $L_0^y$ if and only if
    $z$ is a principal nilpotent in $L^y$, and hence $x$ is regular in $L_0$ if
    and only if $x$ is regular in $L$. 
\end{proof}
We also need the following standard technical lemma.
\begin{lemma}\label{L:technical}
    Let $\mfq$ be a $\Z_{\geq 0}$-graded Lie algebra, and let $\mfn$ denote the
    ideal $\bigoplus_{k>0} \mfq_k$. Suppose $y$ is an element of $\mfq_0$, and
    that $\mfr \subset \mfn$ is a graded subspace such that $\mfn = [\mfn,y]
    \oplus \mfr$. Then for every $x$ in the completion $\hat{\mfn}$ there is
    $g$ in the pro-nilpotent group $\exp (\hat{\mfn})$ such that $g (y+x) \in y
    + \hat{\mfr}$.
\end{lemma}
\begin{proof}
    Let $\{x_i\}$ be the sequence in $\hat{\mfn}$ with $x_0 = x$ and $x_{i+1} =
    \exp(-z_i) (y+x_i) - y$, where $z_i \in \hat{\mfn}$ is chosen so that $x_i
    = [z_i,y] + r_i$ for $r_i \in \hat{\mfr}$. Since $\exp(-z_i) (y+x_i) = y +
    r_i - [z_i,x_i]$, we can show by induction that $z_i$ and the component of
    $x_i$ in $[\mfn,y]$ are both zero below degree $i+1$. Hence the element $g
    = \cdots \exp(-z_2) \exp(-z_1) \exp(-z_0)$ is a well-defined element of
    $\exp(\hat{\mfn})$, and $g (y+x)$ is contained in $y+\hat{\mfr}$ as
desired.
\end{proof}
\begin{proof}[Proof of Theorem \ref{T:kostantslice}]
    If $m < +\infty$ then there is a homomorphism of Lie groups
    $J_m^{\tilde{\sigma}} G \arr J_{m-1}^{\tilde{\sigma}} G$, so the induced map
    $J_m^{\tilde{\sigma}} L \arr J_{m-1}^{\tilde{\sigma}} L$ on Lie algebras
    preserves semisimple (resp. nilpotent) elements. We say that an element of
    $J_{\infty}^{\tilde{\sigma}} L$ is pro-semisimple (resp. pro-nilpotent) if
    the image of the element is semisimple (resp. pro-nilpotent) in
    $J_m^{\tilde{\sigma}} L$ for every $m < +\infty$. Just as in the
    finite-dimensional case, every element of $J_\infty^{\tilde{\sigma}} L$ can
    be written uniquely as $y + z$ where $y$ is pro-semisimple, $z$ is
    pro-nilpotent, and $[y,z] = 0$.

    If $y \in J_m^{\tilde{\sigma}} L$ is semisimple (resp. pro-semisimple) then
    $y(0)$ is semisimple in $L$, and hence $L = L^{y(0)} \oplus [L,y(0)]$. It
    follows from Lemma \ref{L:technical} that there is $g \in J_m^{\tilde{\sigma}}
    G$ such that $g y = y(0) + z$, where $z \in J_m^{\tilde{\sigma}} L^{y(0)}$
    and $z(0) = 0$. Since $z$ is nilpotent (resp. pro-nilpotent), uniqueness of
    the Jordan decomposition implies that $z=0$.

    More generally, if $x$ is an arbitrary element of $J_m^{\tilde{\sigma}} L$
    then there is $g \in J_m^{\tilde{\sigma}}$ such that $g x = y + z$, where
    $y \in L_0$ is semisimple and $z \in J_m^{\tilde{\sigma}} L^y$ is
    nilpotent (resp.  pro-nilpotent). In particular $e = z(0)$ is nilpotent in
    $L_0^y$, so pick an $\mfsl_2$-triple $\{h,e,f\}$ in $L_0^y$ containing $e$.
    Then $L^y = L^{\{y,f\}} \oplus [L^y,e]$, so applying Lemma
    \ref{L:technical} again there is $g' \in J_m^{\tilde{\sigma}} G^y$ such
    that $g' z \in e + J_m^{\tilde{\sigma}} L^{\{y,f\}} = J_m^{\tilde{\sigma}}
    (e + L^{\{y,f\}})$ and $g'(0) z(0) = e$.

    Using this canonical form, we move on to the proof of the theorem
    statement. Pick a principal $\mfsl_2$-triple $\{h,e,f\}$ in $L_0$.
    By Lemma \ref{L:diagramregular} $\{h,e,f\}$ is also principal in
    $L$, so $\nu = e + L^f$ is a Kostant slice in $L$ invariant under
    $\sigma$, and $\nu^{\sigma} = e + L_0^f$ is a Kostant slice in $L_0$.
    It follows immediately that $J_m^{\tilde{\sigma}} \nu \arr
    J_m^{\tilde{\sigma}} Q$ is an isomorphism, and also that $Q^{\sigma} = L_0
    // G^{\sigma}$. Since $J_m^{\tilde{\sigma}} L^{reg} \arr
    // J_m^{\tilde{\sigma}} Q$ is $J_m^{\tilde{\sigma}} G$-invariant,
    each orbit in $J_m^{\tilde{\sigma}} L^{reg}$ can intersect
    $J_m^{\tilde{\sigma}} \nu$ at most once. So we just need to show that the
    multiplication map $J_m^{\tilde{\sigma}} G \times J_m^{\tilde{\sigma}} \nu
    \arr J_m^{\tilde{\sigma}} L^{reg}$ is surjective, or equivalently that
    every fibre of the map $J_m^{\tilde{\sigma}} L^{reg} \arr
    J_m^{\tilde{\sigma}} Q$ is a $J_m^{\tilde{\sigma}} G$-orbit.

    The projection $L^{reg} \arr Q$ is smooth and every fibre is a $G$-orbit,
    so the multiplication map $G \times \nu \arr L^{reg}$ is surjective and
    smooth. Hence by Lemma \ref{L:surjective} the multiplication map
    $J_m G \times J_m \nu \arr J_m L^{reg}$ is surjective. Suppose $x_1$
    and $x_2$ are two points of $J_m^{\tilde{\sigma}} L^{reg}$ with
    the same value in $J_m^{\tilde{\sigma}} Q$. Using the $m=0$ case and
    the canonical form above, we can assume that $x_1(0) = x_2(0) = y + e'$,
    where $y$ is semisimple in $L_0$ and $e'$ is a principal nilpotent in
    $L_0^y$, and that $x_1$ and $x_2$ are in $y + J_m^{\tilde{\sigma}} \nu'$,
    where $\nu'$ is the Kostant slice $e' + L^{\{y,f'\}}$ in $L^y$. Since $x_1$
    and $x_2$ have the same image in $J_m Q$, there is $g \in J_m G$ such that
    $g x_1 = x_2$. Multiplication by $g$ preserves Jordan decomposition, so $g
    \in (J_m G)^y$. The subgroup $G^y$ is a connected reductive subgroup of $G$
    by Lemma 5, page 353 of \cite{Ko63b}, and the exponential map is a
    bijection for nilpotent (resp.  pro-nilpotent) groups, so $(J_m G)^y =
    G^y \cdot \exp (z J_m L^y) = J_m G^y$, the connected subgroup of $J_m L$
    with Lie algebra $J_m L^y$.  Hence $x_1 - y$ and $x_2 - y$ are in the same
    regular $J_m G^y$-orbit of $J_m (L^y)^{reg}$. But $x_1 - y$ and $x_2 - y$
    belong to $J_m^{\tilde{\sigma}} \nu' \subset J_m \nu'$, which we have
    already observed intersects each $J_m G^y$-orbit exactly once, implying
    that $x_1 = x_2$ as desired.
\end{proof}    

Theorem \ref{T:kostantslice} implies that the map $J_m^{\tilde{\sigma}} L^{reg}
\arr J_m^{\tilde{\sigma}} Q$ is surjective for $0 \leq m \leq +\infty$, and
smooth for $0 \leq m < +\infty$. To prove Theorem \ref{T:kostantsmooth}, we
need to account for the parabolic component. Recall that $\mfg_0 = \mfp_0 \cap
\overline{\mfp_0}$. 
\begin{lemma}\label{L:parabolicsmooth}
    The projection $\mfp_0 \cap L_0^{reg} \arr R$ is a surjective smooth
    $\mcP_0$-orbit map. In addition, if $g \in G^{\sigma}$ fixes an element of
    $\mfp_0 \cap L_0^{reg}$ then $g$ belongs to $\mcP_0$.
\end{lemma}
\begin{proof}
    Let $\mfb_0$ be a Borel of $L_0$ contained in $\mfp_0$ and compatible
    with $\mfh_0$. Let $\mfu_0$ be the unipotent radical of
    $\mfp_0$, so that $\mfg_0 = \mfp_0 / \mfu_0$. Finally let $\Delta$ be the
    set of roots of $L_0$, and let $S \subset \Delta$ be the set of simple
    roots.  Similarly let $S_0 \subset S$ be the set of simple roots of
    $\mfg_0$ corresponding to the Borel $\mfb_0 \cap \mfg_0$, and let $\Delta_0
    = \Delta \cap \Z[S_0]$ be the set of roots of $\mfg_0$.

    Now suppose $y \in \mfp_0$ is semisimple in $L_0$, and let $y = y_0 + y_1$
    where $y_0 \in \mfg_0$ and $y_1 \in \mfu_0$. Then $\mfu_0 = [\mfu_0,y_0]
    \oplus \mfu_0^{y_0}$, so by Lemma \ref{L:technical} there is $p \in \mcP_0$ 
    such that $p y = y_0 + z$, where $z \in \mfu_0^{y_0}$. Since $p y$ is
    semisimple, we conclude that $z = 0$, and ultimately that $y$ is conjugate
    by $\mcP_0$ to an element of $\mfh_0$.
    
    Every element $x \in \mfp_0$ can be written as $x = y+z$ where $y,z \in
    \mfp_0$, $y$ is semisimple in $L_0$, $z$ is nilpotent in $L_0$, and $[y,z]
    = 0$. By the previous paragraph, it is possible to conjugate $x$ by an
    element of $\mcP_0$ so that $y \in \mfh_0$. We can then conjugate $x$ by an
    element of $\mcP_0^y$ so that $z$ belongs to $\mfb_0^y$. Assume $x$ is
    given with $y \in \mfh_0$ and $z \in \mfb_0^y$.  By a dimension argument,
    $\mfb_0^y$ is a Borel for the reductive Lie algebra $L_0^y$. The
    corresponding simple roots are the indecomposable elements $S_y$ of
    $\Delta^+_y = \{\alpha \in \Delta^+ : \alpha(y) = 0\}$. Similarly $\mfb_0^y
    \cap \mfg_0^y$ is a Borel for $\mfg_0^y$, and the simple roots are the
    elements of $S_y \cap \Delta_0$.  The element $x$ is regular in $L_0$ if
    and only if $z$ is a principal nilpotent in $L_0^y$, which is true
    if and only if the projection to $(L_0)_{\alpha}$ is non-zero for
    all $\alpha \in S_y$. If this latter condition holds then the image of
    $x$ in $\mfg_0 = \mfp_0 / \mfu_0$ is regular in $\mfg_0$. The projection
    $\mfp_0 \arr \mfg_0$ is $\mcP_0$-equivariant, so we conclude that the
    projection sends regular elements of $L_0$ to regular elements of $\mfg_0$.
    
    Conversely, if $x \in \mfg_0^{reg}$ then we can conjugate $x$ by an element
    of the subgroup of $\mfg_0$ to be of the form $y + z$ where $y \in \mfh_0$
    and $z \in \mfb_0^y$ is a principal nilpotent. This means that the
    projection of $z$ to $(L_0)_{\alpha}$ is non-zero for every $\alpha \in S_y
    \cap \Delta_0$. Let $z'$ be an element of $L_0$ such that the projection
    of $z'$ to $(L_0)_{\alpha}$ is non-zero if $\alpha \in S_y \setminus
    \Delta_0$, and is zero otherwise. Then $x + z'$ is a regular element of 
    $L_0$ which projects $x$. Using equivariance again, we conclude that the
    projection $\mfp_0 \arr \mfg_0$ induces a surjection $\mfp_0 \cap L_0^{reg}
    \arr \mfg_0^{reg}$. The map $\mfg_0^{reg} \arr R$ is a smooth surjection,   
    and $\mfp_0 \arr \mfg_0$ is smooth, so we conclude that $\mfp_0 \cap
    L_0^{reg} \arr R$ is a smooth surjection.
    
    Now suppose $x_1$ and $x_2$ in $\mfp_0 \cap L_0^{reg}$ map to the same
    element of $R$. As in the third paragraph, we can assume without loss of
    generality that $x_i = y_i + z_i$ with $y_i \in \mfh_0$ and $z_i$ a
    principal nilpotent element of $L_0^{y_i}$ contained in $\mfb_0^{y_i}$. In
    addition, the images of $x_1$ and $x_2$ in $\mfg_0$ are conjugate by an
    element of $\mcP_0$, so in particular we can assume that $y_1 = y_2$.
    Thus $z_1$ and $z_2$ are both principal nilpotents of $L_0^{y_1}$ contained
    in $\mfb_0^{y_1}$, and hence are conjugate by an element of the Borel
    subgroup of $\mfb_0^{y_1}$. We conclude that the projection $\mfp_0 
    \cap L_0^{reg} \arr R$ is a $\mcP_0$-orbit map.

    For the last part of the lemma, we again assume that $x \in \mfp_0 \cap
    L_0^{reg}$ is of the form $y + z$ with $y \in \mfh_0$ and $z \in \mfb_0^y$.
    If $x$ is regular then $L_0^x = (L_0^y)^z$ is contained in $\mfb_0 \subset
    \mfp_0$. By Proposition 14, page 362 of \cite{Ko63b}, $(G^{\sigma})^x$ is
    connected, and hence a subgroup of $\mcP_0$.
\end{proof}

\begin{proof}[Proof of Theorem \ref{T:kostantsmooth}]
    Suppose $x_1,x_2 \in \mfp_m \cap J_m^{\tilde{\sigma}} L^{reg}$ have the
    same image in $R \times_{Q^{\sigma}} J_m^{\tilde{\sigma}} Q$. By Theorem
    \ref{T:kostantslice} there is $g \in J_m^{\tilde{\sigma}} G$ such that $g
    x_1 = x_2$, while by Lemma \ref{L:parabolicsmooth} there is $p_0 \in
    \mcP_0$ such that $p_0 x_1(0) = x_2(0)$. Thus $p_0^{-1} g(0)$ fixes $x_1(0)
    \in \mfp_0 \cap L_0^{reg}$, so $g \in \mcP_m$ by Lemma \ref{L:parabolicsmooth},
    and $\mfp_m \cap J_m^{\tilde{\sigma}} L^{reg} \arr R \times_{Q^{\sigma}}
    J_m^{\tilde{\sigma}} Q$ is a $\mcP_m$-orbit map.

    To show surjectivity, observe that $\mfp_m \cap J_m^{\tilde{\sigma}}
    L^{reg} = (\mfp_0 \cap L_0^{reg}) \times_{L_0} J_m^{\tilde{\sigma}}
    L^{reg}$. A point of $R \times_{ Q^{\sigma}} J_m^{\tilde{\sigma}} Q$ is
    determined by a pair of points $x \in R$ and $y \in J_m^{\tilde{\sigma}} Q$
    which have the same image in $Q^{\sigma}$. Given a point specified in this
    manner, choose $x' \in \mfp_0 \cap L_0^{reg}$ mapping to $x$ and $y' \in
    J_m^{\tilde{\sigma}} L^{reg}$ mapping to $y$. Since $x$ and $y$ have the
    same image in $Q^{\sigma}$, there is $g \in G^{\sigma}$ such that
    $g y(0) = x$. Then $g y$ belongs to $\mfp_m$ and maps to the point
    $(x,y) \in R \times_{Q^{\sigma}} J_m^{\tilde{\sigma}} Q$.

    Since $\mfp_m \cap J_m^{\tilde{\sigma}} L^{reg} \arr R \times_{Q^{\sigma}}
    J_m^{\tilde{\sigma}} Q$ is a $\mcP_m$-orbit map, to show smoothness it is
    enough to show that the map $T \mfp_m \cap J_m^{\tilde{\sigma}} L^{reg}
    \arr T (R \times_{Q^{\sigma}} J_m^{\tilde{\sigma}} Q)$ is surjective.
    This follows from a similar argument to the last paragraph. As mentioned
    in the proof of Theorem \ref{T:kostantslice}, if $\nu_0$ is a Kostant
    slice in $L_0$ then $G^{\sigma} \times \nu_0 \arr L^{reg}_0$
    is smooth and surjective, so $T G^{\sigma} \times T \nu_0 \arr T L^{reg}_0$
    is also surjective, and hence if two elements of $T L^{reg}_0$ have the
    same image in $T Q^{\sigma}$ then they are conjugate by an element of $T
    G^{\sigma}$. Surjectivity of $\mfp_0 \cap L^{reg}_0 \arr R$ and
    $J_m^{\tilde{\sigma}} L^{reg} \arr J_m^{\tilde{\sigma}} Q$ follows from
    Lemma \ref{L:parabolicsmooth} and Theorem \ref{T:kostantslice}.
\end{proof}

\begin{proof}[Proof of Corollary \ref{C:bundlemap}]
    $R \times_{Q^{\sigma}} T_{const} J_m^{\tilde{\sigma}} Q$ is isomorphic to
    the pullback $R \times_{T Q^{\sigma}} T J_m^{\tilde{\sigma}}Q$, where the
    map $R \arr T Q^{\sigma}$ is the composition of the zero section $R \arr
    TR$ with the differential $TR \arr TQ^{\sigma}$. The restriction of the
    differential $T \mfp_m \arr T R$ to $T_{>0} \mfp_m$ factors through the
    zero section $R \arr TR$, so the image of $T_{>0} \mfp_m$ is contained in
    $R \times_{ TQ^{\sigma}} T J_m^{\tilde{\sigma}} Q$. To show that this bundle
    map is surjective on fibres, observe that, in the argument for smoothness
    in the proof of Theorem \ref{T:kostantsmooth}, if $x \in T R$ is a zero
    tangent vector, then we can pick $x' \in T \mfp_0 \cap T L^{reg}$ mapping
    to $x$ which is also a zero tangent vector, and hence the resulting point
    of $T \mfp_m$ will be contained in $T_{>0} \mfp_m$.
\end{proof}

\section{Calculation of parahoric cohomology}\label{S:mainproof}

In this section we finish the proofs of Theorems \ref{T:relativecoh} and
\ref{T:nilpotentcoh} and Proposition \ref{P:twistedexp}. We continue to
use the notation of Section \ref{S:orbit}.

\subsection{Proof of Proposition \ref{P:twistedexp}}\label{SS:twistedexp}

Pick a principal $\mfsl_2$-triple $\{h,e,f\}$ in $L_0$, and note that
$\{h,e,f\}$ is principal in $L$ by Lemma \ref{L:diagramregular}. We need to
show that the eigenvalues of $h/2$ on $L_a^e$ agree with the subset of the
exponents defined in Definition \ref{D:twistedexp}. Let $L = \bigoplus L^{(i)}$
denote the principal grading of $L$ induced by the eigenspace decomposition of
$h/2$.  Then $m \geq 0$ appears in the list of exponents of $L$ with
multiplicity $\dim \left(L^{(m)}\right)^e$. 

Let $\nu$ denote the Kostant slice $f + L^e$. As previously mentioned,
Kostant's theorem states that the restriction map $\C[Q] \arr \C[\nu]$ is an
isomorphism.  Actually, a stronger statement is true.  Identity $\C[\nu]$ with
polynomials on $L^e$ in the obvious way. Filter $\C[\nu]$ by setting
$\C[\nu]_m$ to be the subring of polynomials on $\bigoplus_{i=0}^m
\left(L^{(i)}\right)^e$.  Choose homogeneous generators for $\C[Q] = (S^*
L^*)^G$ and let $\C[Q]_m$ be the subring generated by generators of degree at
most $m+1$. Then, by Theorem 7, page 381 of \cite{Ko63b}, the restriction map
gives an isomorphism between $\C[Q]_m$ and $\C[\nu]_m$. Furthermore, if $I$ is
a generator of degree $m+1$ then the restriction of $I$ to $\nu$ takes the form
$f + I_0$ where $f$ is in the dual space of $\left(L^{(m)}\right)^e$ and $I_0
\in \C[\nu]_{m-1}$ does not have constant term.

The automorphism $\sigma$ acts on both $\C[\nu]$ and $\C[Q]$, preserving the
filtration in both cases, and the restriction map is $\sigma$-equivariant.  As
before, let $\mcM$ denote the ideal in $(S^* L^*)^G$ containing all elements of
degree greater than zero, so that $\mcM / \mcM^2$ is the space of generators.
By definition, the multiplicity of $m$ as an exponent is the multiplicity of
$q^{-a}$ as an eigenvalue of $\sigma$ acting on the degree $m+1$ subspace of
$\mcM / \mcM^2$. By the previous paragraph, this is equal to the multiplicity
of $q^{-a}$ as an eigenvalue of $\sigma$ acting on the dual space of
$\left(L^{(m)}\right)^e$, or equivalently the dimension of $q^a$ as an
eigenvalue of $\sigma$ acting on $\left(L^{(m)}\right)^e$ itself.

\subsection{Proof of Theorem \ref{T:relativecoh}}\label{SS:standardbasic}

Let $\Omega^{*}_{const} R \times_{Q^{\sigma}} J_m^{\tilde{\sigma}} Q$ denote
the sections of $\bigwedge^* R \times_{Q^{\sigma}} T_{const}^*
J_m^{\tilde{\sigma}} Q$, where $T_{const}^* J_m^{\tilde{\sigma}} Q$ is the dual
bundle to $T_{const} J_m^{\tilde{\sigma}} Q$. Similarly, let  $\Omega^{*}_{>0}
\mfp_m$ denote the sections of $\bigwedge^* T_{>0}^* \mfp_m$, where $T_{>0}^*
\mfp_m$ is the dual bundle to $T_{>0} \mfp_m$. As per Theorem
\ref{T:basicforms}, we want to calculate the algebra of
$\mcP_{\infty}$-invariant $\mcN_{\infty}$-basic elements of $\Omega^*_{>0}
\mfp_{\infty}$.

\begin{prop}\label{P:pullback}
    Pullback via the bundle map $T_{>0} \mfp_m \arr R \times_{Q^{\sigma}}
    T_{const} J_m^{\tilde{\sigma}}$ gives an isomorphism from the
    algebra $\Omega_{const}^* R \times_{Q^{\sigma}} J_m^{\tilde{\sigma}} Q$ to
    the algebra of $\mcP_m$-invariant $\mcN_m$-basic elements of $\Omega_{>0}^*
    \mfp_m$.
\end{prop}
\begin{proof}
    Every section of $\Omega^*_{>0} \mfp_{\infty}$ is a pullback from
    $\Omega^*_{>0} \mfp_m$ for some $m < +\infty$. By Corollary
    \ref{C:bundlemap}, the pullback map is injective, so it is enough to prove
    surjectivity when $m < +\infty$.

    Let $\mfp^{rs}_m$ denote the open subset $\mfp_m \cap J_m^{\tilde{\sigma}}
    L^{rs}$ of $\mfp_m$. We start by showing that the pullback map is an
    isomorphism from $\Omega_{const}^* R^r \times_{Q^{\sigma}}
    J_m^{\tilde{\sigma}} Q^r$ to the algebra of $\mcP_m$-invariant
    $\mcN_m$-basic elements of $\Omega_{>0}^* \mfp_m^{rs}$. By Theorem
    \ref{T:parahoricslice}, $\mfp_m^{rs}$ is isomorphic to $\mcP_m /
    J_m^{\tilde{\sigma}} H \times_{W(\mfg_0)} J_m^{\tilde{\sigma}} \mfh^r$. By
    Proposition \ref{P:principalbundle},
    \begin{equation*}
        T \left(\mcP_m / J_m^{\tilde{\sigma}}H  \times_{W(\mfg_0)}
            J_m^{\tilde{\sigma}} \mfh^r\right) \iso T \mcP_m / J_m^{\tilde{\sigma}} H
            \times_{W(\mfg_0)} T J_m^{\tilde{\sigma}} \mfh^r.  
    \end{equation*}
    $\mcP_m / \mcN_m$ is isomorphic to the connected reductive subgroup of
    $\mcP_m$ corresponding to the subalgebra $\mfg_0 \subset
    J_m^{\tilde{\sigma}} L$. We work for a moment in the analytic category.
    Suppose $\gamma_t$ is a curve in $\mfp_m^{rs}$ representing an element of
    $T_{>0} \mfp_m$, so that the image $\overline{\gamma_t}$ of $\gamma_t$ in
    $\mfg_0$ is constant.  There are curves $\alpha_t$ and $\beta_t$ in
    $\mcP_m$ and $J_m^{\tilde{\sigma}} \mfh^r$ respectively such that $\alpha_t
    \beta_t = \gamma_t$. Let $\overline{\alpha_t}$ denote the image of
    $\alpha_t \in \mcP_m / \mcN_m$. Then $\overline{\gamma_t} =
    \overline{\alpha_t} \beta_t(0)$ is a constant curve in $\mfg_0$, so
    $\overline{\alpha_0}^{-1} \overline{\alpha_t} \beta_t(0)$ is a constant
    curve in $\mfh_0^r$. This implies that $\overline{\alpha_0}^{-1}
    \overline{\alpha_t} \in w H^{\sigma}$ for some $w \in N(H)^{\sigma}$, from
    which we can conclude that $\overline{\alpha_0}^{-1} \overline{\alpha_t}
    \beta_t(0) = w \beta_t(0)$, so $\beta_t(0)$ is constant, and hence
    represents an element of $T_{const} J_m^{\tilde{\sigma}} \mfh^r$. Since
    $w^{-1} \overline{\alpha_0}^{-1} \overline{\alpha_t} \in H^{\sigma}$, the
    curves $\alpha_t$ and $\alpha_t \overline{\alpha_t}^{-1}
    \overline{\alpha_0} w$ are equal in $\mcP_m / J_m^{\tilde{\sigma}} H$. The
    latter curve projects to a constant curve in $\mcP_m / \mcN_m$, and since
    $\mcP_m \iso \mcP_m / \mcN_m \ltimes \mcN_m$, is tangent to a left
    $\mcN_m$-coset in $\mcP_m$. Since $\mcN_m$ is normal, every left
    $\mcN_m$-coset is a right $\mcN_m$-coset.  We conclude that over
    $\mfp_m^{rs}$, $T_{>0} \mfp_m$ is isomorphic to the subbundle $T_{\mcN_m}
    \mcP_m / J_m^{\tilde{\sigma}} H \times_{W(\mfg_0)} T_{const}
    J_m^{\tilde{\sigma}} \mfh^r$ of $T \left(\mcP_m / J_m^{\tilde{\sigma}} H
    \times_{W(\mfg_0)} J_m^{\tilde{\sigma}} \mfh^r\right)$, where $T_{\mcN_m}
    \mcP_m / J_m^{\tilde{\sigma}} H$ is the subbundle of tangents to
    $\mcN_m$-orbits. 

    Recall from the proof of Theorem \ref{T:parahoricslice} that
    $J_m^{\tilde{\sigma}} \mfh^r / W(\mfg_0)$ is isomorphic to
    $J_m^{\tilde{\sigma}} (\mfh^r / W(\mfm))$, so $T_{const}$ of the former
    space is well-defined.  By Proposition \ref{P:principalbundle}
    again, $(T J_m^{\tilde{\sigma}} \mfh^r) / W(\mfg_0) \iso T
    (J_m^{\tilde{\sigma}} \mfh^r / W(\mfg_0))$, so $T_{const} (J_m
    ^{\tilde{\sigma}} \mfh^r/ W(\mfg_0))$ is a subbundle of $(T
    J_m^{\tilde{\sigma}} \mfh^r) / W(\mfg_0)$. A tangent vector $v \in T
    J_m^{\tilde{\sigma}} \mfh^r$ represents an element of $T_{const} (J_m
    ^{\tilde{\sigma}}\mfh^r / W(\mfg_0))$ if and only if the projection of
    $v(0)$ to $\mfh_0^r / W(\mfg_0) \iso (\mfh^r / W(\mfm))^{\sigma}$ is a zero
    tangent vector, where $v(0)$ is the image of $v$ in $T \mfh_0^r$. Since
    $\mfh^r_0 \arr \mfh^r_0 / W(\mfg_0)$ is etale, this is true if and only if
    $v(0)$ is a zero tangent vector, so $T_{const} (J_m^{\tilde{\sigma}} \mfh^r
    / W(\mfg_0)) \iso \left(T_{const} J_m^{\tilde{\sigma}} \mfh^r\right) /
    W(\mfg_0)$. Similarly the isomorphism $J_m^{\tilde{\sigma}} \mfh^r /
    W(\mfg_0) \iso R^r \times_{Q^{\sigma}} J_m^{\tilde{\sigma}} Q^r$ sends
    $T_{const} J_m^{\tilde{\sigma}} \mfh^r / W(\mfg_0)$ to $R^r
    \times_{Q^{\sigma}} T_{const} J_m^{\tilde{\sigma}} \mfh^r$ (see the proof
    of Corollary \ref{C:bundlemap}). Applying Theorem \ref{T:parahoricslice},
    we want to show that the bundle map
    \begin{equation*} 
        T_{\mcN_m} \mcP_m / J_m^{\tilde{\sigma}} H \times_{W(\mfg_0)} T_{const}
        J_m^{\tilde{\sigma}} \mfh^r \arr T_{const}
        J_m^{\tilde{\sigma}} \mfh^r / W(\mfg_0)
    \end{equation*}
    induced by projection on the second factor gives an isomorphism from
    $\Omega^*_{const} J_m^{\tilde{\sigma}} \mfh^r / W(\mfg_0)$ to the ring of
    $\mcP$-invariant $\mcN$-basic sections of $\bigwedge^* T^*_{\mcN_m} \mcP_m /
    J_m^{\tilde{\sigma}} H \times_{W(\mfg_0)} T^*_{const}
    J_m^{\tilde{\sigma}} \mfh^r$. 

    By pulling back to $\mcP_m / J_m^{\tilde{\sigma}} H \times
    J_m^{\tilde{\sigma}} \mfh^r$, we can identify the ring of
    $\mcP_m$-invariant $\mcN_m$-basic sections of $\bigwedge^* T^*_{\mcN_m}
    \mcP_m / J_m^{\tilde{\sigma}} H \times_{W(\mfg_0)} T^*_{const}
    J_m^{\tilde{\sigma}} \mfh^r$ with a subring of the $\mcP_m$-invariant
    $\mcN_m$-basic sections of $\bigwedge^* T_{\mcN_m} \mcP_m /
    J_m^{\tilde{\sigma}} H \times T_{const}^* J_m^{\tilde{\sigma}} \mfh^r$.
    This latter ring is isomorphic to the ring $\Omega_{const}^*
    J_m^{\tilde{\sigma}} \mfh^r$ by pullback via projection on the second
    factor.  An element of $\Omega_{const}^* J_m^{\tilde{\sigma}} \mfh^r$
    descends to a section over $\mcP_m / J_m^{\tilde{\sigma}} H \times_{W(\mfg_0)}
    J_m^{\tilde{\sigma}} \mfh^r$ if and only if it is $W(\mfg_0)$-equivariant.
    The splitting $T J_m^{\tilde{\sigma}} \mfh^r =
    J_m^{\tilde{\sigma}} \mfh^r \times \mfh_0 \oplus T_{const}
    J_m^{\tilde{\sigma}} \mfh^r$ allows us to identify the $W(\mfg_0)$-module
    $\Omega_{const}^* J_m^{\tilde{\sigma}} \mfh^r$ with the subalgebra of
    differential forms which vanish on $J_m^{\tilde{\sigma}} \mfh^r \times
    \mfh_0$.\footnote{In contrast, there is no such identification for the
    $\mcP_m$-module $\Omega_{>0} \mfp_m$.} There is a similar splitting for
    $T J_m^{\tilde{\sigma}} \mfh^r / W(\mfg_0)$, and thus a similar
    identification for $\Omega_{const}^* J_m^{\tilde{\sigma}} \mfh^r /
    W(\mfg_0)$. The differential $T J_m^{\tilde{\sigma}} \mfh^r \arr T
    J_m^{\tilde{\sigma}} \mfh^r/ W(\mfg_0)$ preserves this splitting, so the
    pullback map $\Omega_{const}^* J_m^{\tilde{\sigma}} \mfh^r / W(\mfg_0) \arr
    \Omega_{const}^* J_m^{\tilde{\sigma}} \mfh^r$ agrees with the pullback map
    on differential forms. Thus pullback gives an isomorphism from
    $\Omega_{const}^* J_m^{\tilde{\sigma}} \mfh^r/ W(\mfg_0)$ to the
    $W(\mfg_0)$-equivariant elements of $\Omega_{const}^*
    J_m^{\tilde{\sigma}} \mfh^r$ (see, e.g., Theorem 1 of \cite{Br98}), and
    this implies that all $\mcP_m$-invariant $\mcN_m$-basic sections of
    $\bigwedge^* T^*_{\mcN_m} \mcP_m / J_m^{\tilde{\sigma}} H \times_{W(\mfg_0)}
    T^*_{const} J_m^{\tilde{\sigma}} \mfh^r$ come from pullback on the
    second factor.

    To finish the proof, let $\mfp_m^{reg} = \mfp \cap J_m^{\tilde{\sigma}}
    L^{reg}$, and let $\phi$ denote the map $\mfp_m^{reg} \arr R
    \times_{Q^{\sigma}} J_m^{\tilde{\sigma}} Q$. By Theorem
    \ref{T:kostantsmooth}, $\phi$ is smooth and surjective. Hence if $f$ is a
    regular function defined on an open dense subset of $R \times_{Q^{\sigma}}
    J_m^{\tilde{\sigma}} Q$ such that $\phi^* f$ extends to $\mfp_m^{reg}$,
    then $f$ has a unique extension to $R \times_{Q^{\sigma}}
    J_m^{\tilde{\sigma}} Q$. 

    Suppose $\omega \in \Omega_{>0}^* \mfp_m^{reg}$ is $\mcP_m$-invariant and
    $\mcN_m$-basic. Then there is $\alpha \in \Omega_{const}^* R^r
    \times_{Q^{\sigma}} J_m^{\tilde{\sigma}} Q$ such that $\phi^* \alpha =
    \omega$ over $\mfp_m^{rs}$. We can write $\alpha = \sum f_i \alpha_i$,
    where the $\alpha_i$'s are elements of $\Omega_{const}^* R
    \times_{Q^{\sigma}} J_m^{\tilde{\sigma}} Q$ which are linearly independent
    in fibres, and the $f_i$'s are functions on $R^r \times_{Q^{\sigma}}
    J_m^{\tilde{\sigma}} Q^r$. Since the bundle map is surjective on fibres,
    the pullbacks $\phi^* \alpha_i$ are linearly independent in fibres. Since
    $\phi^* \alpha = \sum_i \phi^* f_i \phi^* \alpha_i$ extends to
    $\mfp_m^{reg}$, the functions $\phi^* f_i$ must extend to
    $\mfp_m^{reg}$, and consequently $\alpha$ extends to $R
    \times_{Q^{\sigma}} J_m^{\tilde{\sigma}} Q$. The pullback $\phi^* \alpha$
    agrees with $\omega$ on an open dense subset, so every $\mcP_m$-invariant
    $\mcN_m$-basic element of $\Omega_{>0}^* \mfp_m^{reg}$ is the pullback of
    an element of $\Omega_{const}^* R \times_{Q^{\sigma}}
    J_m^{\tilde{\sigma}} Q$ as desired.
\end{proof}

\begin{proof}[Proof of Theorem \ref{T:relativecoh}]
    Let $I_i^{a}$ and $R_i$ be generators for $\C[Q]$ and $\C[R]$ as in the
    statement of Theorem \ref{T:relativecoh}. Choose coordinates $\{y_{ia}\}$ for
    $Q$ such that pullback of $y_{ia}$ via the projection $L \arr Q$ is 
    $I_i^{a}$. Similarly, choose coordinates $\{r_i\}$ for $R$ such that the
    pullback of $r_i$ via the projection $\mfp_0 \arr R$ is $R_i$. Note that
    the coordinates $\{y_{ia}\}$ with $a$ fixed correspond to the subspace
    $Q_a$ of $Q$ on which $\sigma$ acts as multiplication by $q^a$ (by
    previously established convention, this means that $\sigma y_{ia} = q^{-a}
    y_{ia}$). Consider the $r_i$'s as functions on $R \times_{Q^{\sigma}}
    J_{\infty}^{\tilde{\sigma}} Q$, and let $\tilde{y}_{ia}$ denote the induced
    map $J_{\infty} Q \arr Q$. Then the coordinate ring of $R \times_{Q^{\sigma}}
    J_m^{\tilde{\sigma}} Q$ is the free ring generated by the $r_i$'s and
    the functions $[z^{nk-a}] \tilde{y}_{i,-a}$ for $a=0,\ldots,k-1$
    and $n \geq 1$. Consequently the ring $\Omega_{const}^* R \times_{Q^{\sigma}}
    J_m^{\tilde{\sigma}} Q$ is the free super-commutative ring generated
    by the above generators for the coordinate ring, along with the restrictions
    of the differential forms $d [z^{nk-a}] \tilde{y}_{i,-a}$, $a=0,\ldots,k-1$
    and $n \geq 1$. Let $\hat{\mfp} = \mfp_{\infty}$ denote the completion of
    a standard parahoric.  Applying Proposition \ref{P:pullback} we conclude
    that $\Omega_{>0}^* \hat{\mfp}$ is the free super-commutative algebra 
    generated by the $R_i$'s, the functions $[z^{nk-a}] \tilde{I}_i^{-a}$
    for  $a=0,\ldots,k-1$ and $n \geq 1$, and the restrictions of the $1$-forms
    $d [z^{kn-a}] \tilde{I}_i^{-a}$ to $T_{>0} \hat{\mfp}$, again for
    $a=0,\ldots,k-1$ and $n \geq 1$. Theorem \ref{T:relativecoh} then follows
    from Theorem \ref{T:basicforms}.
\end{proof}

\subsection{Proof of Theorem \ref{T:nilpotentcoh}}\label{SS:nilpotentbasic}
The proof of Theorem \ref{T:relativecoh} can be simplified and used to prove
that pullback via the map $\mfp_m \arr R \times_{Q^{\sigma}}
J_m^{\tilde{\sigma}} Q$ gives an isomorphism between algebraic forms on $R
\times_{Q^{\sigma}} J_m^{\tilde{\sigma}} Q$ and $\mcP_m$-basic and invariant
forms on $\mfp_m$. When $\mfp_0 = L_0$, this can be proven without Theorem
\ref{T:parahoricslice}. Namely, if $\nu$ is a Kostant slice in $L$ then, as
previously mentioned, $G \times \nu \arr L^{reg}$ is surjective and smooth. By
Lemma \ref{L:twistedsmooth} and Theorem \ref{T:kostantslice}, the
multiplication map $J_m^{\tilde{\sigma}} G \times J_m^{\tilde{\sigma}} \nu \arr
J_m^{\tilde{\sigma}} L^{reg}$ is surjective for all $m$, and smooth for $m <
+\infty$. Since $J_m^{\tilde{\sigma}} \nu$ is isomorphic to
$J_m^{\tilde{\sigma}} Q$, identification of algebraic forms on
$J_m^{\tilde{\sigma}} Q$ with $J_m^{\tilde{\sigma}} G$-basic and invariant
algebraic forms on $J_m^{\tilde{\sigma}} L$ follows by pulling back to
$J_m^{\tilde{\sigma}} G \times J_m^{\tilde{\sigma}} \nu$. 

This idea can be adapted to determine the algebra of $\mcB$-basic and invariant
forms on $\hat{\mfn}$, where $\mfb$ is an Iwahori subalgebra, $\mcB$ is the
subgroup corresponding to the completion $\hat{\mfb}$, and $\hat{\mfn}$ is the
completion of the nilpotent subalgebra of $\mfb$. More specifically, let
$\mfb_m$ be the image of $\hat{\mfb}$ in $J_m^{\tilde{\sigma}} L$, let $\mcB_m$
be the corresponding connected subgroup of $J_m^{\tilde{\sigma}} G$, and let
$\mfn_m$ be the image of $\hat{\mfn}$ in $J_m^{\tilde{\sigma}} L$. If $X$ is a
variety with finite order automorphism $\sigma$, and $p \in X$, let
$J_{m,p}^{\tilde{\sigma}} X$ denote the subscheme $\{f \in J_m^{\tilde{\sigma}}
X : f(0) = p\}$ of jets with a fixed base point.

\begin{prop}\label{P:nilpotentbasic}
    There is a map $\mfn_m \arr J_{m,0}^{\tilde{\sigma}} Q$,
    and pullback via this map gives an isomorphism between the ring of
    algebraic forms on $J_{m,0}^{\tilde{\sigma}} Q$ and the
    ring of $\mcB_m$-basic and invariant algebraic forms on $\mfn_m$.
\end{prop}
\begin{proof}
    Once again it is sufficient to give the proof for $m < +\infty$. Let $e$
    be a principal nilpotent of $L_0$, contained in $\mfn_0$. Recall that $G^e$
    is a connected subgroup of $\mcB_0$. Let $\left(J_m^{\tilde{\sigma}}
    G\right)_e$ denote the connected subgroup $\{f \in J_m^{\tilde{\sigma}} G :
    f(0) \in G^e\}$ of $J_m^{\tilde{\sigma}} G$ with Lie algebra $\{f \in
    J_m^{\tilde{\sigma}} L : f(0) \in L^e_0\}$. Since $f \in \mfn_m$ belongs to
    $J_m^{\tilde{\sigma}} L^{reg}$ if and only if $f(0)$ is a principal
    nilpotent in $\mfn_0$, and all principal nilpotents in $\mfn_0$ are
    conjugate by an element of $\mcB_0$, it follows that the map
    \begin{equation*}
        \mcB_m \times_{\left(J_m^{\tilde{\sigma}} G\right)_e} J_{m,e}^{\tilde{\sigma}}
            L \arr \mfn_m \cap J_{m}^{\tilde{\sigma}} L^{reg} 
    \end{equation*}
    is an isomorphism. Consequently $\mcB_m$-basic and invariant forms on
    $\mfn_m \cap J_m^{\tilde{\sigma}} L^{reg}$ correspond to
    $\left(J_m^{\tilde{\sigma}} G\right)_e$-basic and invariant forms on 
    $J_{m,e}^{\tilde{\sigma}} L$. 
    
    The projection $L \arr Q$ sends $e$ to zero, so the restriction of
    $J_m^{\tilde{\sigma}} L \arr J_m^{\tilde{\sigma}} Q$ to
    $J_{m,e}^{\tilde{\sigma}} L$ factors through $J_{m,0}^{\tilde{\sigma}} Q$.
    Choose a principal $\mfsl_2$-triple $\{h,e,f\}$ in $L_0$ containing $e$,
    and let $\nu = e + L^f$. The isomorphism $J_m^{\tilde{\sigma}} \nu \arr
    J_m^{\tilde{\sigma}} Q$ identifies $J_{m,e}^{\tilde{\sigma}} \nu$ with
    $J_{m,0}^{\tilde{\sigma}} Q$, and every $\left(J_m^{\tilde{\sigma}}
    G\right)_e$-orbit on $J_{m,e}^{\tilde{\sigma}} L$ intersects
    $J_{m,e}^{\tilde{\sigma}} \nu$ in a unique point. Consequently the map
    $J_{m,e}^{\tilde{\sigma}} L \arr J_{m,0}^{\tilde{\sigma}} Q$ is a
    surjective smooth $\left(J_m^{\tilde{\sigma}} G\right)_e$-orbit map. It
    follows that the multiplication map $\left(J_m^{\tilde{\sigma}} G\right)_e
    \times J_{m,e}^{\tilde{\sigma}} \nu \arr J_{m,e}^{\tilde{\sigma}} L$ is
    smooth and surjective. We conclude that the pullback map from algebraic
    forms on $J_{m,e}^{\tilde{\sigma}} L$ to algebraic forms on
    $\left(J_m^{\tilde{\sigma}} G\right)_e \times J_{m,e}^{\tilde{\sigma}} \nu$
    is injective, and thus pullback via the map $J_{m,e}^{\tilde{\sigma}} L
    \arr J_{m,0}^{\tilde{\sigma}} Q$ gives an isomorphism between algebraic
    forms on $J_{m,0}^{\tilde{\sigma}} Q$ and $\left(J_m^{\tilde{\sigma}}
    G\right)_e$-basic and invariant forms on $J_{m,e}^{\tilde{\sigma}} L$.

    Thus every $\mcB_m$-basic and invariant form on $\mfn_m \cap
    J_m^{\tilde{\sigma}} L^{reg}$ is the pullback of a form from
    $J_{m,0}^{\tilde{\sigma}} Q$. Since $\mfn_m \cap J_m^{\tilde{\sigma}}
    L^{reg}$ is dense in $\mfn_m$, the proposition follows. 
\end{proof}
This proof does not extend to nilpotent subalgebras of other parahorics, as
$\mfu \cap L^{reg}[[z]]$ is non-empty only in the Borel case. As in the proof
of Theorem \ref{T:relativecoh}, Theorem \ref{T:basicforms2} and Proposition
\ref{P:nilpotentbasic} together imply Theorem \ref{T:nilpotentcoh}. 

\section{Spectral sequence argument for the truncated algebra}\label{S:truncatedcoh}

In this section we finish the proof of Theorem \ref{T:truncatedcoh} using a
spectral sequence argument. Recall from Section \ref{S:laplacian} the
definition of the operators $d_R(S)$ and $d_L(T)$ on $\bigwedge^*\hat{\mfp}
\otimes S^* \hat{\mfp}$. The operator $d_R(S)$ is a generalized interior
product, while $d_L(T)$ is a generalized exterior derivative. Hence we
have the following version of Cartan's identity:
\begin{lemma}\label{L:anticommutator}
    $d_R(S) d_L(T) + d_L(T) d_R(S) = (ST)^{Sym} + (TS)^{\wedge}$, where
    $(ST)^{Sym}$ is the extension of $ST$ to the symmetric factor as a
    derivation, and $(TS)^{\wedge}$ is the extension of $TS$ to the exterior
    factor as a derivation.
\end{lemma}
Let $P : \hat{\mfp}^* \arr \hat{\mfp}^*$ be the dual of multiplication by
$z^N$ on $\hat{\mfp}$. Define $Q : \hat{\mfp}^* \arr \hat{\mfp}^*$ by
$(Qf)(x) = f\left(\frac{\overline{x}}{z^N}\right)$, where $\overline{x}$
is the projection to $z^N \hat{\mfp}$ using the splitting $\hat{\mfp} = (z^N
\hat{\mfp}) \oplus (\hat{\mfp} / z^N \hat{\mfp})$ suggested by the root
grading. Note that $P Q = \Id$, while $QP$ is projection to $(z^N
\hat{\mfp})^*$ using the corresponding splitting of $\hat{\mfp}^*$. Thus
$(d_R(P) d_L(Q) + d_L(Q) d_R(P)) \omega = (n+q) \omega$ if $\omega
\in \bigwedge^* (\hat{\mfp} / z^N \hat{\mfp})^* \otimes \bigwedge^n (z^N
\hat{\mfp})^* \otimes S^q \hat{\mfp}^*$. Then $d_R(P)^2 = 0$, and we can use
Cartan's identity to show that
\begin{equation*}
    \xymatrix{ 0 \ar[r] & \bigwedge^* (\hat{\mfp} / z^N \hat{\mfp})^* \ar[r] &
        \bigwedge^* \hat{\mfp}^* \ar[r]^-{d_R(P)} & \bigwedge^{*-1} \hat{\mfp}^*
        \otimes S^1 \hat{\mfp}^* \ar[r]^-{d_R(P)} & \ldots }
\end{equation*}
is exact. Further, $d_R(P)$ commutes with the Lie algebra cohomology operator
$\deltabar$ with coefficients in $S^* \hat{\mfp}^*$. Since $d_R(P)$ is
$\mfp$-equivariant and preserves the subset of cochains which vanish on
$\mfg_0$, we can restrict to the relative cochain complex to get an
exact sequence
\begin{equation*}
    \xymatrix{ 0 \ar[r] & \left(\bigwedge^* (\hat{\mfu} / z^N \hat{\mfp})^*
        \right)^{\mfg_0}
        \ar[r] & K^{*,0} \ar[r]^{d_R(P)} & K^{*,1} \ar[r]^{d_R(P)} & \ldots },
\end{equation*}
where $K^{*,*}$ is the bigraded algebra $\left( \bigwedge^* \hat{\mfu}^*
\otimes S^* \hat{\mfp}^*\right)^{\mfg_0}$ graded by tensor (ie. combined
exterior and symmetric) degree and symmetric degree, regarded as a bicomplex
with differentials  $\deltabar$ (the Lie algebra cohomology differential for
$\hat{\mfu}$ with coefficients in $S^* \hat{\mfp}^*$) and $d_R(P)$. Note that
both $\deltabar$ and $d_R(P)$ are derivations of the algebra structure. 
\begin{lemma}\label{L:truncresolution}
    Give $K^{*,*}$ a $z$-grading by taking the usual $z$-degree for the
    exterior factor, and $z$-degree $+\ N$ on $\hat{\mfp}^*$ for the symmetric
    factor. This $z$-grading descends to $H^*(\total K^{*,*})$, and there is an
    isomorphism $H^*(\mfp/z^N \mfp, \mfg_0) \arr H^*(\total K^{*,*})$ which
    preserves $z$-degrees. 
\end{lemma}
\begin{proof}
    We have just shown that there is a chain map from the Koszul complex for
    $(\mfp/z^N\mfp, \mfg_0)$ to $\total K^{*,*}$. Consider the spectral sequence
    induced by the column-wise filtration of $K^{*,*}$, ie. the descending
    filtration where the $p$th level contains all elements of $K^{a,b}$ with $a
    \geq p$. The $E_1$-term of this spectral sequence is 
    \begin{equation*}
        E_1^{p,q} = \begin{cases} 
                        \left(\bigwedge^p (\hat{\mfu} / z^N \hat{\mfp})^*\right)^{\mfg_0} & q=0 \\ 
                        0 & q>0 
                    \end{cases},
    \end{equation*}
    with differential the restriction of $\deltabar$. Hence
    \begin{equation*}
        E_2^{p,q} = \begin{cases} H^p(\mfp / z^N \mfp, \mfg_0) & q=0 \\
                                    0 & q>0 \end{cases}.
    \end{equation*}
    It follows from naturality of the spectral sequence that the induced map
    $H^*(\mfp/z^N \mfp, \mfg_0) \arr H^*(\total K^{*,*})$ is an isomorphism. The
    $z$-degrees on $K^{*,*}$ are preserved by $\deltabar$ and $d_R(P)$, so the
    $z$-grading descends to homology and likewise is preserved by the isomorphism.
\end{proof}

To calculate $H^*(\total K^{*,*})$, consider the spectral sequence of the
bicomplex $K^{*,*}$ induced by the row-wise filtration, ie. the descending
filtration where the $p$th level contains all elements of $K^{a,b}$ with $b
\geq p$. This spectral sequence has $E_{1}^{p,q} =
H^{q-p}_{cts}(\hat{\mfp},\mfg_0; S^p \hat{\mfp})$ with differential $d_R(P)$
(note that the order of the degrees is swapped compared to $K^{*,*}$, so $p$ is
symmetric degree and $q$ is tensor degree). Thus $E_1^{*,*}$ is a freely
generated differential super-commutative algebra, with generating cocycles
explicitly described in Theorem \ref{T:relativecoh} as follows. If
$r_1,\ldots,r_{l_0}$ is a list of exponents for $\mfg_0$ then there is a
generator in $E_1^{r_i+1,r_i+1}$ represented by a cocycle $R_i$.  If
$m_1^{(-a)},\ldots,m_{l_{-a}}^{(-a)}$ is a list of twisted exponents then there
is a generator in $E_1^{m_i^{(-a)}+1, m_i^{(-a)}+1}$ for every $n \geq 1$,
represented by a cocycle $f^{nk-a}_i = [z^{nk-a}] \tilde{I}_i^{(-a)}$, and a
generator in $E_1^{m_i^{(-a)},m_i^{(-a)}+1}$ for every $n \geq 1$, represented
by a cocycle $\omega^{nk-a}_i = J_{\Delta} d [z^{nk-a}] \tilde{I}_i^{(-a)}$.
Since $d_R(P)$ is a derivation, we just need to determine its action on these
generators. By degree considerations, $d_R(P)$ kills the generators $R_i$ and
$f_i^{nk-a}$.  Note that $f_i^{0} = [z^0] \tilde{I}_i^{(0)}$ lies in
$E_1^{*,*}$, as it belongs to the algebra $\C[R]$ generated by the $R_i$'s
(apply Theorem \ref{T:kostantslice} with $m=0$).  If the reductive algebra $L$
splits as a direct sum $L = \mfz \oplus \bigoplus L^{(i)}$, where $\mfz$ is the
centre and the $L^{(i)}$'s are $\sigma$-invariant simple components, then we
can assume that the generators $I_i^{(-a)}$ of $(S^* L^*)^L$ used to construct
the cochains $f^{nk-a}_i$ belong either to $S^* \mfz^*$ or to $\left(S^*
\left(L^{(i)}\right)^* \right)^{L}$ for some $i$. With this assumption we have:
\begin{lemma}\label{L:differential}
    The differential $d_R(P)$ on $K^{*,*}$ sends $\omega^{nk-a}_i$ to a
    non-zero scalar multiple of $f_i^{nk-a-N}$ if $nk - a \geq N$, and to zero
    otherwise. 
\end{lemma}
\begin{proof}
    The generator $\omega^{nk-a}_i$ can be rewritten as $d_L(J) f_{i}^{nk-a}$. 
    Both $d_R(P)$ and $d_L(J)$ preserve the subalgebras $(S^* (L^{(i)})^*)^L$
    and $S^* \mfz^*$, so we can assume that $L$ is either simple or abelian.
    Since $d_R(P) f_i^{nk-a} = 0$, we can use Lemma \ref{L:anticommutator} to
    get $d_R(P) \omega_i^{nk-a} = (P J)^{Sym} f_{i}^{nk-a}$. As an element of
    the dual of $S^* \hat{\mfp}$, $(PJ)^{Sym} f_{i}^{nk-a}$ is defined by
    \begin{equation*}
        x_1 \circ \cdots \circ x_{m^{(-a)}_i+1} \mapsto \sum_j [z^{nk-a}]
            I_i^{(-a)}(\cdots \circ J z^N x_j \circ \cdots).
    \end{equation*}
    Suppose $L$ is abelian. Then, as noted after the statement of Theorem
    \ref{T:basicforms}, we can assume that $J$ is the identity, so
    $(PJ)^{Sym} f_i^{nk-a} = f_i^{nk-a-N}$ as required. 

    This leaves the case that $L$ is simple, in which case $J$ is defined as
    the derivation of $\hat{\mfp}$ acting on weight spaces $\mfg_{\alpha}$ as
    multiplication by $\langle \rho,\alpha \rangle$, where $\rho$ is the weight
    of the associated Kac-Moody satisfying $\rho(\alpha_i^{\vee}) = 1$ if
    $d_i>0$ in the grading of type $d$ determining $\mfp$, and
    $\rho(\alpha_i^{\vee}) = 0$ otherwise. Following the Kac convention in
    \cite{Ka83}, the Kac-Moody associated to $\mfg$ is $\tilde{\mfg} = \mfg
    \oplus \C c \oplus \C d$, where $c$ is central and $d$ acts by $z
    \frac{d}{dz}$. The roots of $\tilde{\mfg}$ belong to the dual of the Cartan
    $\mfh_0 \oplus \C c \oplus \C d$, and are defined similarly to the roots of
    $\mfg$, with $d^*$ replacing $\delta$.  If $\alpha_0 = d^* -
    \psi,\alpha_1,\ldots,\alpha_l$ is a list of simple roots for
    $\tilde{\mfg}$, then the associated coroots are $\alpha_0^{\vee} = c^* -
    \psi_0,\alpha_1^{\vee},\ldots,\alpha_{l}^{\vee}$, where $\psi_0$ is either
    $\psi^{\vee}$ in the untwisted case, or the element of $\mfh_0$ such that
    $\langle x,\psi_0\rangle = \psi(x)$ in the twisted case.  The standard
    non-degenerate invariant form $\langle,\rangle$ for $\tilde{\mfg}$
    satisfies $\langle \mfh_0,c\rangle = \langle \mfh_0,d\rangle = \langle
    c,c\rangle = \langle d,d\rangle = 0$ and $\langle c,d\rangle \neq 0$.

    Write $\rho = \rho_0 + A c^*$ for some $\rho_0$ in $\mfh_0^*$.  If $x_j \in
    \mfg_{\alpha}$ in the above equation then $z^N x_j \in \mfg_{\alpha+Nd^*}$,
    so $J z^N x_j = z^N (N \langle \rho, d^* \rangle + J) x_j$. Since $\langle
    \rho, d^*\rangle = A \langle c^*, d^* \rangle$, we have
    \begin{equation*}
        d_R(P) \omega_i^{nk-a} = \begin{cases} \left(N A (m_i^{(-a)}+1)\langle c^*,d^*\rangle 
                                    + J^{Sym}\right) f^{nk-a-N}_i & nk>N \\
                                    0 & nk \leq N
                                \end{cases}.
    \end{equation*}
    Take a basis $\{x_{\alpha,i}\}$ for $\mfg_{\alpha}$, and let $x_{\alpha}^i$ be
    the dual basis. Then $J^{Sym} x_{\alpha}^i = \langle \rho,\alpha \rangle
    x^{i}_{\alpha} = (\langle \rho_0,\alpha\rangle + A \langle c^*,\alpha
    \rangle) x^i_{\alpha}$.  There is $\tilde{\rho}_0 \in \mfh_0$ such that
    $\alpha(\tilde{\rho}_0) = \langle \rho_0,\alpha \rangle$ for all roots
    $\alpha$, so $\ad^t(\tilde{\rho}_0) x_{\alpha}^i = - \langle \rho_0,\alpha
    \rangle x_{\alpha}^i$. Hence on the subring of $\mfh_0$-invariant functions
    of $S^*\hat{\mfp}^*$, $J^{Sym}$ agrees with the derivation which sends
    $x_{\alpha}^i$ to $A \langle c^*,\alpha \rangle x_{\alpha}^i$. The product
    $\langle c^*, \alpha \rangle$ is equal to $\langle c^*, d^* \rangle$ times
    the $z$-degree of $x_{\alpha}^i$. We conclude that $J^{Sym} f^{nk-a-N}_i =
    A \langle c^*, d^* \rangle (nk-a-N) f^{nk-a-N}_i$, and consequently that
    \begin{equation*}
        d_R(P) \omega_i^{nk-a} = A \langle c^*, d^* \rangle \left(
            N (m_i^{(-a)}+1) + nk-a-N \right) f^{nk-a-N}_i
    \end{equation*}
    if $N \leq nk-a$. Since $nk-a-N \geq 0$, the coefficient is non-zero as
    required. 
\end{proof}
We now have a situation parallel to when we defined $d_R(P)$. Let $V_0$ be the
free vector space spanned by basis elements $v_i^{nk+a}$ for $n \geq 0$,
$a=0,\ldots,k-1$, and $i=1,\ldots,l_a$. For any integer $m$, let $V_m$
be the subspace of $V_0$ spanned by the $v_i^{nk+a}$'s with $nk+a \geq m$.
Identify $\bigwedge^* V_1 \otimes S^* V_0$ with a subalgebra of $E_1^{*,*}$ by
sending $v_i^{nk+a}$ to $f_i^{nk+a}$ in the symmetric term and to
$\omega_i^{nk+a}$ in the exterior term. Let $P'$ be the linear map $V_0 \arr
V_0$ sending $v_i^{nk+a}$ to $v_i^{nk+a-N}$ if $nk+a \geq N$, and to zero
otherwise.  Let $Q'$ be the operator $V_0 \mapsto V_1$ sending $v_i^{nk+a}
\mapsto v_i^{nk+a+N}$. Then $P' Q' = \Id$, while $Q' P'$ is projection to
$V_N$. So $d_R(P') d_L(Q') + d_L(Q') d_R(P')$ acts as multiplication by the
combined symmetric degree and exterior $V_N$-degree. By Lemma
\ref{L:differential}, the differential on $E_1^{*,*}$ restricts to $d_R(P')$ on
$\bigwedge^* V_1 \otimes S^* V_0$, and hence the homology of the differential
on this subspace is the subalgebra $\Lambda_1$ of the $E_2$-term generated
by $\omega^{nk+a}_i$'s with $0 < nk + a < N$. To get the whole $E_2$-term, recall:
\begin{lemma}\label{L:free}
    $(S^*\mfg_0^*)^{\mfg_0}$ is a free $(S^* L_0^*)^{L_0}$-module. 
\end{lemma}
\begin{proof}
    Let $S = (S^* \mfg_0^*)^{\mfg_0}$ and $A = (S^* L_0^*)^{L_0}$.  Restriction
    to the Cartan $\mfh_0$ gives isomorphisms $S \iso (S^* \mfh_0^*)^{W(\mfg_0)}$
    and $A \iso (S^* \mfh_0^*)^{W(L_0)}$, so $A$ is a subalgebra of $S$. By 
    the Chevalley-Shephard-Todd theorem \cite{Ch55} \cite{Ko63b}, there is a
    subset $H_0 \subset S^* \mfh_0^*$ such that $S^* \mfh_0^* \iso A \otimes
    H_0$ as a $W(L_0)$-module, where the isomorphism is given by
    multiplication, and $H_0$ is isomorphic to the regular representation.
    Hence $S \iso A \otimes H$ where $H = H_0^{W(\mfg_0)}$. 
\end{proof}
The algebra $\C[Q^{\sigma}]$ generated by the $f_i^0$'s is a subalgebra of
$\bigwedge^* V_1 \otimes S^* V_0$. Note that $d_R(P')$ is
$\C[Q^{\sigma}]$-linear, since it kills $\C[Q^{\sigma}]$ and is a derivation.
The $E_1$-term is isomorphic to the base extension $\C[R]
\otimes_{\C[Q^{\sigma}]} \bigwedge^* V_1 \otimes S^* V$, with differential
given by the base extension $\Id \otimes d_R(P')$ of $d_R(P')$. Freeness
implies that the $E_2$-term is $\C[R] \otimes_{\C[Q^{\sigma}]} \Lambda_1$.
Since the action of $\C[Q^{\sigma}]$ on $\Lambda_1$ sends everything of
symmetric degree $>0$ to zero, the $E_2$-term is isomorphic to
$\Coinv(L_0,\mfg_0) \otimes \Lambda_1$. 
\begin{lemma}\label{L:spectral}
    The spectral sequence collapses at the $E_2$-term. Consequently the
    graded algebra of $H^*(\total K^{*,*})$ with respect to the row-wise
    filtration is isomorphic to $\Coinv(L_0,\mfg_0) \otimes \Lambda_1$.
\end{lemma}
\begin{proof}
    $\Lambda_1$ is a free algebra with a generator $\omega^{nk-a}_i \in
    E_2^{m_i^{(-a)},m_i^{(-a)}+1}$ for every twisted exponent $m_i^{(-a)}$ of
    $L$ and $n$ such that $0 < nk-a < N$. The subring $\Coinv(L_0,\mfg_0)$ lies
    in bidegrees $(a,a)$, so the entire $E_2$-term is contained in bidegrees
    $(a,b)$ with $a \leq b$. Suppose more generally that the $E_r$-term is
    contained in bidegrees $(a,b)$ with $a \leq b$, and is generated in
    bidegrees $(a,a+1)$ and $(a,a)$. The $E_2$-term differential has bidegree
    $(2,-1)$, and thus annihilates $\Coinv(L_0,\mfg_0)$ and the generators
    $\omega_i^{nk-a}$. The same argument works for higher $E_r$-terms
    as well.
\end{proof}
    
Now we just need to determine the ring structure of $H^*(\total K^{*,*})$.  The
row-wise filtration of $K^{*,*}$ is the descending filtration where $F^P
K^{*,*} = \bigoplus_{r \geq p} K^{*,r}$. Likewise $F^p H^*(\total K^{*,*})$ is
the subspace of homology classes which have a representative cocycle in $F^p
K^{*,*}$. If $k \in K^{q,p}$ is such that $\deltabar k = d_R(P) k = 0$, then
$k$ determines a homology class $[k]$ in $F^p H^{p+q}(\total K^{*,*})$.
Referring to the construction of the spectral sequence of a filtered
differential module (see, e.g., pages 34-37 in \cite{MC01}), we also see that
$k$ determines a persistent element of the spectral sequence, ie. $k$ 
represents an element in each $E^{p,q}_r$ (once again, note that the degrees
are swapped between $K^{*,*}$ and $E^{*,*}$) that is killed by the $r$th
differential, and the homology class of this element corresponds to the element
represented by $k$ in $E^{p,q}_{r+1}$. The projection $F^p H^{p+q}(\total
K^{*,*}) \arr E^{p,q}_{\infty}$ sends $[k]$ to the element represented by $k$
in $E_{\infty}$. Finally, when $E_1^{p,q}$ is identified with
$H^q(K^{*,p},\deltabar)$ the element of $E_1$ represented by $k$ is simply the
homology class represented by $k$ in $H^q(K^{*,p},\deltabar)$, and consequently
the same is true of the identification of $E_2$ with
$H^*(H^*(K^{*,*},\deltabar), d_R(P))$. Note that this would not necessarily be
true if $k$ was not homogeneous.

We know that the $E_2$-term is generated by classes represented by elements
$R_i$, $i=1,\ldots,l_0$ and $\omega^{nk+a}_i$, $i=1,\ldots,l_a$ and $0 < nk+a <
N$ in $K^{*,*}$. Let $\Lambda$ denote the subalgebra of $K^{*,*}$
generated by the elements $\omega_i^{nk+a}$, $i=1,\ldots,l_a$, $0 < nk+a<N$. 
By Theorem \ref{T:relativecoh} and Lemma \ref{L:differential}, $\C[R] \otimes
\Lambda \subset K^{*,*}$ is annihilated by both $\deltabar$ and
$d_R(P)$. Hence there is a homomorphism $\C[R] \otimes \Lambda \arr
H^*(\total K^{*,*})$. Since $\deltabar \omega_i^{N} = 0$, Lemma
\ref{L:differential} implies that the image of $f_i^0$ in $H^*(\total K^{*,*})$
is zero, so the homomorphism $\C[R] \otimes \Lambda \arr H^*(\total
K^{*,*})$ descends to a map $\Coinv(L_0,\mfg_0) \otimes \Lambda \arr H^*(\total
K^{*,*})$. By Lemma \ref{L:spectral} and the argument of the last paragraph,
this map is a bijection. We record this calculation in the following proposition.
\begin{prop}\label{P:truncatedrel}
    Let $\Coinv(L_0,\mfg_0)$ denote the algebra of Lemma \ref{L:free}, graded
    by symmetric degree. Give $\Coinv(L_0,\mfg_0)$ a cohomological grading by
    doubling the symmetric grading, and a $z$-grading by multiplying the
    symmetric grading by $N$. Then $H^*(\mfp / z^N \mfp, \mfg_0)$ is
    isomorphic to $\Coinv(L_0,\mfg_0) \otimes \Lambda$, where $\Lambda$ is
    the free algebra generated in cohomological degree $2m_i^{(a)}+1$,
    $z$-degree $N m_i^{(a)} + nk+a$, for $a=0,\ldots,k-1$, $i=1,\ldots,l_a$,
    and $n$ such that $0 < nk+a < N$.
\end{prop}

Consider the untwisted case where $n=1$. In this case, $\mfp / z \mfp$ is the
semi-direct product $\mfp_0 \ltimes L_0 / \mfp_0$, where $L_0 / \mfp_0$ has Lie
bracket equal to zero. Then Proposition \ref{P:truncatedrel} implies that
$H^*(\mfp_0 \ltimes L_0 / \mfp_0,\mfg_0)$ is isomorphic to
$\Coinv(L_0,\mfg_0)$. The following Lemma implies that Proposition
\ref{P:truncatedrel} actually gives a direct Lie algebra proof of Borel's
theorem that $\Coinv(L_0,\mfg_0)$ is isomorphic to $H^*(L_0,\mfg_0)$. Note that
the $z$-grading on $H^*(\mfp_0 \ltimes L_0 / \mfp_0,\mfg_0)$ is half the
cohomological grading, and thus corresponds to the holomorphic grading on
$H^*(L_0,\mfg_0)$. 
\begin{lemma}\label{L:borelpic}
    Let $\mfp_0 \ltimes L_0 / \mfp_0$ be the semi-direct product where $L_0 /
    \mfp_0$ has Lie bracket equal to zero. The cohomology ring $H^*(\mfp_0
    \ltimes L_0 / \mfp_0,\mfg_0)$ is isomorphic to $H^*(L_0,\mfg_0)$.
\end{lemma}
\begin{proof}
    Let $X$ be the generalized flag variety $G^{\sigma} / \mcP_0$, where
    $\mcP_0$ is the parabolic subgroup of $G^{\sigma}$ corresponding to
    $\mfp_0$. The complex-valued de Rham complex of $X$ can be realized as the
    relative Koszul complex 
    \begin{equation*}
        C^*\left(L_0,\mfg_0; C^{\infty}(K;\C)\right) = 
            \left(\bigwedge^* (L_0 / \mfg_0)^* \otimes C^{\infty}(K;\C)
                \right)^{\mfg_0},
    \end{equation*}
    where $K$ is a compact form of $X$. The de Rham differential $d$ translates
    to the Lie algebra cohomology boundary operator for $(L_0,\mfg_0)$.  Let
    $\mfu_0$ be the nilpotent radical of $\mfp_0$. The holomorphic structure on
    $X$ gives the de Rham complex a bigrading, which can be written in terms of
    $C^*(L_0,\mfg_0; C^{\infty}(K;\C)$ as
    \begin{equation*}
        C^{p,q}(C^{\infty}(K;\C)) = \left(\bigwedge^p \overline{\mfu_0}^*
            \otimes \bigwedge^q \mfu_0^* \otimes C^{\infty}(K;\C) 
                \right)^{\mfg_0},
    \end{equation*}
    where $p$ is the holomorphic degree, and $q$ is the anti-holomorphic
    degree. The differential $d = \partial + \deltabar$, where $\partial$
    and $\deltabar$ are the holomorphic and anti-holomorphic differentials
    respectively. On $C^{*,*}$, $\partial$ is the Lie algebra cohomology
    differential of $\overline{\mfu_0}$ with coefficients in $\bigwedge^*
    \mfu_0^* \otimes C^{\infty}(K;\C)$, where $\mfu_0$ is the
    $\overline{\mfu_0}$-module $L_0 / \overline{\mfp_0}$. Similarly $\deltabar$
    is the Lie algebra cohomology differential of $\mfu_0$ with coefficients in
    $\bigwedge^* \overline{\mfu_0}^* \otimes C^{\infty}(K;\C)$. The Kahler
    identities then imply that the Laplacian $d d^* + d^* d$ of $d$ with
    respect to a Kahler metric is equal to twice the Laplacian $\deltabar
    \deltabar^* + \deltabar^* \deltabar$. In particular the two differentials
    give the same cohomology. 

    A theorem of Chevalley-Eilenberg implies that the de Rham complex is
    quasi-isomorphic to the subcomplex $C(L_0,\mfg_0; \C)$ of equivariant forms
    \cite{CE48}. Since $K$ acts by holomorphic maps on $X$, the same is true of
    the de Rham complex with the anti-holomorphic differential. Hence the
    Kahler identities imply that the cohomology of $C^*(L_0,\mfg_0; \C)$ is
    the same with respect to either $d$ or $\deltabar$. Finally $\left(C(L_0,\mfg_0;
    \C), \deltabar\right)$ can be identified with the Koszul complex for the
    Lie algebra cohomology of $H^*(\mfp_0 \ltimes L_0 / \mfp_0, \mfg_0)$.
\end{proof}

To finish the section, we observe that if $\mfp$ is an Iwahori, then a similar
spectral sequence calculation can be made with $S^* \hat{\mfp}^*$ replaced by
$S^* \hat{\mfu}$.  In this case the spectral sequence will converge to
$H^*(\hat{\mfp} / z^N \mfn,\mfh_0)$, while the $E_1$-term of the
spectral sequence is the free super-commutative algebra $H^*_{cts}(\hat{\mfp},\mfh_0;
S^* \hat{\mfu})$ generated by elements $f_i^{nk-a} \in
E^{m_i^{(-a)}+1,m_i^{(-a)}+1}_1$ and $\omega_i^{nk-a} \in
E^{m_i^{(-a)},m_i^{(-a)}+1}$ for every $n \geq 1$, $a=0,\ldots,k-1$, and
$i=1,\ldots,l_a$. The differential on the $E_1$-term sends $\omega_i^{nk-a}$ to
$f_i^{nk-a-N}$ if $nk-a > N$, and to zero otherwise.  Thus the $E_2$-term will
be the free algebra generated by the $\omega^{nk-a}_i$'s with $0 < nk-a \leq
N$. Since the algebra is free, the isomorphism on graded algebras lifts to
give:
\begin{prop}\label{P:nilpotentrel}
    Let $\mfb$ be an Iwahori subalgebra of $\mfg$, and let $\mfn$ be
    the nilpotent subalgebra. Then $H^*(\mfb/z^N \mfn, \mfh_0)$
    is a free algebra generated in cohomological degree $2m_i^{(a)}+1$,
    $z$-degree $N m_i^{(a)}+nk+a$, for $a=0,\ldots,k-1$, $i=1,\ldots,l_a$, and
    $n$ such that $0 < nk+a \leq N$.
\end{prop}

\bibliographystyle{plain}

\end{document}